\documentclass[12pt,reqno]{amsart}
\usepackage[headings]{fullpage}

\usepackage{amssymb,amsmath,amscd}

\usepackage{bbm}
\usepackage{tikz}
\usepackage{url}
\usepackage[bookmarks=true,%
colorlinks=true,% 
linkcolor=blue,%
citecolor=blue,%
filecolor=blue,%
menucolor=blue,%
urlcolor=blue,%
breaklinks=true]{hyperref}

\usepackage{verbatim}
\usepackage{mathtools}
\usepackage[normalem]{ulem}

\graphicspath{{figures/}}

\usepackage{mathtools} %(1) correct various bugs/defeciencies in amsmath until these are fixed by the AMS and (2) provide useful tools for mathematical typesetting

\usepackage{multirow}

\usepackage{amsfonts}

\usepackage{float}

\usepackage{enumerate}

\usepackage{array}

\usepackage{abbrCKKY2021}

\newtheorem{thm}{Theorem}[section]

\theoremstyle{definition}
\newtheorem{prop}[thm]{Proposition}
\newtheorem{lem}[thm]{Lemma}

\newtheorem{cor}[thm]{Corollary}
\newtheorem{defn}[thm]{Definition}
\newtheorem{exam}[thm]{Example}
\newtheorem{rmk}[thm]{Remark}

\newtheorem{conj}{Conjecture}

\newtheorem{nota}{Notation}

\usepackage{setspace}
\usepackage[foot]{amsaddr}  % make author index in author address. https://tex.stackexchange.com/questions/226786/amsart-option-titlepage-puts-the-address-on-the-last-page 
\makeatletter
\renewcommand{\email}[2][]{%
	\ifx\emails\@empty\relax\else{\g@addto@macro\emails{,\space}}\fi%
	\@ifnotempty{#1}{\g@addto@macro\emails{\textrm{(#1)}\space}}%
	\g@addto@macro\emails{#2}%
}
\makeatother
\newcommand\restr[2]{{% we make the whole thing an ordinary symbol
  \left.\kern-\nulldelimiterspace % automatically resize the bar with \right
  #1 % the function
  \vphantom{\big|} % pretend it's a little taller at normal size
  \right|_{#2} % this is the delimiter
  }}

  \allowdisplaybreaks

\begin{document}

\title{Parabolic representations and generalized Riley polynomials}

\date{\today}

\author[Y.Cho]{Yunhi Cho}
\address[Y.Cho]{Department of Mathematics, University of Seoul, Seoul, 02504, Korea (Korea Institute for Advance Study, Seoul, 02455, Korea)}
\email[Y.Cho]{yhcho@uos.ac.kr}

\author[H.Kim]{Hyuk Kim}
\address[H.Kim]{Department of Mathematical Sciences, Seoul National University, Seoul, 08826, Korea}
\email[H.Kim]{hyukkim@snu.ac.kr}

\author[S.Kim]{Seonhwa Kim} %\fnref{s_kim}}
\address[S.Kim]{Department of Mathematics, Natural Science Research Institute, University of Seoul, 02504, Seoul, Korea}
\email[S.Kim]{seonhwa17kim@gmail.com}

\author[S.Yoon]{Seokbeom Yoon}
\address[S.Yoon]{Departament de Matem\`atiques, Universitat Aut\`onoma de Barcelona, 08193 Cerdanyola del Vall\`es, Spain}
\email[S.Yoon]{sbyoon15@gmail.com}
	%\thanks{
		%}
	
	\date{\today}
%\dedicatory{Dedicated to Walter Neumann and Don Zagier, with admiration.}

\begin{abstract}
	We generalize R. Riley's study about parabolic representations of two bridge knot groups to the general knots in $S^3$.
% that is a $\sl$- or $\psl$-representation with the condition the meridian image is parabolic. 
%			we generalize Riley's polynomial not restricted to two-bridge cases.   
We utilize the parabolic quandle method for general knot diagrams and adopt symplectic quandle for  better investigation, which gives such representations and their complex volumes explicitly.
For any knot diagram with a specified crossing $\cc$, we define a generalized Riley polynomial $\Riley_\cc(y) \in \Qbb[y]$ 
whose roots correspond to the conjugacy classes of parabolic representations of the knot group.
The \emph{sign-type} of parabolic quandle is newly introduced and we obtain a formula for the obstruction class to lift to a boundary unipotent $\sl$-representation. Moreover, we  define another polynomial $g_\cc(u)\in\Qbb[u]$, called $u$-polynomial, and prove that $\Riley_\cc(\u^2)=\pm g_\cc(\u)g_\cc(-\u)$. 
Based on this result, we introduce and investigate \emph{Riley field} and \emph{$u$-field} which are closely related to the invariant trace field. 
%	, and give an application about homomorphisms between knot groups. 
This method eventually leads to the complete classification of parabolic representations of knot groups along with their complex volumes and cusp shapes up to 12 crossings. 
% we present thorough computations of parabolic representations by our method for the examples of knots, $4_1$, $7_4$, $8_{18}$, $9_{29}$, and  the link $5_1^2$, along with the complex volume and cusp shape for each representation. 						

%	  Moreover, we see a strong relation between $\Riley_\cc(y)$ and meridian-preserving homomorphisms (including epimorphisms) between knot goups.
\end{abstract}

%		\begin{keyword}
	%			\MSC[2010] 57M25 
	%		\end{keyword}

\maketitle

{\footnotesize
	\tableofcontents
}

%%%%%%%%%%%%%%%%%%%%%%%%%%%%%%%%%%%%%%%%%%%%%%%%%%%%%%%%%%%%%%%%%%%%%%%%%%%%
%%%%%%%%%%%%%%%%%%%%%%%%%%%%%%%%%%%%%%%%%%%%%%%%%%%%%%%%%%%%%%%%%%%%%%%%%%%%

	\section{Introduction}\label{sec:introduction}

\subsection{Background on parabolic representations}

% \footnote{\red{SY:I will leave some comments in red. Please erase them after you check. \blue{SK: I leave some SY's comments as footnotes in order to consider it later}}}

Parabolic representations of two-bridge links were intensively studied by R. Riley in the 1970s \cite{riley_parabolic_1972, riley_parabolic_1975} and his far-reaching computations  produced  lots of discrete faithful representations of knot groups, which contributed to the earlier development of 3-dimensional geometry and topology by W. Thurston \cite{thurston_geometry_1977,riley_personal_2013}. 
A \emph{parabolic representation} of a knot group is an $\sl$ or $\psl$ representation which sends a meridian to a parabolic element.
The parabolicity condition is algebraically the same  as the completeness condition of Thurston gluing equations for a complete (pseudo-) hyperbolic structure.

% They\red{Parabolic representations?} have several better properties than general non-parabolic representations,  just as  a geometric representation is more important  than incomplete deformed representations.

% For example, we can define complex volume well for a parabolic case.
% \footnote{\red{For example, the complex volume is well-defined for parabolic representations}} 
One advantage of boundary-parabolic representations is that 
 the complex volume is well-defined
\cite{neumann_extended_2004,zickert_volume_2009}. The Chern-Simons invariant for 
$\sl$-representations has a subtle ambiguity to be defined and requires additional data about a logarithmic branch in the boundary \cite{marche_geometric_2012} although the hyperbolic volume is defined regardless of the boundary condition \cite{francaviglia_hyperbolic_2004}.

 Parabolic representations naturally produce  elements in the extended Bloch group  coming from the image of the fundamental class of knot exterior, and hence the volumes of parabolic representations would be related to Neumann's conjecture \cite{neumann_realizing_2011}. 
Arithmetic invariants like the invariant trace field and quaternion algebras \cite{maclachlan_arithmetic_2003}  can be also defined for any parabolic representation just as for the holonomy of a complete hyperbolic manifold of finite volume.
These arithmetic invariants have been well-studied for the past several decades, in particular, with the aid of  computer software like SnapPy. 
Our work is an extension of these previous studies of the holonomy representation  to general parabolic representations and 
provides a diagrammatic framework that allows us various examinations and computations directly related to a knot diagram.
%and a lot of examples from the parabolic non-geometric representation (many of them are not even Galois conjugates to the geometric representation)

On the other hand,
parabolic representations naturally appear in the study of the volume conjecture \cite{murakami_volume_2018}.
An optimistic limit of a colored Jones polynomial (resp.  Kashaev invariant) is a critical value of a potential function  $W(w_1,\dots,w_{N+2})$ (resp. $V(z_1,\dots,z_{2N})$ ) made up of dilogarithm function $\Li$, where $w_i$ (resp. $z_i$) $\in \Cbb\setminus\{0\}$ and $N$ is the number of crossings in the diagram. Here, parabolic representations are given  as the critical points of $W(w)$ (resp.  $V(z)$) since they are  solutions to Thurston gluing equation of the corresponding ideal triangulation called \emph{octahedral decomposition}. 

 A remarkable point is that  the diagrammatic formulas for the explicit computation of parabolic representations as well as   their hyperbolic invariants can be derived in this setting.
For example, 
in \cite{cho_optimistic_2013,cho_optimistic_2014, kim_octahedral_2018}, complex volume, cusp shape, and representation matrices are obtained in terms of $z$- or $w$-variables suggesting the direct connection to quantum invariants. 
Cho and Murakami \cite{cho_optimistic_2013, cho_optimistic_2016, cho_quandle_2018}  express these $w$-variables or $z$-variables 
% for the hyperbolicity equations from $W(w_1,\dots,w_{N+2})$ or $V(z_1,\dots,z_{2N})$ 
by using  parabolic quandle for vector coloring introduced by A. Inoue and Y.Kabaya \cite{inoue_quandle_2013}.
Moreover, Ptolemy coordinates \cite{zickert_volume_2009} of octahedral decomposition
for parabolic representations can also be described by these $z$- or $w$-variables \cite{kim_octahedral_2019}. 
 An interesting recent result of C. McPhail-Snyder \cite{mcphail-snyder_hyperbolic_2022} is that  Kashaev-Reshetikhin's quantum holonomy invariant can be written in terms of this $z$- and $w$- variables. 
The $u$-variable introduced in this paper is equivalent to such $z$- or $w$- variables which are in general much harder to compute than $u$-vaiable.

One of the themes in this paper is to study further about these previous works and the classification of the parabolic representation and their invariants by focusing on the parabolic quandle using $u$-variable. Note that this setting  suggests a practical framework to obtain a complete list of the conjugacy classes of
parabolic representations. The parition funtions in Chern-Simons theory usually involve a complete set of gauge equivalence classes of flat connections  by stationary phase approximation of the path integral, and finding all the parabolic representations could be essential.

\subsection{Parabolic quandles and Riley polynomials}
The system of parabolic quandle equations in \cite{inoue_quandle_2013,jo_symplectic_2020} is essentially equivalent to that of the Wirtinger relation, but is simpler and has a different flavor.
% \footnote{\red{I am not sure what this sentence means}} %in \cite{yetter_quandles_2003,navas_symplectic_2007} 
 It may be more tractable and interesting as we can reduce it to a symplectic quandle. 
From this consideration, Jo and Kim have employed such an idea for the case of two bridge links \cite{jo_symplectic_2020, jo_continuant_2022}. 
They not only reproduce  Riley's results but also  find some new algebraic properties of the Riley polynomial. In particular,  in \cite[Theorem 5.1]{jo_symplectic_2020} they showed that  the Riley polynomial admits a \emph{square-splitting} property, that is, $\Riley(\u^2)=\pm g(\u)g(-\u)$ for some $g \in \Zbb[\u]$.
One of the motivations of this paper is to answer the question of whether the square-splitting property  holds for general knots not just restricted to two-bridge knots. 

To do this, we first introduce a generalization of  Riley polynomial to non-two-bridge cases. We observe that the roots of Riley polynomial $\Riley(\y)$ are nothing but a trace of $\rho(\meri_1 \meri_2)$ where $\rho$ is a parabolic representation and $\meri_1$ and $\meri_2$ are the two meridians at the  top-most( or bottom-most) crossing in Conway normal form. Riley proved that $\Riley(\y)$ has distinct roots and hence there is a natural bijection between the set of roots of $\Riley(\y)$ and the conjugacy classes of parabolic representations. 

A two bridge diagram (of Conway form) has a kind of canonical choice of crossing i.e., the bottom-most or the top-most crossing. However,   there is no canonical choice of a crossing in general link diagram  and hence we just consider a knot diagram $D_c$ with a specified \emph{base crossing} $c$.
Then we  define \emph{generalized Riley polynomial} $\Riley_c(\y)$ for $D_c$ to be a polynomial whose roots are the trace of $\rho(\meri_i \meri_j)$ where $\meri_i$ and $\meri_j$ are two arcs at the based crossing $c$. 

%By a simple calculation (see Lemma \ref{lem:lambdaalpha}), we can  observe 
An important observation is that each root of $\Riley_c(\y)$ is the square of the determinant of two parabolic quandle coloring vectors at the crossing $c$. It naturally follows that $\Riley_\cc(\u^2)$ can be written  in the form of $g_\cc(\u)g_\cc(-\u)$. Thus, the natural question arises whether  
$\Riley_\cc(\u^2)$ splits in $\Qbb[\u]$.  It turns out that $g(\u)$ has rational coefficients only for knots  and we can check this does not hold for link examples (see Section \ref{sec:whiteheadlink}). 

\subsection{Sign-types and obstruction classes}
To clarify when $g_\cc(\u)$ has rational coefficients, we study a more precise relation between the system of parabolic quandle equations and parabolic representations. 
The natural correspondence between parabolic quandles and  Wirtinger generators has a certain sign-ambiguity, but 
we can obtain an unambiguous solution to the quandle equation by introducing a \emph{sign-type} as follows. 

Let  the sign of each coloring vector be fixed. Then we define a \emph{sign-type} $\ee$ of the parabolic quandle that is  a sign-assignment at each crossing, i.e.,
\[
\ee:\{\text{crossings in $\D$} \} \longrightarrow \{\pm1\},
\]
which is  the sign of the quandle equation at each crossing in Definition \ref{def:signedrel}.
%Definition \ref{def:signedrel}.  
Then, we obtain  that 
 the product of $\ee(i)$   over all crossings  doesn't depend on the choice of the sign-type (Proposition \ref{prop:totalsignconstant}).
  It would be natural to ask what the unchanged value is. We prove that it coincides with \emph{obstruction class} $\ob(\rho)$ of the associated representation $\rho$ in Theorem \ref{thm:obs} as follows,
$$\ob(\rho)=\prod\limits_{c\in\text{crossings}} \ee(c) = \text{the eigen value of a longitude}. $$  
In other words, it determines whether $\rho$ (or a lifting of $\psl$-representation $\rho$) is a boundary-unipotent $\sl$-representation or not.

Remark that  $9_{29}$  admits a parabolic representation of positive obstruction. See the computation in Section \ref{sec:9_29}. As looking at the datasets of \cite{curveproject} and  \cite{diagramsite}, we can observe that  almost all parabolic representations  have negative obstruction. It turned out that the geometric  representation must have negative obstruction \cite{calegari_real_2006}.
%  However, we don't know the reason why most of the parabolic representations have  negative obstruction. It would be interesting to see the condition of the positive obstruction as in $9_{29}$.  

\subsection{$\u$-polynomials and their properties}
Furthermore,
if we fix any sign-type $\ee$ for a knot diagram then we have an exact bijection between the conjugacy classes of parabolic representations and the solution of parabolic quandle equation with the obstruction $\ob$. Based on this setting,  we can prove that $g_c(\u)$ has  $\Qbb$-coefficients by  elementary Galois theory and we call $g_\cc(u)$  \emph{$\u$-polynomial}. 
As an immediate corollary, we prove  the square-splitting property of generalized Riley polynomial, i.e. $\Riley_c(\u^2)=\pm g_c(\u)g_c(-\u)$ with $g(\u)\in \Qbb[\u]$.  

We remark that the $\u$-polynomial has more information than Riley polynomial especially when the generalized Riley polynomial is reducible. 
Let us suppose $\Riley(\y)=\Riley_1(\y)\Riley_2(\y)$, then we have the corresponding $g(\u)=g_1(\u)g_2(\u)$ with $\Riley_i(\u^2)=\pm g_i(\u) g_i(-\u)$ for $i=1,2$. But the choice of individual $g_i(\u)$ or $\pm g_i(-\u)$ for each $i$ cannot be chosen  arbitrarily since  all factors of $g_i(\u)$ are affected   simultaneously as we change the sign-type. Hence we have a subtle ``sign-coupling'' phenomenon interrelating irreducible factors of $g(\u)$. At the moment the meaning of this coupling is not clear but this is quite remarkable since the interrelation between all irreducible components of a  character variety is rare, while 
many known properties
about a character variety are essentially related to a single irreducible component.  
It would be interesting if the sign-coupling phenomenon can be revealed in a further study.
Remark that  another example interrelating all irreducible components of character variety is the vanishing property of adjoint Reidemeister torsion. See  \cite{gang_adjoint_2019},\cite{tran_adjoint_2021}.

\subsection{Riley field, $u$-field, and trace field}
Instead of choosing a base crossing $\cc$,
a Riley polynomial $\Riley_\ij(y)$ (resp. $u$-polynomial
$g_\ij(u)$) can be naturally considered for a pair of arcs $\arc_i$ and $\arc_j$ as well.
Similarly, each root of $\Riley_\ij(y)$ (resp. $g_\ij(u)$) corresponds to a parabolic representation $\rho$ and hence we define \emph{Riley field}  (resp. \emph{$u$-field} ) of $\rho$  as the number field $\yfield$ (resp. $\ufield$) generated by the corresponding $y_\ij$'s (resp. $u_\ij$'s) over all pairs of arcs.

At first glance, the Riley field seems to be a more natural object than the $u$-field since the generators are the traces of all words of length two. An obvious fact is that $\yfield$ is a subfield of  $\trfield$, but it seems to be difficult to see whether it is a knot invariant or how it is related to  the trace field $\trfield$. 
However, in contrast to the Riley field, it can be naturally proved that  $\ufield$ is a knot invariant and $\ufield \supset \trfield$ (see Theorem \ref{thm:invufield} and Theorem \ref{thm:inclusionFields}.). 
Therefore, the $u$-field seems to be a better object to study than the Riley field. 

Based on a lot of computer experiments, we conjecture that the $u$-field is always the same as the Riley field (see Conjecture \ref{conj:ufield=trfield} and Conjecture \ref{conj:yfield=ufield}). If the conjecture is true, then the Riley fields, $u$-field, and the trace field, all three are identical and in particular, the trace field is generated by only the words of length two, which is a special property of parabolic representation of a knot group. It could be quite surprising because  the trace relation usually requires the word of length three.

\subsection{Homomorphisms between knot groups} 
An interesting application is that
Riley polynomial gives a necessary condition to admit a homomorphism $\phi:\Gk \to \Gk' $ such that $\phi(\alpha_\cc)=\alpha_{\cc'}$  for a crossing loop $\alpha_\cc$ of $K$ and $\alpha_{\cc'}$ of $K'$. 
If there is such a homomorphism then each irreducible factor of $R_\cc'(y)$ should divide $R_{\cc}(y)$. In particular, $R_{\cc'}(y) \mid R_{\cc}(y)$ if $\phi$ is an epimorphism and $R_{\cc'}(y) = R_{\cc}(y)$ if $\phi$ is an automorphism.
This result may help study  epimorphism or automorphism between knot groups, in particular, for meridian preserving cases.

\subsection{Computations}

We give  several expository computations as examples by using parabolic quandle for the cases of $4_1$, $7_4$, $8_{18}$, $9_{29}$, and $5_1^2$(Whitehead link). 
We provide a detailed procedure to obtain  a complete list of parabolic representations and compute their obstruction classes, Riley polynomials, and $\u$-polynomials. The complex volumes and cusp shapes are also computed by a new diagrammatic method developed recently, which is  inspired by the volume conjecture (\cite{cho_optimistic_2014}, \cite{cho_optimistic_2016-1}, \cite{kim_octahedral_2018}, \cite{kim_octahedral_2019}). 
All these computations are cross-checked by SnapPy \cite{SnapPy} and CURVE project \cite{curveproject}  data. 

Remark that Chern-Simons invariants in the CURVE project are only computed modulo $\frac{\pi^2}{6}$ but our computation is modulo $\pi^2$ as the Chern-Simons invariant  is originally defined by modulo $\pi^2$. 
Moreover, we stress that the list of parabolic representations by our method is complete, i.e. there is no missing irreducible representation. For example, we found that $9_{29}$ knot has a representation of positive obstruction which is missing in the CURVE project. 

The $7_4$ is the first example which has two irreducible components of $\Xpar$,\emph{ the character variety of parabolic representations}, and moreover, we give an example of a generalized Riley polynomial whose constant term is not $\pm1$. 
The $9_{29}$ is known as the first example having a non-integral trace \cite{reid_infinitely_2021}. Note that non-integral Riley polynomial for this example also provides a non-integral trace.

Remark that, in the case of $4_1$, $7_4$,  and $5_1^2$, we can compute $\Xpar$ essentially by hands without the aid of a computer since it uses only long-division of polynomials of small degree. But the other examples require the aid of an elimination algorithm like  Gr\"oebner basis to solve the system.
 
The computation method  works for link cases as well. But the square-splitting property doesn't hold for link cases and we can see this for the case of Whitehead link $5_1^2$. We point out  several differences  for link cases with the $5_1^2$ example in Section \ref{sec:whiteheadlink}.

\subsection{$8_{18}$ knot}
In the case of  $8_{18}$, the structure of $\Xpar$ is fairly complicated and the whole set of the parabolic representations   is not known yet so far. Although the CURVE project \cite{curveproject} provides a list of parabolic representations of $8_{18}$, it contains only 20 representations. By our method, we assert that \emph{there are exactly 26 parabolic representations up to conjugation}. Furthermore, we verify that the Riley polynomials at each crossing respect the rich symmetry of the $8_{18}$ diagram by  the necessary condition of Riley polynomials  in Theorem \ref{thm:homomorfactor} to admit a knot group automorphism.  Using  these together, we obtain that  $8_{18}$ has essentially different  12 parabolic representations and they are completely classified by the complex volumes. See Theorem \ref{thm:818essential}.

%
%The $8_5$ example is the first example which can not be computed only by polynomial division. So we use  
%
% Note that the other system of equation $8_{18}$ or $9_{29}$
%

\subsection{Classification and final remark}
%\subsection{Final remarks and questions}

We would like to highlight that the parabolic quandle system in this paper provides a very efficient system of equations to obtain all parabolic representations. For example,
%  the calculations for $8_{18}$ and $9_{29}$ usually not so easy with the other known methods.  
if one tries to solve the system of equation directly from the Wirtinger relation of $8_{18}$, it would be quite difficult even with  the help of a computer. Note that  Thurston gluing equation used by SnapPEA is much more efficient to solve by computer but it never guarantees that all parabolic representations are obtained.  
 The parabolic quandle method in this paper, however, can compute the complete list of parabolic representations for $8_{18}$ easily by a low-performance laptop. 
% In a subsequent paper, we will provide a complete list of parabolic representations for the whole Rolfsen knot table and beyond. The ongoing computational data can be accessed in \cite{diagramsite}
%
Based on this method we actually  provide a complete list of parabolic representations for the whole Rolfsen knot table and beyond, which can be  accessed in \cite{diagramsite}.    One may tabulate parabolic representations up to 9 crossings just following the computations in Section \ref{sec:examplecomputations}, but additional computational difficulties begin to emerge for the knots beyond 10 crossings. The  technical details in performing the practical calculations will be addressed in a seperate subsequent paper. The computational data has currently been completed up to 12 crossings and is continuously being updated. 

Moreover, the parabolic quandle method studied in this paper is not only effective but also conceptually interesting,  as can be seen in the advent of sign-type or $u$-field. 
% It would definitely deserve a further study. 
%
Moreover, these approaches using quandle vector can be generalized to non-parabolic representations and still produce an effective system to solve and it helps to understand the whole representation variety, character variety, and $A$-polynomial. We will see these studies in the  forthcoming papers.

Finally, we ask  a basic and fundamental question on the existence of parabolic representations of knot group. 
For irreducible $SU(2)$-representations (and hence  irreducible $\sl$-representations as well),
its existence for any nontrivial knot is known due to Kronheimer and Mrowka \cite{kronheimer_dehn_2004}. On the other hand,
 the existence of (non-abelian)  parabolic representation doesn't seem to be discussed yet in the literature  despite its fundamental importance.
We formulate the following conjecture in a little stronger statement. 
\begin{conj}
A knot $K$ is nontrivial  if and only if 	$\Xpar(K)$ is nonempty. Furthermore it always has  a zero-dimensional component. 	
\end{conj}
%We expect that this conjecture is true and studying parabolic quandle system could give a practical unknot detecting algorithm.
%
This question looks quite plausible and is true for all the known computations. 

\section*{Acknowledgments}
The authors express their gratitude to Professor Dong-il Lee
at Seoul Women's University for consulting about Gro\"ebner basis technique,  Dr. Dong Uk Lee for the valuable  comments about algebraic and arithmetic geometry, and Phillip  Choi for computer calculation support. The work of YC was supported by the 2021 sabbatical year research grant of the University of Seoul.
The work of HK and SK was supported by the National Research Foundation of Korea (NRF) grant funded by the Korean government (MSIT) respectively (NRF-2018R1A2B6005691) and (No. 2019R1C1C1003383).

\section{Preliminaries}
Let us review some basic definitions and facts to set the notation and terminology in this paper. We will follow common usage in most textbooks  if the precise definition is not given here.
%
% Throughout this paper, mathematical alphabets with a specified font are used  carefully: $\K$, $\D$, $\n$, $m_i$, $\arc_i$, $\aa_i$, $\cc_i$, $\rho$, $\P$, $\Q$, $\QQ$, $\Riley$, $\Rpar$, $\Xpar$ and etc. will always stand for the definitions in this section.

Let $\K$ be an oriented knot in $S^3$. A \emph{knot group} $\Gk$ is the fundamental group $\pi_1(S^3\setminus \K)$ of the knot complement $S^3\setminus \K$. 
Sometimes we deal with   links in $S^3$, but we always consider only knots unless specified.
Let  $\D$ be an oriented knot diagram of $\K$ with $\n$ crossings. 
Let $\cc_1 \dots \cc_\n$ denote the crossings in $\D$ and  $\arc_1,\dots,\arc_\n$ denote the arcs, where an \emph{arc} (i.e., \emph{over-arc}) is a connected curve of $\D$ between two under-passing crossings, that is through only over-passing crossings.

\subsection{Parabolic representation}

\emph{Wirtinger presentation} of $\Gk$ is
a well-known diagrammatic group presentation for a given $D$  defined as
\[\Gk=\gen{ \wg_1,\dots,\wg_\n \mid \wr_1,\dots \wr_{\n} },\]	
where the generators $\wg_i$ are meridian loops  corresponding to each arc $\arc_i$,    
 and the relators $\wr_n$ 
% are an cycle relation
 corresponding to each crossing $\cc_n$ are given as in Figure \ref{fig:WirtingerRelation} as usual.
\begin{figure}[H]
	$${\begin{tikzpicture}
		\draw [thick,->,blue] (1.1,-0.3)-- (1.1,0.3);
		\draw [line width=4pt, white] (0,0)--(2,0);
		\draw [thick,-stealth] (0,0)--(2,0);
		\node  at (2,0) [right] {$ \arc_i$~,};
		\node at (1.1,0.5)  {$ \wg_i$};
\end{tikzpicture}}$$
	\begin{equation*}
		\vcenter{\hbox{\begin{tikzpicture}
					\draw[-stealth,thick] (0,0) -- (1,1);
					\draw[line width=5pt,white]  (0,1)--(1,0);
					\draw[-stealth,thick]  (0,1)--(1,0) ;
					\node at (0.5,0.5) [above,yshift=0.5ex] {$\cc_n$};
					\node at (1,1) [xshift=1.5ex] {$\arc_j$};
					\node at (0,0) [xshift=-1ex] {$\arc_i$};
					\node at (1,0)  [xshift=1.5ex] {$\arc_k$};
				\end{tikzpicture}
		}}
		: \wr_n = \wg_j^{-1} \wg_k^{-1} \wg_i \wg_k ,~ 
		\vcenter{\hbox{\begin{tikzpicture}
					\draw[-stealth,thick]  (0,1)--(1,0) ;
					\draw[line width=5pt,white]  (0,0)--(1,1);
					\draw[-stealth,thick] (0,0) -- (1,1);
					
					\node at (0.5,0.5) [above, yshift=0.5ex] {$\cc_n$};
					\node at (1,1) [xshift=1.5ex] {$\arc_k$};
					\node at (0,1) [xshift=-1ex] {$\arc_i$};
					\node at (1,0)  [xshift=1.5ex] {$\arc_j$};
				\end{tikzpicture}
		}}
		: \wr_n = \wg_j^{-1} \wg_k \wg_i \wg_k^{-1}
	\end{equation*}
	\caption{Wirtinger relation}
	\label{fig:WirtingerRelation}
\end{figure}
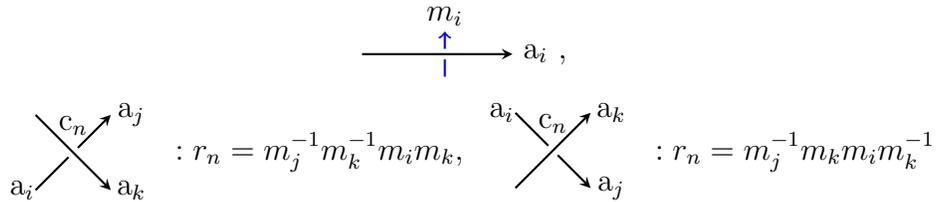

Let $\sl$ be a matrix group of $\sm{ a &b \\ c& d}$ with $ a d - b c =1$ and $\psl$ be the quotient group factored by $\Set{ \pm \id }$. Here, we choose  the coefficient ring to be  $\Cbb$ of complex numbers,  but almost all the arguments in this paper hold for arbitrary algebraically closed fields.

\begin{defn}
	An $\sl$ (resp. $\psl$)  representation of a link $\K$ is a  group homomorphism $\rho$ from $\Gk$ to $\sl$ (resp. $\psl$). A non-abelian  representation $\rho$ is \emph{parabolic} if   $\tr(\rho(\meri))=2$(resp. $\pm 2$) and $\rho(\meri)\neq \id$ (resp. $\pm\id$)  for any meridian $\meri$. 	
\end{defn}
\begin{rmk}
The definition  requires that $\rho$-image is non-abelian, as the original Riley's definition does \cite{riley_parabolic_1972}. However, 
we sometimes say about \emph{abelian parabolic representations}  just ignoring non-abelian requirement. 
In fact, abelian parabolic representations are very simple and we will briefly think of these in the next section and compute an example of  link case in Section \ref{sec:whiteheadlink}. 
\end{rmk}
 \begin{rmk}\label{rmk:bdparabolic}
	If the condition requiring  non-trivial meridian is omitted then such a  representation may not be boundary-parabolic  for a link  case. In spite of  $\tr(\rho(m))=2 $, the peripheral image of a link component may have the trivial (parabolic) meridian of $\id$ and a non-parabolic longitude. So we require the non-triviality for any meridian. 
 \end{rmk}

% we remove the condition in the definition. We sometimes say about \emph{abelian parabolic representations}. 
% Note that, as removing the non-abelian requirement, parabolic condition becomes, by definition, a Zariski closed condition  without any further argument.

An $\sl$-representation $\rho$ naturally induces a  $\psl$-representation $[\rho]$ by the composition of the canonical projection $\sl\to\psl$. Moreover, for parabolic representations, $\rho\mapsto[\rho]$ is a bijection as follows.
\begin{lem}\label{lem:pslsllifting}
	Any parabolic $\psl$ representation $\rho$ lifts  to a parabolic $\sl$ representation $\tilde\rho$ uniquely.
\end{lem}
\begin{proof}
	Consider $\rho$-images of the Wirtinger generators as $\psl$  matrix which satisfy the Wirtinger relations of Figure \ref{fig:WirtingerRelation}. There are two choices of lifting at each $\tilde\rho(\wg_i)$, i.e., $\tr(\tilde\rho(\wg_i))=$ $+2$ or $-2$. Take a lifting of $\tr=+2$, then  all Wirtinger relations are also satisfied in $\sl$. The uniqueness is obvious since the trace is fixed as  $+2$.
\end{proof}
Therefore,  we may focus on $\sl$ representations in many practical cases. 
% when we consider parabolic representations of $\Gk$. 
From now on, a \emph{representation}  always means a parabolic $\sl$-representation unless specified otherwise.

\subsection{Character variety}\label{sec:charactervariety}
Let $\P$ be the set of $\sl$ matrices of $\tr=2$, i.e.,
$$\P :=\{\left(\begin{smallmatrix}
	a &b \\ c &d
\end{smallmatrix} \right) \in \Cbb^4  \mid a d - b c = 1 \text{ and } a + d = 2\}.$$

Once we have a knot diagram $D$, let us define the following set
$$\Rparst :=\Rparst(K):= 
\Set{ (\bm\wg_1,\dots,\bm\wg_\n) \in \P^{\n}\subset\Cbb^{4n} \mid \bm\wr_1,\dots, \bm\wr_{\n} }
$$
where each bold  $\bm\wg_i$ stands for a matrix corresponding to each Wirtinger generator $\wg_i$ of $D$ and each $\bm\wr_i$ is also the  Wirtinger relator replacing the symbols of $\wg_i$  with $\bm\wg_i$ at each $\wr_i$.
Therefore, $\Rparst$ consists of $\sl$-representations of $\tr(\bm\wg_i)=2$.
Note that the relations $
\{\bm\wr_i=1\}$ is a system of polynomial equations of entries in $\sl$ matrices and hence the $\Rparst$ is Zariski-closed and an affine variety in $\Cbb^{4\n}$ by definition.

As we are  interested in representations up to conjugation, we let  $\bRparst$ be the quotient by conjugation, 
$\bRparst:= \Rparst /_\sim$  where  $$\rho \sim \rho' \iff A \rho A^{-1} = \rho' \text{  for some } A \in \sl$$
and $\Xparst$ be defined by
$
\Xparst:= \Rparst /_{\bm\sim}$
where  
$$ \rho \bm\sim \rho' \iff \tr(\rho(a)) = \tr(\rho'(a)) \text{  for all } a \in \Gk.$$

Since the coefficient $\Cbb$ is algebraically closed, the image $\im(\rho)$ of any reducible representation $\rho \in \Rpar$ is conjugated into $\Set{\sm{1&*\\0&1}}$   (\cite[Section 2]{riley_parabolic_1972}). Thus
$\rho$ is abelian if and only if $\rho$ is reducible. Considering the basic setting in  \cite{culler_varieties_1983} and the fact that $\Gk$ is generated by conjugations of a single generator, we have several   obvious facts as follows.\
\begin{align*}
	&	\bRab =\bRred=\Set{(\id,\dots,\id), [(\sm{1&1\\0&1},\dots,\sm{1&1\\0&1})] } = \text{two points},\\
	& \bRnt = \bRnab \sqcup \{[(\sm{1&1\\0&1},\dots,\sm{1&1\\0&1})]\},  \\
	&	\Xab=\Xred= \text{ a singleton of } [\rho] \text{ such that } \tr(\rho(a))=2 \text{ for all } a  \in \Gk,  \\
	&	\bRnab=\bRirr=\Xnab=\Xirr,\\
	& \Xparst = \Xab \sqcup \Xnab, 
\end{align*}
where $ab$, $nab$, $nt$, $red$, and $irr$ stand for \emph{abelian}, \emph{non-abelian}, \emph{non-trivial}, \emph{reducible}, and \emph{irreducible}, respectively and the equalities are just in the  set-theoretic sense. 
In particular, $\bRab$ has only two elements from the trivial representation and abelian representations, and
 $\Rnt$ is  the set of all representations except the trivial representation  $\rho_\circ =(\id,\dots,\id)$, i.e., the set  of only non-trivial abelian representations and non-abelian representations. 

Our main concern is the non-abelian parts $\Rnab$ and $\Xnab$ of $\Rparst$ and $\Xparst$,  and denote them simply by $\Rpar$ and $\Xpar$ respectively.
Now let us define the set of parabolic representations concretely as follows. 
\begin{defn}
$$\Rpar:=\Rpar(K):= 
	\{ \rho \in \Rparst(K) \mid   \bm\wg_i\neq\bm\wg_j \text{ for some }i,j \}$$
\end{defn}
\begin{prop}
	$\Rpar$ is  the set of parabolic representations.
\end{prop}
\begin{proof}
	Note that  $\rho$ is abelian if and only if $\rho(\wg_i)=\rho(\wg_j)$ for all $1\leq i,j \leq N$ since
	the Wirtinger relation $\rho(\wr_n)=\bm\wr_n=\id$    becomes $\bm\wg_j=\bm\wg_k^{\pm1}\bm\wg_i
	\bm\wg_k^{\mp1}=\bm\wg_i$ by the commutativity. The requirement of $\rho(\wg_i) \neq \id$ is also obvious.
\end{proof}
	
Since  $\bm\wg_i\neq\bm\wg_j$ is a Zariski open condition, the $\Rpar$ is a quasi-affine variety  by definition. But we can say $\Rpar$ is an affine variety as follows.
\begin{prop}
	$\Rpar$ is  Zariski closed.
\end{prop}
\begin{proof}
	Considering the whole representation variety $\Rep(\K)$ in the sense of Culler-Shalen  \cite{culler_varieties_1983}, it is obvious that 
$\Rpar^* =(t_m: \Rep(\K)\to \Cbb )^{-1}(2)$
where $t_\meri(\rho) = tr(\rho(\meri))$ is a regular function on $\Rep(\Gk)$ for a meridian loop $\meri$.
For the first eigenvalue $\mu$ of  $\rho(\meri)$ for any parabolic representation, $\mu=1$ and hence $\mu^2$ cannot be zero of the Alexander polynomial $\Delta(t)$ of $K$ because  $\Delta(1)=\pm1$. 
	By  \cite[Theorem 19]{klassen_representations_1991},  an abelian representation $\rho_{\mu=1}^{ab}$ with $\Delta(\mu^2)\neq0$ has a neighborhood in $R(K)$ entirely consisting of abelian representations and hence  a limit point of  parabolic representations cannot be an abelian representation. Since $\Rpar$ is quasi-affine, the  closure of $\Rpar$ in the usual topology is Zariski-closed \cite[Section I.10]{mumford_red_1999} and it completes the proof.
\end{proof}

Maybe one can think of $\mathcal{X}_{par}$ by GIT quotient for a more elaborated theory. 
It would be possible to see that the coordinate ring $\Cbb[\Xpar]$ is given by $ \Cbb[\mathcal{X}_{par}] /_{\!\!\sqrt{0}}$ the quotient by the nilradical.
In this paper,  we  consider only $\Xpar$, not $\mathcal{X}_{par}$, and usually in  the case when $\Xpar$ is zero-dimensional. For the case of $\dim_\Cbb (\Xpar) \neq 0$, every statement is for the zero-dimensional subvariety $\Xpar^{(0)}$ instead of $\Xpar$.

% $$|\Xpar|<\infty.$$ 
%%For example, 
%In summary, \emph{ the main object in the paper is $\Xpar = \bRnt$ under the condition that $\Set{\rho\in\Xpar}$  is a finite set.}
%	

%\subsection{Symbols}
%
%$\K$  : knot
%
%$\D$  : knot diagram. 
%$\n$ : number of crossing of $\D$ 
%
%$\wg_i$ : Wirtinger generator 
%$\cc_i$ : crossing
%$\aa_i$ : over-arc 
%
%$\meri$ : meridian curve
%
%$\P$.

\subsection{Parabolic quandle}\label{sec:paraquan}
In \cite{inoue_quandle_2013}, A. Inoue and Y. Kabaya introduced parabolic quandle. Notice 
 the subtle difference between our setting here and theirs. They considered  $\psl$ representations, but we consider $\sl$ representations. 
Although it seems not to produce any practical difference due to Proposition \ref{lem:pslsllifting}, we eventually obtain new results as dealing with the sign choices in the arc-colorings, in contrast to all the previous studies with sign-ambiguity. 
% This is, in fact, one of essential parts of our study. 

% In this paper, we usually see an element of $\Cbb^2$ as a 2-dimensional column vector. 

\begin{defn}\label{defn:mmap}
	A \emph{parabolic quandle map} $\M$ into $\sl$-matrices with trace 2 is defined by
	\begin{equation} \label{eqn:m}
		\M : \Cbb^2 \rightarrow \P  \quad\text{ by } \binom{x}{y} \mapsto 
		\begin{pmatrix}
			1-x y & x^2 \\ -y^2 & 1 + x y
		\end{pmatrix}.
	\end{equation} 
\end{defn}
\begin{prop}\label{prop:paramap}
	$ \M$ is surjective  and  two-to-one except at the origin $\qvec{0}{0}$.
\end{prop}
\begin{proof}
	Any element in $P$ is conjugate to a translation $\sm{1&1\\0&1}$ and can be written in the form, 
	$$\begin{pmatrix}x &z \\ y &w\end{pmatrix}
	\begin{pmatrix}1&1\\0&1 \end{pmatrix}
	\begin{pmatrix} x&z\\y&w \end{pmatrix}^{-1}=\begin{pmatrix}1-xy&x^2\\-y^2&1+xy\end{pmatrix}.$$
 The preimage of $\M$ is exactly  $\{\sm{x\\y}, \sm{-x\\-y}\}$.
\end{proof}
% Therefore  every non-zero vector in  $\Cbb^2$  is identified with an element in $\P\subset \sl$ with sign ambiguity.

Moreover, $\M$ has an important property as follows.
\begin{lem}\label{lem:paraequi}
	The parabolic quandle map $ \M$ is equivariant under left multiplication and conjugation by any $\sl$ matrix $A$, i.e. 
	$$
	\M(A b) = A \, \M(b) A^{-1} \text{ for $b\in \Cbb^2$ and $A \in \sl$}.
	$$  
\end{lem}
\begin{proof}
	Straightforward computation.
\end{proof}
Let us adopt usual quandle notations $\qr$ and $\ql$  for the following operation in $\Cbb^2$,
\begin{equation}\label{eqn:quandledef}
	a \qr b := \M (b)^{-1} a ~~\text{ and }~~ a \ql b := \M(b) a.
\end{equation}  
Then, one can check that $\Cbb^2$ with $\qr$ satisfy the quandle axioms by 
%the definition of (\ref{eqn:quandledef}) and 
Lemma \ref{lem:paraequi}, i.e. 
$$\textrm{(i)}~ a\qr a= a,~ 
\textrm{(ii)}~ (a\qr b ) \ql b = a,~ 
\textrm{(iii)}~ (a \qr b )\qr c= (a \qr c) \qr (b \qr c) ~\text{ for } a,b,c \in \Cbb^2.$$ 
Therefore, we let $\Q:=(\Cbb^2, \qr)$ be a \emph{parabolic quandle}.

\begin{rmk}
	The origin $\nullvec \in \Cbb^2$ would often cause  complication, although it satisfies the quandle axiom  just as the other  vectors. 
	It corresponds to the identity element in $\sl$ and always produces the trivial  representation which has the same character as abelian parabolic representations.
	We usually exclude it in practice but we sometimes need to consider the null vector in $\Q$ in particular for the link cases as in  Section \ref{sec:whiteheadlink}.
	% It usually produces a parabolic representation that is not a non-boundary parabolic. 	
\end{rmk}

%
% Sometimes we call an element in $\Q$  a \emph{quandle vector}.

Let us consider a system of parabolic quandle  for a given diagram $\D$ as follows,  
\begin{equation}
	\QQ^*:=\QQ_\D^* :=\{(\aa_1,\dots,\aa_N) \in \Q^N \mid r_1,\dots,r_{N} \}
\end{equation}
where the relation $r_n$ at crossing $\cc_n$ is given as in Figure \ref{fig:QuandleRelationWithSignAmbiguity}.
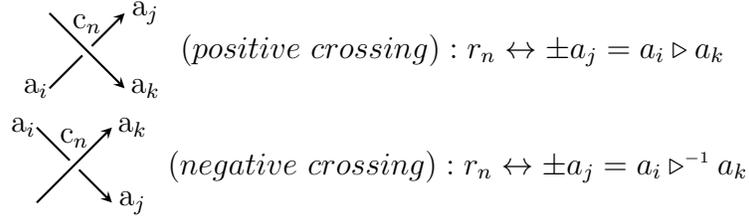
\begin{figure}[H]
	\begin{equation*}
		\begin{matrix}
			\vcenter{\hbox{\begin{tikzpicture}
						\draw[-stealth,thick] (0,0) -- (1,1);
						\draw[line width=5pt,white]  (0,1)--(1,0);
						\draw[-stealth,thick]  (0,1)--(1,0) ;
						\node at (0.5,0.5) [above, yshift=0.5ex] {$\cc_n$};
						\node at (1,1) [xshift=1.5ex] {$\arc_j$};
						\node at (0,0) [xshift=-1ex] {$\arc_i$};
						\node at (1,0)  [xshift=1.5ex] {$\arc_k$};
					\end{tikzpicture}
			}}
			(positive~ crossing): r_n\leftrightarrow \pm a_j=a_i \qr a_k ~~~\text{}  \\
			
			\vcenter{\hbox{\begin{tikzpicture}
						\draw[-stealth,thick]  (0,1)--(1,0) ;
						\draw[line width=5pt,white]  (0,0)--(1,1);
						\draw[-stealth,thick] (0,0) -- (1,1);
						
						\node at (0.5,0.5) [above, yshift=0.5ex] {$\cc_n$};
						\node at (1,1) [xshift=1.5ex] {$\arc_k$};
						\node at (0,1) [xshift=-1ex] {$\arc_i$};
						\node at (1,0)  [xshift=1.5ex] {$\arc_j$};
					\end{tikzpicture}
			}}
			(negative~ crossing): r_n\leftrightarrow \pm a_j=a_i \ql  a_k	
		\end{matrix}
	\end{equation*}
	\caption{Parabolic quandle relation with sign-ambiguity}
	\label{fig:QuandleRelationWithSignAmbiguity}
\end{figure}

$\QQ^*$ is also an affine variety in $\Cbb^{2N}$ since it is the set of solutions to polynomial equations  obtained as in Figure \ref{fig:QuandleRelationWithSignAmbiguity}, which are just equivalent to the Wirtinger relations.
Note that the parabolic quandle map $\M$ extends to the whole system $\QQ^*$ onto $\Rparst$ by the equivariance given in Lemma \ref{lem:paraequi}, as follows.
\begin{prop}\label{prop:surjM}
	There is a well-defined surjective regular map,
	$$\M: \QQ^* \to \Rparst ~\text{ by } a_i \mapsto \M(a_i) \text{ for } i=1,\dots,N. $$
	In particular, 
	$
	|\M^{-1} (\rho)|=  2^N 
	$
	in $\Rnt$ and $\M^{-1} (\rho_{\circ})=\{(0,\dots,0)\}$ with the trivial representation $\rho_{\circ}$.
	%	
	%	  $$
	%	  |(\MD)^{-1} (\rho)|= \left\{\begin{aligned}
		%	& 1 && \text{if $\rho$ is trivial, i.e., $\rho(\gamma)=\left(\begin{smallmatrix}1&0\\0&1\end{smallmatrix}\right)$ for all $\gamma\in \Gk$ }\\
		%	& 2^N && \text{ otherwise.} 
		%	\end{aligned} \right. $$
\end{prop}
\begin{proof}
	%For the trivial representation, $\MD$ is clearly bijection because $\M(\sm{0\\0})=\left(\begin{smallmatrix}1&0\\0&1\end{smallmatrix}\right)$.
	Let us consider a non-trivial representation $\rho$ and its Wirtinger generators 
	$$(A_1,\dots A_N)\in \Rnt.$$ 
	For any $A_i \in \P$, we can find $a_i\in \Q$ such that $\M(a_i)=A_i$. Then  $(a_1,\dots,a_N)$ should be contained in $\M^{-1}((A_1,\dots ,A_n))$ since
	the following observation by Lemma \ref{lem:paraequi} and (\ref{eqn:quandledef}),
	\begin{equation}\label{eqn:signedrel} 
		\begin{aligned}
			&\pm c= a\qr b &\text{ if and only if } &&&  \M (c)=\M (b)^{-1} \M (a) \M(b), \\
			&\pm c= a\ql b &\text{  if and only if }&&&  \M (c)=\M (b) \M (a) \M (b)^{-1}.
		\end{aligned}
	\end{equation}	
	
	Moreover, 
	% as seeing Proposition \ref{prop:paramap},
	 $\M$ is always well-defined for any sign-choice of arcs  
	since there are two preimages $+a_i$ and $-a_i$ of the same $A_i$ for each arc.
	Unless $a_i=\nullvec$, it assures that there are $2^N$ preimages of $\M$ where $N$ is the number of arcs in diagram $D$.
\end{proof}
Each entry of  $\aa=(\aa_1,\dots,\aa_N)$ of $\QQ$ assigned to each arc is called an \emph{arc-coloring} traditionally and we also  say  each arc is colored by $\Q$ and each $\aa_i$ is a \emph{coloring vector}. Remind that we use the notation $\arc_i$ for an arc and   $\aa_i$ for the corresponding coloring vector.  
The \emph{associated representation} $\rho_\aa:\Gk\to\sl$ with respect to an arc-coloring $\aa$ is  obtained by $$\rho_\aa(\wg_i):=\M(\aa_i)$$ 
where $\aa_i$ is the \emph{coloring vector} corresponding to Wirtinger generator $\wg_i$. 
Conversely, we say $\aa_{\rho}$ is the \emph{associated arc-coloring} with respect to a representation $\rho$ as well.
Now, let us define our main concern $$\QQ := \M^{-1}(\Rpar)$$,
called a \emph{parabolic quandle system} (with sign-ambiguity) and  each element gives  a parabolic representation by a regular function $\M$. Obviously $\QQ$ is affine variety. One of the goals in the paper is to construct a nice subvariety of $\QQ$ with a (set-theoretic) bijection with $\Xpar$.
\begin{rmk}
	One can consider $\QQ_+$ naively just as removing the sign-ambiguity from $\QQ$, i.e., replacing   $\pm$ with $+$  in Figure \ref{fig:QuandleRelationWithSignAmbiguity}. In this case, the quandle axioms are still established and  $\M$ is a well-defined map as well. However, the  parabolic quandle map $\M$ combined together along  $\D$ cannot be surjective to $\Rpar$ in general. To be precise, only a boundary-unipotent representation can be obtained through  $\QQ_+$. See Section \ref{sec:obs} for details.
\end{rmk}

% For a zero-coloring, the following observation is easy to check but should be kept in mind.
% \begin{lem}\label{lem:zerocoloring}
% 	Suppose that $\aa_i=0$ for some $i$.
% 	Then all coloring vectors  should be zero, i.e., 
% 	$\aa_1 =\cdots= \aa_N=0$.
% \end{lem}
% Remark that this lemma holds only for knot cases. For link cases, only  the colorings contained in the link-component including  zero-coloring are zero and the other coloring vectors may not be zero. 

% \begin{rmk}\label{rmk:nullcoloring}
% 	The zero-coloring $(0,\dots,0)$ and $\rho_\circ=(\id,\dots,\id)$ is very often produces various exceptions. So, from now on, let $\QQ$ stand for the set of only non-zero colorings  unless specified otherwise, i.e., $\M(\QQ)=\Rnt$.	
% \end{rmk} 

\subsection{Symplectic quandle}\label{sec:symquan}

We use symplectic quandle to investigate the structure of parabolic quandle equations. The notion of symplectic quandle  seems to appear first  in \cite{yetter_quandles_2003, navas_symplectic_2008} and it was a key feature of computations in \cite{jo_symplectic_2020}. 
Let us begin with the definition of symplectic quandle.
\begin{defn}
	Let $R$ be a commutative ring with the unity. A free module $M$ over $R$ equipped with an antisymmetric bilinear form $\lrbar{ ,} : M \times M \rightarrow R$ is called a \emph{symplectic quandle} where the quandle operation $\qr$ is given by 
	\begin{equation}\label{eqn:stracal}
		x \qr y := x+  \lrbar{ x, y}   y   
	\end{equation}
	for all  $x,y \in M$.
\end{defn}
% It is easily verified that a symplectic quandle satisfies the quandle axioms where its binary operation $\qr$ is given by 
% \begin{equation}\label{eqn:stracal}
% 	x \qr y := x+  \lrbar{ x, y}   y   
% \end{equation}
% for all  $x,y \in M$. 
From the property of the symplectic form $\lrbar{,}$, it is easily verified that the operation $\qr$ satisfies the quandle axioms and the inverse operation $\qr^{-1}$ is given by $x \qr^{-1} y = x-\lrbar{ x, y } y$.
We shall use the symplectic quandle approach for $\Q$ with the bilinear form given by the determinant, 
$$\lrbar {a,b}  := \textrm{det}(a,b) ~~\text{ for }a,b \in \Q.$$ 
Then the definition of (\ref{eqn:stracal}) is exactly the same as in Section \ref{sec:paraquan}, i.e.,
\begin{equation}
	a+\lrbar{ a,b} b = a \qr b = \M(a)^{-1} b ~~\text{ for }a,b \in \Q.
\end{equation}
Furthermore, we introduce a convenient notation $\widehat a$ for a column vector $a \in \Cbb^2$ by 
\begin{equation}\label{eq:} 
	\widehat a=\widehat{\qvec{x}{y}} := (-y,~ x).
\end{equation}
Then we have an expression for $\M$ as follows, 
\begin{equation}\label{eqn:Mqbyhat}
	\M(a) = I+ a \widehat{a} ~~\text{ and }~~ \M(a)^{-1} = I- a \widehat{a} ~~ \text{ for } a \in \Q,\\
\end{equation}
where $I$ is the identity matrix $\sm{1&0\\0&1}$.
From this, we can easily obtain that for $a\in\Q \text{ and } A\in\sl$,
\begin{equation}\label{eqn:reverse}
	\M(\pm a)=A ~~\text{ if and only if }~~ \M(\pm \sqrtii a)=A^{-1}.
\end{equation}
Moreover,   $\lrbar{ a,b} = \widehat a  b ~$ for $a,b\in\Q$ and hence
\begin{equation}\label{eqn:symphat}
	a\qr b = a + (\widehat  a b) b ~~\text{ and }~~ a \ql b = a-   (\widehat a b) b ~~\text{ for } a,b\in\Q.
\end{equation}
Although  the symplectic quandle and hat-notation may not be essential  mathematically, many parts in  
quandle computation become   convenient practically by using the notations.

\subsection{Riley polynomial}\label{sec:originalRiley}
Let us briefly review \emph{Riley polynomial}.  See \cite[Theorem 2]{riley_parabolic_1972} for details.
For a two-bridge knot $K$, the knot group $\Gk$  has a presentation, due to Schubert, of two generators $x_1$ and $x_2$ and a single relation. Riley computed all non-abelian parabolic representations $\rho$ as follows. We can put 
\begin{equation}\label{eqn:twogenerator}
	\rho(x_1) = \begin{pmatrix} 1 & 1 \\ 0 & 1 \end{pmatrix}  \textrm{ and } \rho(x_2) = \begin{pmatrix} 1 & 0 \\ -y & 1 \end{pmatrix}
\end{equation}
up to conjugation and these two matrices satisfy the one matrix relation  if and only if there exists a non-abelain representation $\rho$.  In \cite[Theorem 2]{riley_parabolic_1972}, he showed that the matrix relation is reduced to solving a single integral polynomial $\Riley(y)\in\Zbb[y]$. 
Riley named it \emph{representation polynomial} and is now  called  \emph{Riley polynomial}.
We can summarize that  Riley established a bijection,  
$$\{y\in \Cbb \mid \Riley(y)=0\}  ~\overset{\cong}{\longrightarrow}~ \Xpar ~~\text{ 
by } y \mapsto [\rho_y].$$
Furthermore, he showed the followings. 
\begin{enumerate}
	\item $\Riley(y)$ is  monic.\footnote{ In Riley's original notation and definition,
		$
		\Lambda(y)= 1+ c_1 y + \dots + c_{\lambda-1} y^{\lambda-1} +(-1)^{\lambda} y^\lambda \in \Zbb[y]
		$, i.e., the   constant term is fixed by $+1$. For the generalized Riley polynomial in this paper, the constant term may not be $\pm1$ in $\Zbb[y]$.  So we define the Riley polynomial $R(y)$ to be  with the monic leading coefficient, i.e., $R(y)=(-1)^\lambda \Lambda(y)$. }
	\item The constant term is $\pm1$.
	\item $\Riley(y)$ doesn't have any multiple root.
\end{enumerate}
%In \cite{kitano_note_2017}, T. Kitano and T. Morifuji  proved that the Riley polynomial determined a two-bridge knot completely.
Recently, Jo and Kim  \cite{jo_symplectic_2020} showed  that some additional properties of $\Riley(y)$ as follows
\begin{enumerate}
	\setcounter{enumi}{3}
	\item  Riley polynomial is square-splitting, i.e., $\Riley(u^2)=(-1)^{\deg(g)}g(u)g(-u)$ for $g(u)\in\Zbb[u]$.
	\item  the root $u$ of $y$ is contained in the trace field $\trfield$ as a unit of the ring of integers where  $\rho$ is the associated representation $\rho_y$. 
\end{enumerate}
We are going to generalize the Riley polynomial to a general knot, not limited to two bridge knot.
%, and check whether the above properties can be generalized or not.

%	 and the square-root $u$ of $y$ is contained in the trace field $\Qbb(\rho)$  as a unit of the ring of integers`' where  $\rho$ is the associated representation of the root \cite{jo_symplectic_2020}. 

\section{Parabolic quandle with sign-type}\label{sec:paraquandlesigntype}
Let us introduce \emph{sign-type} of a parabolic quandle system.
\subsection{Sign-types}
Let $\aa=(\aa_1,\dots,\aa_N) \in \QQ(D)$ of a knot diagram $D$. 
We will elaborate parabolic quandle system $\QQ$ by specifying a sign choice. 
\begin{defn}\label{def:signedrel}
	A \emph{parabolic quandle system $\Qe$ with  sign-type} $$\ee=(\e_1,\dots,\e_N)  \in \{ \pm 1\}^N$$ is the subset of $\QQ$ satisfying the following equation at each crossing $\cc_n$ for $n=1,\dots,N$ as in Figure \ref{fig:QuandleRelationWithSignType}.
		An element of $\Qe$ is called an  \emph{arc-coloring  with sign-type~$\ee$}.
\end{defn}

\begin{figure}[H]
	\begin{equation*}
			\vcenter{\hbox{\begin{tikzpicture}
						\draw[-stealth,thick] (0,0) -- (1,1);
						\draw[line width=5pt,white]  (0,1)--(1,0);
						\draw[-stealth,thick]  (0,1)--(1,0) ;
						\node at (0.5,0.5) [above, yshift=0.4ex] {$\cc_n$};
						\node at (1,1) [xshift=1.5ex] {$\arc_j$};
						\node at (0,0) [xshift=-1ex] {$\arc_i$};
						\node at (1,0)  [xshift=1.5ex] {$\arc_k$};
					\end{tikzpicture}
			}}
			\leftrightarrow \e_n a_j=a_i \qr a_k,~~~ 
			\vcenter{\hbox{\begin{tikzpicture}
						\draw[-stealth,thick]  (0,1)--(1,0) ;
						\draw[line width=5pt,white]  (0,0)--(1,1);
						\draw[-stealth,thick] (0,0) -- (1,1);
						
						\node at (0.5,0.5) [above, yshift=0.4ex] {$\cc_n$};
						\node at (1,1) [xshift=1.5ex] {$\arc_k$};
						\node at (0,1) [xshift=-1ex] {$\arc_i$};
						\node at (1,0)  [xshift=1.5ex] {$\arc_j$};
					\end{tikzpicture}
			}}
			\leftrightarrow \e_n a_j=a_i \ql  a_k.
		\end{equation*}	
		\caption{Parabolic quandle relation with sign-type}
\label{fig:QuandleRelationWithSignType}
\end{figure}
% Let a restriction  of $\M$ to $\Qe$ be denoted by  
% $\Me: \Qe \to \Rnt$.
% , is called the parabolic quandle map with sign-type $\ee$.
	% If an arc-coloring $\aa$ is contained in $\Qe$ then $\aa$ has a sign-type $\ee$ 
	% Each equation in $\Qe$ is called a \emph{(parabolic) quandle equation with sign-type $\e$}.
	% Note that sometimes $\Qe$ may be empty for some $\e$.

	 We  remark about the redundancy of quandle relation with sign-type.
	Any single choice of relation $r_i$  in $\QQ$ (with sign-ambiguity) is redundant since  the single   relation $\wr_i$ among Wirtinger relations is derived from the other relations. The corresponding statement for $\Qe$ is as follows.  
	\begin{prop}\label{prop:signed1rel}
		Let $\Qe=\{(a_1,\dots,a_N) \in \Q^N \mid r_1^\e,\dots,r_{N}^\e \}$ where the relations are given by Definition \ref{def:signedrel}. Let $\aa_*$ be an arc-coloring satisfying only $N-1$ relations with removing $r_m^\e$  as follows
		$$ \aa_* \in \{(\aa_1,\dots,\aa_N) \in \Q^N \mid r_1^\e,\dots,r_{m-1}^\e, r_{m+1}^\e,\dots,r_{N}^\e \}.$$ 
		In this case, $\aa_*$ satisfy $r_m^{\e}$  with sigh-ambiguity $\e_m\in \{\pm1\}$ as in Figure \ref{fig:QuandleRelationWithSignAmbiguity}.
	\end{prop}
	\begin{proof}
		The corresponding $N-1$ Wirtinger relations except $r_m$ are satisfied and hence the remaining  relation $r_m$ is automatically satisfied with sign-ambiguity.
	\end{proof}
	% \begin{proof}
		% Since $\aa$ satisfies $N-1$ relations in $\QQ$, the corresponding Wirtinger relations of $\Rpar$ in terms of $\sl$ matrices are satisfied and hence the remaining Wirtinger relation is automatically satisfied.  By considering  (\ref{eqn:signedrel}), we obtain that the corresponding quandle relation $r_m$ holds up to sign.
	% \end{proof}
	But $\aa_*$ may or may not satisfy the $r_m^\e$ strictly. The obstruction class of $\rho$ affects the sign of $\e_m$. See Section \ref{sec:obs}. 
	% strictly and we can conclude it as follows.
	% \begin{prop}
	% 	If the sign-type is fixed,  
	% 	one parabolic quandle relation $r_m^\e$ never be deduced from the other $N-1$ relation $r_1^\e,\dots,r_{m-1}^\e,r_{m+1}^\e,\dots,r_N^\e$.	
	% \end{prop}
	% \begin{proof}
	% 	It is obvious by confirming the situation that $r_m^\e$ is not  satisfied by mismatching between the total-sign in  Section \ref{sec:totalsign}  and  the obstruction class in Section \ref{sec:obs}. 
	% \end{proof}

\subsection{Total-sign}	\label{sec:totalsign}

	 We have several basic properties on sign-type as follows. 
\begin{lem}\label{lem:signtypelem}
	\begin{enumerate}
		\item $\QQ$ is the disjoint union of  $\Qe$'s, i.e., 
		\begin{equation}\label{eqn:signtypedecomp}
			\QQ = \bigsqcup_{\ee\in \{\pm1\}^N} \Qe.
		\end{equation}	
		\item  For a restriction  of $\M$ to $\Qe$, i.e.,   
		$\Me:=\restr{\M}{\Qe} : \Qe \to \Rpar$,  is a $2$-to-$1$ map and the preimage consists of $\{+\aa,-\aa \}$ unless $\Qe$ is an empty set. 
	\end{enumerate}
\end{lem}
\begin{proof}
	Suppose that $\aa \in \Qe \bigcap \QQ_{\ee'}$, i.e., there is the same $\aa$ for the different $\ee$ and $\ee'$. Then $\aa_n=-\aa_n=0$   since  $\e_n \neq \e'_n$ for some $n$  in Figure \ref{fig:QuandleRelationWithSignType}, and  all coloring vectors  in $\aa$ must be zero. 
	For the second assertion, it is obvious that if $\aa \in \Qe$ then $-\aa$ is also in $\Qe$. If   $\aa' \in \Qe$ different from $\pm\aa$, there is a pair of $\aa_i$ and $\aa_j$ such that  $\aa_i'=\pm\aa_i$ and $\aa_j'=\mp \aa_j$. It implies that the corresponding sign-types are mismatched. It contradicts  $\aa'\in \Qe$.
\end{proof}

% Remark that
	% if one considers the zero coloring $\aa_\circ=(\sm{0\\0},\dots,\sm{0\\0})$ in $\QQ^*$, then $\aa_\circ$ is a solution to the quandle equation of $\QQ_\ee$ with  arbitrary sign-type $\ee$.%  This is another reason to define  $\QQ$ with removing the zero coloring, as in Remark \ref{rmk:nullcoloring}.

Let us consider  the set of all possible sign-types for $\rho$, 
$$S(\rho) := \{ \e \in \{+,-\}^N \mid  \Me^{-1}(\rho)  \neq \varnothing\}.$$
%  \subset \M^{-1}(\rho) 
Then, we can count the number of sign-types in $S(\rho)$ as follows. 
\begin{lem}
	 $|S(\rho) | = 2^{N-1}$.
\end{lem}
\begin{proof}
	Consider the surjectivity of $\M$ in Proposition \ref{prop:surjM} and hence $|\M^{-1}(\rho)|=2^N$ . If $\Me^{-1}(\rho)\neq\varnothing$ then  $|\Me^{-1}(\rho)|=2$ by Lemma \ref{lem:signtypelem} and hence  $|S(\rho) | \geq 2^{N-1}$.  As considering $\Me^{-1}(\rho) \cap \M_{\ee'}^{-1}(\rho) =\varnothing$ for $\ee'\neq\ee$, we conclude that  $|S(\rho) | = 2^{N-1}$.  
\end{proof}
% In other words, for a  $\rho$, half of a total of $2^N$ sign-types is possible to give $\rho$ and the other half is impossible.

Let us define the \emph{total-sign} $[\ee]$ of a sign-type $\ee$ to be the product of all $\e_i$'s, i.e., $[\ee] :=\ee_1  \ee_2\cdots \ee_N$. Then 
\begin{prop}\label{prop:totalsignconstant}
	$[\ee] \in \{ +1,-1\}$ is constant on $\M^{-1}(\rho)$ 
\end{prop}
\begin{proof}
	Consider $\aa=\{\aa_1,\dots,\aa_N\}$ and $\aa'=\{\aa'_1,\dots,\aa'_N\}$ in  $\M^{-1}(\rho)$. Since $\aa_i=\pm\aa'_i$ for each $i$, we can   transform $\aa$ into  $\aa'$ by a finite sequence of the sign-change of a coloring vector $\aa_i$ and it suffices to consider only one step. If one change the sign of a single coloring vector $\aa_n$ into $-\aa_n$ then the corresponding change $\ee'$ of the sign-type $\ee=(\ee_1,\dots,\ee_N)$ only occur at two $\ee_i$ and $\ee_j$ which are the initial and terminal crossings of the arc $\aa_n$. Therefore $[\ee]=[\ee']$ and it completes the proof  without loss of generality.  
\end{proof}

\begin{nota}\label{nota:crossingindex}
	For convenience, from now on, we choose  compatible indices for arcs $\{\arc_1,\dots,\arc_N\}$ and crossings $\{\cc_1,\dots,\cc_N\}$ as taking the ending point of each oriented arc $\arc_i$ to be $\cc_i$. Sometimes,
	an index $i \in \Zbb$ is outside the range of $\{1,\dots,N\}$, in which case it is always considered  modulo $N$.  The indices of $\e_i$ and $\aa_i$ also follow those of $\cc_i$ and $\arc_i$ respectively.	
	\end{nota}

Then we can write down an explicit bijection between $\Qe$ and $\QQ_{\e'}$ as follows.
\begin{prop}\label{prop:Qebijection}
	For any given $\tau \in \{+1,-1\}$, we have
	$$\Phi: \Qe \to \QQ_{\ee'} ~\text{  by }~ a_i \mapsto a'_i:= \tau 	a_i \prod_{k=1,\dots,i-1} { \e'_k}/{\e_k} $$
	with $\tau=a'_1/a_1$.

\end{prop}
\begin{proof}
	For $i=1$, put $a'_1=\tau a_1$.
Then the formula  is obtained by induction on $i$, due to the quandle relation with sign-type in Definition \ref{def:signedrel}.
\end{proof}

Now we know that the total-sign of $\rho$ is an invariant for $\rho$ and a natural question arises asking  the geometric meaning of  the total-sign. In the next section, we will prove that the total sign is nothing but  the obstruction class of $\rho$.

\subsection{Obstruction classes}\label{sec:obs}

For a parabolic representation with  $\tr(\rho(\meri))=+2$ for a meridian $\meri$,  the trace of longitude can be $+2$ or $-2$, i.e., a parabolic $\sl$ representation may or may not be boundary-unipotent. It is exactly determined by an \emph{obstruction class} of $\rho$. Recall that a representation $\widetilde{\rho} : \Gk\to\sl$ is called \emph{boundary-unipotent}  if  $\tr(\widetilde{\rho}(\gamma))=2$ for every loop $\gamma$ in the boundary of a small tubular neighborhood of $K$. The obstruction to a boundary-unipotent lifting can be identified with an element of $\{ \pm 1\}$, which is referred to as the \emph{obstruction class} of $\rho$. We refer \cite{cho_hikamiinoue_2020} for details. 
\begin{prop} \cite[{Proposition 2.2}]{cho_hikamiinoue_2020} \label{prop:knot} Let $\rho :\Gk\rightarrow \psl$ be a boundary-parabolic representation. 
	The obstruction class $\ob(\rho)$ of $\rho$ is half of $\tr(\widetilde{\rho}(\lambda))$ where $\widetilde{\rho}$ is a  $\sl$ lifting of $\rho$ and $\lambda$ is any longitude of $K$.
\end{prop}
%\begin{proof}
%	Considering any Wirtinger presenation of $\pi_1(M)$, it is easy to check that $\rho$ has only two $\sl$ liftings $\widetilde{\rho}_+$ and $\widetilde{\rho}_- : \pi_1(M)\rightarrow \sl$ such that $\tr(\widetilde{\rho}_+(\mu))=2$ and $\tr(\widetilde{\rho}_-(\mu))=-2$, respectively. Here $\mu$ is a merdian of the knot. Since $\pi_1(\partial M)$ is an abelian group generated by $\mu$ and the longitude $\yv$, $\rho$ admits a $(\sl,P)$-lifting if and only if $\tr(\widetilde{\rho}_+(\yv))=2$. Therefore, by definition, the obstruction class $\alpha(\rho)\in\{\pm1\}$ coincides with half of $\tr(\widetilde{\rho}_+(\yv))$. On the other hand, the canonical longitude $\yv$ is contained in the commutator subgroup of $\pi_1(M)$. Thus it should be expressed in Wirtinger generators of even length and we have $\widetilde{\rho}_+(\yv)=\widetilde{\rho}_-(\yv)$.
%\end{proof}
For a parabolic $\sl$ representation $\rho$,
we also refer to the obstruction class of $\rho$  as the obstruction class  of $p\circ \rho$ with $p:\sl\to\psl$.

\begin{thm} \label{thm:obs} Let $\aa=(\aa_1,\dots,\aa_N)$ be an arc-coloring of  $\D$ of sign-type $\ee=(\e_1,\cdots,\e_N)$ with the associated representation $\rho_\aa \in \Rpar$. % $:  \Gk \rightarrow \psl$ is non-trivial. 
	Then the obstruction class $\ob$ of $\rho_\aa$ is the total-sign of $\rho$, i.e., $\ob=[\ee]=\e_1\cdots \e_n \in \{ \pm 1\}$. 	
\end{thm}
\begin{proof}

	Suppose arc index is given as in Figure \ref{fig:indexalong}. Let $m_i$ be the Wirtinger generator corresponding to an arc $\arc_i$.   
	%	We choose $m_1$ as a meridian. 
	We may assume $\aa_1 = \sm{1\\0}$, i.e. $\rho_\aa(m_1) =  
	\sm{1&1\\0&1}$
	%	
	%	\begin{psmallmatrix}
		%		1&1\\[1.5pt]
		%		0&1 
		%	\end{psmallmatrix}
	%	$
	%
	by conjugation since $\rho_\aa$ is non-trivial.
	
	\begin{figure}[H]
		\centering
		\scalebox{1}{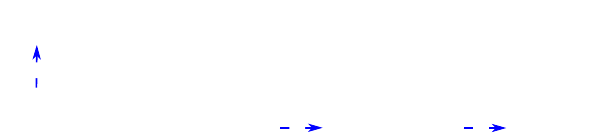}
		\caption{Arc index along a diagram $\D$}
		\label{fig:indexalong}
	\end{figure}
	
	Let $\sigma_k$ be the crossing-sign of the $k$-th crossing from the left in Figure \ref{fig:indexalong}. 
	Then the blackboard framed longitude $\lb$ is given by 
	$$\lb = m_{i_1}^{\sigma_1} \cdots m_{i_n}^{\sigma_n}.$$
	By the definition of $\rho_\aa$ and using the hat-notation in Section \ref{sec:symquan}, we have
	\begin{equation*}
		\begin{array}{rcl}
			\rho_\aa(\lb) &=&  \M(\aa_{i_1})^{\sigma_1} \cdots  \M(\aa_{i_n})^{\sigma_n} \\[3pt]
			&=& (I+\sigma_1 \mkern 2mu \aa_{i_1}\widehat{\aa_{i_1}}) \cdots (I+\sigma_n\mkern 2mu  \aa_{i_n}\widehat{\aa_{i_n}}\big)
		\end{array}
	\end{equation*} 
	%	where the hat-notation means $\widehat{\binom{c}{d}} := (-d \mkern 5mu c)$. Note that $ \M(x) = I+ x \widehat{x}$ and $ \M(x)^{-1} = I- x \widehat{x}$ for $x \in \Cbb^2$. 
	Expanding the above equation, we have
	\begin{equation} \label{eqn:lbf} 
		\begin{array}{rcl}
			\rho_\aa(\lb) &=& I + \sum \sigma_{j_1} \sigma_{j_2}\cdots\sigma_{j_k} \aa_{i_{j_1}} \widehat{\aa_{i_{j_1}}} \aa_{i_{j_2}} \widehat{\aa_{i_{j_2}}} \cdots \aa_{i_{j_k}} \widehat{\aa_{i_{j_k}}} \\[3pt]
			&=& I + \sum \sigma_{j_1} \sigma_{j_2}\cdots\sigma_{j_k} \aa_{i_{j_1}} \lrbar {\aa_{i_{j_1}}, \aa_{i_{j_2}}}  \cdots \lrbar {\aa_{i_{j_{k-1}}}, \aa_{i_{j_k}}} \widehat{\aa_{i_{j_k}}} \\[3pt]
			&=& I + \sum \sigma_{j_1} \sigma_{j_2}\cdots\sigma_{j_k} \lrbar {\aa_{i_{j_1}}, \aa_{i_{j_2}}}  \cdots \lrbar {\aa_{i_{j_{k-1}}}, \aa_{i_{j_k}}} \aa_{i_{j_1}} \widehat{\aa_{i_{j_k}}}
		\end{array}
	\end{equation} where the summations are over all possible $1\leq j_1< \cdots < j_k \leq n$. 
	
	% Therefore, 
	% \begin{equation} \label{eqn:lbf2}
		% \begin{array}{rcl}
			% \tr(\rho_{\aaa}(\lb))& =& 2+ \sum  \sigma_{j_1}
			% \sigma_{j_2}\cdots\sigma_{j_k} \langle s_{i_{j_1}}, s_{i_{j_2}} \rangle  \cdots \langle s_{i_{j_{k-1}}}, s_{i_{j_k}} \rangle \tr(s_{i_{j_1}}  \widehat{s_{i_{j_k}}})\\
			% &=& 2+ \sum  \sigma_{j_1} \sigma_{j_2}\cdots\sigma_{j_k} \langle s_{i_{j_1}}, s_{i_{j_2}} \rangle  \cdots \langle s_{i_{j_{k-1}}}, s_{i_{j_k}} \rangle \langle s_{i_{j_{k}}}, s_{i_{j_1}} \rangle 
			% \end{array}
		% \end{equation}
	On the other hand, from the definition of type $(\e_1,\cdots,\e_n)$ arc-coloring, we have
	\begin{equation*}
		\begin{array}{rcl}
			\e_2 \aa_2 &=& \aa_1 + \sigma_1 \lrbar{ \aa_1, \aa_{i_1}} \aa_{i_1}\\
			\e_3 \aa_3 &=& \aa_2 + \sigma_2 \lrbar {\aa_2, \aa_{i_2} } \aa_{i_2}\\
			&\vdots &\\
			\e_n \aa_n &=& \aa_{n-1} + \sigma_{n-1} \lrbar {\aa_{n-1}, \aa_{i_{n-1}} } \aa_{i_{n-1}} \\
			\e_1 \aa_1 &=& \aa_n + \sigma_n \lrbar {\aa_n, \aa_{i_n} } \aa_{i_n}.
		\end{array}
	\end{equation*} Plugging the first equation to $\e_2$ times the second one, we have
	\begin{equation*}
		\begin{array}{rcl}
			\e_2 \e_3 \aa_3 &=&  \aa_1 + \sigma_1 \lrbar{ \aa_1, \aa_{i_1} } \aa_{i_1}+ \sigma_2\lrbar{ \e_2 \aa_2, \aa_{i_2} } \aa_{i_2} \\
			&=& \aa_1 + \sigma_1 \lrbar{ \aa_1, \aa_{i_1} } \aa_{i_1}+ \sigma_2\lrbar{ \aa_1, \aa_{i_2} } \aa_{i_2}+\sigma_1\sigma_2 \lrbar{ \aa_1, \aa_{i_1} } \lrbar{ \aa_{i_1}, \aa_{i_2} } \aa_{i_2}.
		\end{array}
	\end{equation*} Keep plugging till the last equation, we obtain
	\begin{equation} \label{eqn:ep}
		\begin{array}{rcl} 
			\e_1 \cdots \e_n \aa_1 &=& \aa_1+\sum \sigma_{j_1} \cdots\sigma_{j_k} \lrbar{ \aa_{1}, \aa_{i_{j_1}} }  \cdots \lrbar{ \aa_{i_{j_{k-1}}}, \aa_{i_{j_k}} } {\aa_{i_{j_k}}}
		\end{array}
	\end{equation} where the summation is over all possible $1\leq j_1< \cdots < j_k \leq n$. 
	Letting $\aa_i = \binom{c_i}{d_i} \in \Q$, we have $d_{i_{j_1}}=\lrbar{\aa_1,\aa_{i_{j_1}}}$ and  the first entry of the equation (\ref{eqn:ep}) gives
	\begin{equation*} 
		[\e]= \e_1 \cdots \e_n  = 1 +\sum \sigma_{j_1} \cdots\sigma_{j_k}  d_{i_{j_1}}  \lrbar{ \aa_{i_{j_{1}}}, \aa_{i_{j_{2}}} } \cdots \lrbar{ \aa_{i_{j_{k-1}}}, \aa_{i_{j_k}} } c_{i_{j_k}}.
	\end{equation*} Recall that $\aa_1 = \binom{1}{0}$. Also the equation (\ref{eqn:lbf}) gives
	\begin{equation*} 
		\rho_\aa(\lb) = I + \sum \sigma_{j_1} \sigma_{j_2}\cdots\sigma_{j_k}  \lrbar{ \aa_{i_{j_1}}, \aa_{i_{j_2}} }  \cdots \lrbar{ \aa_{i_{j_{k-1}}}, \aa_{i_{j_k}} } 
		\sm{c_{i_{j_1}} \\  d_{i_{j_1}}} 
		(-d_{i_{j_k}} \ \ c_{i_{j_k}}).
	\end{equation*}  
	Therefore, the (2,2)-entry of $\rho_\aa(\lb)$ is $[\e]$. 
	On the other hand, $\rho_\aa(\lb)$ commutes with $\rho_\aa(m_1) = \sm{1&1\\0&1}$ so it should be of the form  $\pm \sm{1 &*\\0&1}$. 
	Therefore, we have $\tr(\rho_\aa(\lb))=2 [\e]$ and thus
	%	 $\tr(\rho_\aa(\yv))=2 \e_1 \cdots \e_n$. T
	$\ob=[\e]$ finally follows from Proposition \ref{prop:knot}.
\end{proof}
\begin{rmk}
	Theorem \ref{thm:obs} holds for any abelian parabolic representation but doesn't hold for the trivial representation $\rho_\circ$ of null vectors since  it has the obstruction class of $+1$ but an arbitrary choice of $\e$ is possible, i.e., we can put $[\e]=\e_1\cdots\e_N=-1$ for $\rho_\circ$.
\end{rmk}
	% One may want to consider $\Q$ containing the null-vector $\sm{0\\0}$. Then, for any choice of $\ee$, $\Qe$  always contains a common null solution, i.e., $a_i=\sm{0\\0}$ for all $a_i$, whose associated representation $\rho$ is trivial. Remark that  Theorem \ref{thm:obs} doesn't hold for the trivial representation since  such $\rho$ has the obstruction class is $+1$ even if $\e_1\cdots\e_N=-1$. 
% \end{rmk}

\subsection{Parabolic representations}% and Parabolic quandles with sign-type}
Let us decompose the set of parabolic representations along the obstruction class,% of $\rho$,
\[\Rpar = \Rpar^+ \sqcup \Rpar^-.\]
The decomposition of $\QQ$ along sign-type assembles two groups according to the obstruction class as follows, 
$$\QQ= \bigsqcup_{\ee\in \{\pm1\}^N} \Qe ~=~ \QQ_+ \sqcup \QQ_-$$
where $ \QQ_+:= \bigsqcup\limits_{\e_1 \cdots \e_N=+1} \Qe$ and $\QQ_-:= \bigsqcup\limits_{\e_1 \cdots \e_N=-1} \Qe$.\\
All $\Rpar^+$, $\Rpar^-$, and $\Q_\ee$ are obviously affine varieties. 

By the surjectivity of $\M$ in Proposition \ref{prop:surjM}, 
it is also obvious that  
\begin{align*}
&\M_+:=\M|_{\QQ_+}:\QQ_+ \to \Rpar^+ ~~\text{and}~~ \\
&\M_-:=\M|_{\QQ_-}:\QQ_- \to \Rpar^-	
\end{align*}
are surjective.
These $\M_+$ and $\M_-$ are  $2^N$-to-$1$ maps as well as $\M$. We can restrict $\M_+$ (resp. $\M_-$) to a $2$-to-$1$ surjective map $\Me$ for any sign-type 
$\ee=(\e_1,\dots,\e_N)$ with  $\e_1\cdots\e_N=+1$ (resp. $-1$). We summarize this correspondence   as follows. 

\begin{thm}\label{thm:2-1map}
	Let $\rho : \Gk\rightarrow \sl$ be a  parabolic representation whose obstruction class is
	$\ob \in \{\pm1\}$.
	For  any sign-type 
$\ee=(\e_1,\dots,\e_N)$ satisfying  $\e_1\cdots\e_N=\ob $, 
there exists an 
 arc-coloring $\aa=(\aa_1,\dots,\aa_N)$ of the sign-type $\e$ such that the associated representation   $\rho_\aa$ is the given $\rho$. Moreover, all such arc-colorings are exactly  $\aa$ and $-\aa$.

 \begin{proof} 
	In Lemma \ref{lem:signtypelem}, we have a surjective $2$-to-$1$ map $\Me \to \Rpar$. By Theorem \ref{thm:obs}, the image of $\Me$ is contained in $\Rpar^\ob$ and 
the parabolic quandle map restricted to $\Qe$,   
$$\Me:=\M \big|_{\Qe} : \Qe \to \Rpar^\ob$$ which
is  surjective.
Since we have a given representation $\rho$ with the obstruction class $\ob$, $\Rpar^\ob$ cannot be empty and we can find a preimage $\aa$ such that $\Me(\aa)= \rho$. The last assertion is obvious since $\Me$ is  a $2$-to-$1$ map by Lemma \ref{lem:signtypelem}.
 \end{proof} 

 % Moreover $\Me^{-1}(\rho)$ is exactly two elements of $\pm(a_1,\dots,a_N)$. 
\end{thm}

Therefore, only two pre-specified sign-type $\ee_+$ and $\ee_-$ for $\{+1,-1\}$ obstructions are just enough to consider the whole non-trivial representations.  We often denote the chosen two sign-type $\ee_+$ and $\ee_-$ by $\ee_\pm$, or simply $\ee$ if not confused, i.e., 
\begin{equation}\label{eqn:Qdecpm}
	\Me:= \M_{\ee_+} \sqcup \M_{\ee_-}: \QQ_{\ee_+} \sqcup~ \QQ_{\ee_-} \to \Rpar^+ \sqcup~ \Rpar^-.
\end{equation}
In short,
$\Me: \Qe \to \Rpar$ and it is exactly $2$-$1$ and surjective.

Now, let us consider $\Xpar=\bRpar$, our main concern.  
Let  $\Qe /_\sim$ be denoted by $\wQe$ where $ \aa \sim \aa'$ if and only if $B \aa =  \aa'$ for $B \in \sl$. Then we have a one-to-one  correspondence as follows.

\begin{thm}\label{thm:perfect1-1}
The induced map
$$ \wMe : \wQe \longrightarrow \Xpar$$
is well-defined and bijective.
\end{thm}
\begin{proof}
By Lemma \ref{lem:paraequi}, 
$$\Me (B \aa ) = B\Me( \aa) B^{-1}$$ 
and hence  $\Me$ is  well-defined. By Theorem \ref{thm:2-1map},  there are two arc-colorings for a representation $\rho$, i.e.,  $$\{ \aa,-\aa\} = \Me^{-1}(\rho).$$ 
They are in the same class by the left multiplication, i.e., $-\aa= \sm{-1&0\\0&-1}\aa$ and it completes the proof.
\end{proof}
 By this correspondence, we can investigate $\Xpar$ through $\wQe$ with any convenient sign-type $\ee=\ee_+\sqcup\ee_-$.
%  =(\ee_1,\dots,\ee_N)$.
We usually use $\e=(1,\dots,\pm1,\dots,1)$ with  one specified signed crossing for the obstruction $\ob=\ee_1\cdots\ee_N = \pm1$. We will denote such a sign-type of only $i$-th  crossing  simply by \emph{sign-type of $\e_i=\pm1$}.

Finally, we should be careful that this theorem holds only for knots and fails  for links.

\subsection{Normalization}\label{sec:normalization}
By choosing two arcs $\arc_i$ and $\arc_j$ we have an explicit set-theoretic bijection
, 
 $$\wQe \cong \Qu \sqcup \Qv $$
  in a similar way to \cite[Section 2]{riley_parabolic_1972}.
\begin{prop}\label{prop:normalizeQ}
Let us fix two arcs $\arc_i$ and $\arc_j$ in a knot diagram $D$. 
For a parabolic arc-coloring $\aa \in \wQe$ there is a unique element in 
$\QQ_u$ with nonzero $u\in\Cbb$  if  $\rho_\aa(\wg_i\wg_j)\neq\rho_\aa(\wg_j\wg_i)$ (resp. in $\QQ_v$  with nonzero $v\in\Cbb$ if otherwise), where  
%	$a_i=\binom{1}{0}$ and $a_j=\binom{0}{u}$ with $\u \in \Cbb$ (resp. $\binom{v}{0}$ with $v \in \Cbb$) by left-multiplication if $\wg_i \wg_j\neq \wg_i \wg_j$  (resp. $\wg_i \wg_j = \wg_i \wg_j$)
%	and  $ \wQe$ is decomposed into two disjoint sets $ \Qu \sqcup \Qv $ %\addtag\label{eqn:normalizeQ}\]
%	such that
\begin{align*}
	\Qu &=\Set{ (\aa_1,\dots,\aa_N)\in \Qe \given a_i=\sm{1\\0},a_j=\sm{0\\u} \text{ for } u\in \CCzero } , \\
	\Qv &=\Set{ (\aa_1,\dots,\aa_N)\in \Qe  \given a_i=\sm{1\\0},a_j=\sm{v\\0}  \text{ for } v\in \CCzero } / \sim, 
\end{align*}
where $\aa \sim \aa'$ if and only if $\sm{1 & *\\ 0&1}\aa=\aa'$.	
\end{prop}
\begin{proof}	
By \cite[Lemma 1]{riley_parabolic_1972}, two matrices in $\P$ can be normalized to $\binom{ 1~1}{0~1}$ and $\binom{1~0}{x~1}$ (resp.  $\binom{1~x}{0~1}$) for nonzero $x \in \Cbb$. By the correspondence between $\Q$ and $\P$ we obtain $\binom{1}{0}$ and $\binom{0}{u}$(resp. $\binom{v}{0}$) for a nonzero $u\in\Cbb$. If we consider left-multiplications by $\sl$ on the whole $\wQe$, then the isotropy subgroups are $\Set{\binom{1~0}{0~1} } $ and $\Set{ \binom{1~*}{0~1} }$ respectively and it proves the uniqueness.
\end{proof}
When we compute   $\Qv$ by solving quandle equations in $\QQ^*$, the solution contains abelian representations. So  we should remove such a component manually to obtain $\Qv$. 
\begin{rmk}
	Note that $\QQ_u$ is an algebraic variety by definition and we can see that $\QQ_v$  is also algebraic.
However, $\wQe$ and $\QQ_u \sqcup \QQ_v$ may not be algebraically isomorphic in general, since the non-commuting assumption for $\QQ_u$ is not Zariski closed condition.
\end{rmk}
%  by  the further decomposition along a sequence of generators in \cite[Section 2]{riley_parabolic_1972}.
%
Nevertheless, this normalization is not only very powerful in computing representations but also is essential in proving Theorem \ref{thm:Qcoef} later.

\section{Riley polynomial and $\u$-polynomial}
\subsection{Generalized Riley polynomial}

\subsubsection{Roots of Riley polynomial}
Let us see that, by direct computation from (\ref{eqn:twogenerator}), each root of Riley polynomial $\Riley(y)$ comes from  
$$ y= 2-\tr(\rho(x_1)\rho(x_2)),$$
where  $x_1 \asttt x_2 \in \Gk$ is the loop that winds the topmost crossing when we put the two-bridge knot $\K$ in  Conway normal form. 
%So, a $\yv$-value is just a linear transform $2-tr(\rho(\gamma))$ of the trace at a loop $\gamma= x_1\asttt x_2$. 
%
%\begin{defn}\label{def:Lvalue}
%	Let us define \emph{$\yv$-value} at $\gamma \in \Gk$, denoted by $\yv_\gamma$, of a representation $\rho$,
%	$$\yv_{\gamma}(\rho):= 2-\tr \comp \rho (\gamma).$$
%Let us define \emph{$\yv$-polynomial} $\F_\gamma(\yv)$ at $\gamma$  by
%\begin{equation}\label{eqn:RileyL}
%\F_\gamma(\yv):= \prod_{\rho \in \ntRep}(\yv-\yv_{\gamma}(\rho)).
%\end{equation}

%\end{defn}

Let us consider in general  
$$\yv_\gamma(\rho):=  2-\tr \comp \rho (\gamma)$$ for any $\gamma \in \Gk$. Note that $\yv_\gamma$ is  a regular function on $\Xpar$ and, in particular, it does not depend on the orientation of $\gamma$ as follows.
\begin{lem}\label{lem:ori} Let $\gamma^{-1}$ denote a loop $\gamma\in \Gk$ with its opposite orientation. Then we have $\yv_\gamma=\yv_{\gamma^{-1}}$.
\end{lem}
\begin{proof} From the Caley-Hamilton equation, we have $A -\tr(A) I + A^{-1}=0$ for any $A \in \sl$. Taking the trace, we have $\tr(A)=\tr(A^{-1})$. 
\end{proof}

In particular, let us consider $y_\gamma(\rho)$ when $\gamma$ is a crossing loop in Figure \ref{fig:crossingloop}.

\begin{figure}[H]
\centering
\scalebox{1}{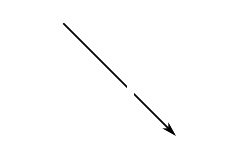}
\caption{Loops winding a crossing.}
\label{fig:crossingloop}
\end{figure}
Let $\arc_i$ and $\arc_j$ be arcs at a crossing as in Figure \ref{fig:crossingloop}. Let $\meri_i$ and $\meri_j$ be the Wirtinger generator corresponding to $\arc_i$ and $\arc_j$.

\begin{prop} \label{prop:crossing} Let $\alpha$ and $\beta \in \Gk$ be loops as in Figure \ref{fig:crossingloop}. Then we have 
\begin{enumerate}
	\item  $\yv_\alpha=\yv_{\meri_i\asttt\meri_j}=\yv_{\meri_j\asttt\meri_i}$.
	\item $\yv_\beta=\yv_{\meri_i\asttt\meri_j^{-1}}=\yv_{\meri_i^{-1}\asttt\meri_j} 
	=\yv_{\meri_j\asttt\meri_i^{-1}}=\yv_{\meri_j^{-1}\asttt\meri_i}$.
	\item 	$\yv_\beta =-\yv_\alpha$.
	\item $\rho(\meri_i)\rho(\meri_j)=\rho(\meri_j)\rho(\meri_i)$ if and only if $\yv_{\beta}(\rho)=\yv_{\alpha}(\rho)=0$. 
\end{enumerate}	
%	
%	$\yv_\beta =-\yv_\alpha$. Moreover, $\rho$-images of Wirtinger generators around a crossing commute if and only if 
\end{prop}
\begin{proof}  (1) and (2) are straightforward consequences by the Wirtinger generators. By the trace relation,  we have 
\begin{equation*}
	\begin{array}{rcl}
		\yv_\alpha &=& \yv_{m_1*m_2}=2 -  \tr \big({\rho}( m_1 \asttt m_2) \big)  \\[2pt]
		&=& 2 - \tr \big({\rho}( m_1) \big) \tr \big({\rho}( m_2) \big) + \tr \big({\rho}( m_1^{-1} \asttt m_2) \big)\\[2pt]
		&=& \tr({\rho}(m_1^{-1}\asttt m_2)) -2 \\[2pt]
		&=& -\yv_\beta.
	\end{array}
\end{equation*}	
On the other hand, any two $\sl$ matrices $A$ and $B$ of trace $2$ commute if and only if $\tr(AB)=2$. Therefore, $\rho$-images of Wirtinger generators around a crossing commute if and only if  $\yv_{\beta}(\rho)=\yv_{\alpha}(\rho)=0$.
\end{proof}
Note that the above Proposition holds regardless of over/under information of the crossing.

\begin{defn}\label{def:RileyL}
Let us define \emph{$\yv$-value} at a crossing $\cc$ as $\yv_\alpha(\rho)$, 
often denoted by $\yv_{\cc}$ or $\yv_{ij}$ where $\arc_i$ and $\arc_j$ are  
two arcs at the crossing $\cc$ in Figure \ref{fig:crossingloop}.  
Moreover, a \emph{(generalized)  Riley polynomial $\Riley_\cc(y)$ at a crossing $\cc$} is defined by 
\begin{align*}
	\Riley_\cc(y)&:= \prod_{\rho \in \Xpar}(y-\yv_{\cc}(\rho)),
\end{align*}
where the subindex $\cc$ is often omitted or replaced by $ij$ if the context is clear.
\end{defn}
Remark that we need the assumption that $\Xpar$ is zero-dimensional for the definition of generalized Riley polynomials. 

For a two-bridge knot, if we choose a crossing $\cc$ as the topmost crossing, the `generalized' Riley polynomial $\Riley_\cc(y)$ coincides with the original Riley's polynomial.  We usually omit the word `generalized' and simply say \emph{Riley polynomial}	
%	 at a crossing $\cc$} or a meridian pair $ij$ 
unless we need to emphasize it.

% $\yv_c=-\yv_\beta$ and
%\[\Riley_\cc(y) = \prod_{\rho \in \naRep}(\yv-\yv_{\beta}(\rho)). \]
%\begin{equation}
%\end{equation}
%\F_\cc (\yv) &:= \F_\beta(\yv)\\
%y \Riley_\cc(y) &:=F_\alpha(y).  

%By Proposition \ref{prop:crossing}, the relation between Riley polynomial and $\yv$-polynomial is 
%\[ y \Riley_\cc(y) = (-1)^{\deg(\F_\cc)} \F_\cc(-y). \addtag \label{eqn:RileyLambda}\]

\begin{rmk}
When we  define a $\yv$-value $\yv_\cc$ at a crossing $\cc$, there are two choices of crossing loops : $\alpha$ and $\beta$  in Figure \ref{fig:crossingloop}.   
Maybe someone prefers to choose $\beta$ instead of $\alpha$
and it differs only in sign.
%	 since $
%	\Riley_\cc(y) = \prod_{\rho \in \Rnab}(y-\yv_{\cc}(\rho)) 
%	$ with the convention of $\yv_c$ chosen $\beta$ woule be more natural.  
%	We, however, choose $\alpha$ for $\yv_\cc$. The reason is that the $\yv$-value at a crossing $\cc$ itself has  geometric meaning and better proporties. 
Remark that the negative $\yv$-value, i.e., $\yv_\beta$ choosing $\beta$ instead of $\alpha$ is exactly the same as \emph{crossing label} in \cite[Remark 5.11]{kim_octahedral_2018}. When $\rho$ is a holonomy representation of hyperbolic structure, this $y_\beta$ is also the same as a geometric quantity,  called \emph{intercusp parameter} in \cite{neumann_intercusp_2016}.
%	To distinguish the original Riley polynomial $\Riley(y)$ we use different notation $\Riley_\cc(\yv)$ than $\Riley_\cc(y)$ which exactly has the solution of $\yv$-values at the corssing $\cc$ as follows,
%	\begin{align*}
%		\Riley_\cc(\yv)&:= \prod_{\rho \in \Rnab}(\yv-\yv_{\cc}(\rho)),
%	\end{align*}
%	Sometimes, we call $\Riley_\cc(\yv)$ is $\yv$-polynomial at a crossing $\cc$.
\end{rmk}

% % One can generalize Lemma \ref{lem:crossing} to as follows.

% % \begin{lem} \label{lem:crossing2} Let us consider $n$ half-twists and let $\gamma$ and $\alpha$ be loops as in Figure \ref{fig:crossingloop2}. Then we have $\yv_{\\alpha}(\rho) =P_n(-\yv_{\gamma}(\rho))$ where $P_n$ is the $n$-th Fibonacci polynomial.{\color{red}Check and add the definition of Fib. poly.}
% % \end{lem}
% % \begin{figure}[!h]
% % 		\centering
% % 		\scalebox{1}{\input{figures/crossingloop2.pdf_tex}}
% % 		\caption{Loops winding n half-twists.}
% % 		\label{fig:crossingloop2}
% % \end{figure}
% % \begin{proof} Let $m_i$ be the Wirtinger generator as in Figure \ref{fig:crossingloop2}. Then we have $\gamma = m_2 \asttt m_1$ and $\alpha= m_1 \asttt m_{n+1}$. Let $x_i \in \CC^2$ such that $ \M(x_i) = \rho(m_i)$. We may assume that $x_{i+1}=x_{i-1} - \langle x_{i-1},x_i \rangle x_i$. Taking $\langle \cdot, x_i \rangle$, we have $-\langle x_{i},x_{i+1} \rangle = \langle x_{i-1},x_i \rangle$. 
% % \end{proof}
\begin{rmk} \label{rmk:constantcounterexample}
As we  reviewed in Section \ref{sec:originalRiley}, the original Riley polynomial is monic integral and has no multiple root, and the constant term is $\pm1$. These properties doesn't hold anymore. We can easily find an example of a generalized Riley polynomial with multiple roots and non-unit constant term.  Let us see  $7_4$ knot of Section \ref{sec:7_4}.    The Riley polynomial at the crossing $\cc_5$ is  $\Riley_{\cc_5}=(y^3+y^2+13y-4)(y^2+y+1)^2$. We can see the multiple roots and the constant term of  $-4$. 
Non-integral Riley polynomial seems to be much rarer and firstly appears in $9_{29}$ knot. We will prove that, unlike the other properties, the square-splitting property always hold.
\end{rmk}
%
%\subsection{Properties}
%Let us investigate the properties of the original Riley polynomial in Section \ref{sec:originalRiley}, for a two bridge knot has distinct roots and is monic and the constant term is $\pm1$. We can easily find a counter example that these three properties are failed in generalized Riley polynomial.
%the Riley polynomial $\Riley_{\cc_5}=(y^3+y^2+13y-4)(y^2+y+1)^2$. 
%Where the geometric component $y^3+y^2+13y-4$ has constant term of $-4$.
%
%
%
%Let us see the computation for $7_4$ knot in Section \ref{sec:7_4}, which is a two bridget knot with Schbert form  $(15,11)$ or Conway form $[3,1,3]$.
%There are two irreducible component of $\Xpar(7_4)$ with  the geometric component having the hyperbolic representation of $3$ parabolic representations and non-geometric component of $4$ parabolic representations. Let us look at the crossing $\cc_5$, the Riley polynomial $\Riley_{\cc_5}=(y^3+y^2+13y-4)(y^2+y+1)^2$. 
%Where the geometric component $y^3+y^2+13y-4$ has constant term of $-4$.
%
%$9_29$ first example.
%
%
% 

\subsection{$\yv$-value and $\u$-value}\label{sec:uvalyval}

An important point is that  $\yv$-value $\yv_{ij}$ is  the square of the symplectic form  $\lrbar{ \aa_i,\aa_j}$ given by  determinant. 
\begin{lem}\label{lem:lambdaalpha} Consider an arc-coloring $\aa=(\aa_1,\dots,\aa_N)$ associated to $\rho$. Let $\aa_i$ and $\aa_j$ be  two quandle vectors and  $\wg_i$ and $\wg_j$ be the corresponding  Wirtinger generators respectively. Then we have 
$$\yv_{ij}= \lrbar{ \aa_i,\aa_j}^2.$$ % and $\yv_{\aaa}(\gamma) = \lrbar {\aa_1,\aa_2} ^2$.
\end{lem} 
\begin{proof} We have
\begin{equation*}
\begin{array}{rcl}
	\yv_{ij}(\rho) &=& 2- \textrm{tr}(\rho(\wg_i \asttt \wg_j))  \\
	&=& 2 - \textrm{tr}\big( I +\aa_i \widehat{\aa_i})(I + \aa_j \widehat{\aa_j} \big)  \\[2pt]
	&=& 2- \textrm{tr}\big(I+\aa_i \widehat{\aa_i}+\aa_j \widehat{\aa_j} +\aa_i \widehat{\aa_i}\aa_j \widehat{\aa_j} \big)  \\[2pt]
	&=& - \widehat{\aa_i} \aa_j \tr(\aa_i\widehat{\aa_j})=\lrbar{ \aa_i,\aa_j}^2. \qedhere
\end{array}
\end{equation*}	
\end{proof}
As recalling $\lrbar{\aa_i,\aa_j}=\det(\aa_i \aa_j)$, we will call the square root of $\yv$-value   \emph{$\u$-value}.\footnote{The name $\yv$-value or $\u$-value is taken from the variable symbols used in \cite[222p for $\u$, 227p for $\yv$]{riley_parabolic_1972}. The assignments here are essentially the same as those of R. Riley originally used, although `quandle' didn't appear explicitly.}
More precisely, we define $\u$-value as follows,
\begin{defn}\label{def:uvalue}
$\u_\ij:=\lrbar{a_i,a_j}.$
\end{defn} 
Note that,  
when we consider of $u$-value of a representation $\rho$,  the square root $\u_\ij$ of $\yv_\ij$ has two choices of $\pm\sqrt{\yv_\ij}$ and  the $\u$-value $\u_\ij=\lrbar{\aa_i,\aa_j}$ also changes according to the sign change of $\aa_i$ and $\aa_j$. 
If, however, one fixes a sign-type $\ee$ then $\lrbar{\aa_i,\aa_j}$ is well-defined without sign-ambiguity since the signs of $\aa_i$ and $\aa_j$ for $\rho$ are changed simultaneously as in Theorem \ref{thm:2-1map}. When a sign-type $\ee$ needs to be specified, we use the notation $\u^\ee$ or $\u_{\ij}^\ee$.

\begin{rmk}\label{rmk:uvalue}
For a given sign-type $\ee$, the $\u$-value  $\u^\ee_\ij$ has no sign-ambiguity. However, even though the sign-type $\ee$ is determined, $\u$-value  $\u_\cc$ at a crossing $\cc$ has sign-ambiguity because of the following Lemma \ref{lem:signambiguity}. 
To define $\u$-value  at a crossing, we have to decide which two arcs will be used in $\lrbar{a_i,a_j}$ among three arcs around the crossing. 
We will set the  assigning rule for \emph{ crossing $\u$-value} so that $\u^\ee_\cc=\u^\ee_{ij}$ where $a_i$ is the under-incoming arc and $a_j$ is the over-arc.  
\end{rmk}

\begin{lem}\label{lem:signambiguity}
Let three coloring vectors at a crossing $\cc$ be $a$, $a'$ and $b$ as in \\
\makebox[\textwidth]{
\begin{tikzpicture}[scale=0.6]
	\draw[,thick] (0,0) -- (1,1);
	\draw[line width=5pt,white]  (0,1)--(1,0);
	\draw[,thick]  (0,1)--(1,0) ;
	%			\node at (0.5,0.5) [above, yshift=0.4ex] {$\cc_n$};
	\node at (1,1) [xshift=1ex] {$a$};
	\node at (0,0) [xshift=-1ex] {$a'$};
	\node at (1,0)  [xshift=1ex] {$b$};
\end{tikzpicture} 
}
and $\e_\cc$ is the sign-type at the crossing $\cc$.
Then, 
\begin{equation}
\lrbar {a,b} =\e_\cc  \lrbar {a',b}
\end{equation}	
\end{lem}
\begin{proof}
$\lrbar{ a',b }  =\lrbar {\e_\cc(a \pm  (\widehat a  b) b), b} = \e_\cc\lrbar{ a,b} \pm    \widehat a  b \lrbar {b, b} 
= \e_\cc\lrbar {a,b}$.
\end{proof}

%Therefore, in order to define the \emph{determinant value} at a crossing $\cc$ without sign-ambiguity we need to specify sign-type.
%
%a rule choosing two arcs near the crossing, in other words, just fixing the sign-type is insufficient to remove the sign-ambiguity from the determinant value.
%We will set the rule by choosing the incoming two arcs at the crossing as follows.
%

\subsection{$u$-polynomials}
Let us  define a \emph{$\u$-polynomial} $g(u)$ whose roots are $\u$-values, similar to Riley polynomial $\Riley(y)$ having $\yv$-values as roots. Unlike $y$-values, we need to fix sign-type $\ee$ to determine  $\u$-values without sign-ambiguity. 

\begin{lem}
For a given sign-type $\ee$,
a $\u$-value $\u_{ij}^\ee$ is well-defined in $\wQe$ without sign-ambiguity. % ~$(=$\wQep$).
\end{lem}
\begin{proof}
For $B\in \sl$, $\lrbar{ B a_i , B a_j} =  
%	\det [ B a_i  ~ B a_j]=
\det B \det ( a_i ~ a_j )  = \lrbar{ a_i, a_j}$.
\end{proof}

\begin{defn}\label{def:acpoly}
Let us define a  \emph{$\u$-polynomial} $g(u)$ at a pair of two arcs $\arc_i$ and $\arc_j$   by
\begin{equation}
g(u):=g^{\ee}_{ij}(u):= \prod_{ \aa \in \Xpar } (u- \u_\ij^\ee ),
\end{equation}
where  $\aa=(a_1,\dots,a_N)$ is an arc-coloring with sign-type $\ee$.
%	 We have a bijection $\wQe$  and $\Xpar$ and $\Xpar$ is finite-set. 
%	%
%	In many cases, we want to see only nonabelian representations.	A  \emph{arc-pairing polynomial} or \emph{$u$-polynomial} $g(u)$ is defined by 
%	\[ g^{\ee}_{ij}(u) := \frac{1}{u}~ \widetilde g^{\ee}_{ij}(u). \]	
If 
$\arc_i$ and $\arc_j$ are  two arcs at a crossing $\cc$ in Remark \ref{rmk:uvalue},
the subindex  $ij$ of 	$g^{\ee}_{ij}$ is often   replaced by $\cc$. These subscript $ij$, $\cc$, or $\ee$ is also often omitted unless confusing.
\end{defn} 
%For a knot $K$, we have bijection between $\wQe$ and $\Xpar$. Therefore
%\begin{equation}\label{eqn:uPolyXpar}
%	g_\ij(u)= \prod_{ \aa \in \Xpar  } (u- \u_\ij^\ee ).
%\end{equation}
%But, if $\arc_i$ and $\arc_j$ are in the different component of a link, (\ref{eqn:uPolyXpar})  is not true anymore since $\Xpar \neq \wQe$ for a link.

If $g(u)$ is not contained in $\Qbb[u]$, this definition would not be meaningful. We prove $g(u) \in \Qbb[u]$ as follows.

\begin{thm}\label{thm:Qcoef}
For any $i,j$,	the $u$-polynomial $g_{ij}(u)$ is contained in  $\Qbb[u]$.
\end{thm}
\begin{proof}
Let $g(u)=g_{ij}(u)$ and look at the decomposition of $\wQe$ into $\Qu$ and $\Qv$ by Proposition \ref{prop:normalizeQ}. We put $a_i=\sm{1\\0}$, $a_j=\sm{0\\u}$ for $\QQ_u$ (resp. $a_j=\binom{v}{0}$ for $\QQ_v$) and the other colorings are given indeterminate variables $t_1,t_2,\dots,t_{2(N-2)}$. 
By the assumption that $\Xpar$ is zero-dimensional,  $\Qu$ and $\Qv$ are finite sets.
Note that each arc-paring $u_{ij}$ in $\Qv$ becomes zero  and hence $g(u)=u^{d'}g'(u)$ with  $d':=\abs{\Qv}$. 
Each entry $u_\ij, t_1, t_2,\dots$ of in an element of $\Qu$ is a root of the system of integer coefficient polynomial equations given by quandle relation. They should be  algebraic numbers  since $\Qu$ is finite. Note that  $Gal(\overline \Qbb /\Qbb)$ acts on $\Qu$ by a permutation. 
As considering 
\[ g_\ij(u) = \prod (u-u_{ij})=u^n+C_{n-1}u^{n-1}+\dots+C_1u+C_0\]
and each coefficient $C_k$ is a symmetric function of $u_{ij}$'s and doesn't change under Galois conjugate. We conclude that $g_\ij\in\Qbb[u]$ since every coefficient $C_k$ is invariant under $Gal(\overline \Qbb /\Qbb)$.
% 
%By the definition of $g(u)$, $\deg(g')=\abs{\Qu}$, denoted by $d\in \NN$. 
%Let $u_1$ be a root of $g'(u)$. 
%Then, there exists 
%$$\aaa_1\in \wQDep \text{ such that } \aaa_1=(\aa_1,\dots \aa_N) \text{ and } 		\aa_j=\sm{0\\u}.$$ 
%Consider the minimum polynomial $g_1(u)$ of $u_1$,
%\begin{equation}\label{key}
%g_1(u) = \prod (u-u_i) \in \ZZ[u].
%\end{equation}
%and a Galois conjugate isomorphsm $\sigma_i\in Gal_\QQ(\QQ(u)) $ such that $\sigma_i(u_1)=u_i$.  By the field extension theorem, we extend $\sigma_i$ to be $\overline \QQ \to \overline \QQ$.
%
%Since $\sigma_i(\wQDep)=\wQDep$,  $g_1(u)$ divides $g'(u)$ and
% we obtain  $g'(u)=g_1(u)h(u)$. If $h_1(u_1)=0$ then there exist another $\aaa_1' \in \wQDep$ such that $(\aaa_1')_j=\sm{0\\u_1}$ and $\aaa_1 \neq \aaa_1'$.
%Because every $\sigma_i$ is an isomorphism of $\overline \QQ$, $\sigma_i(\aaa_1)\neq \sigma_j(\aaa_1')$ for all $i$,$j$. 
%Thereofore $g_1(u)$ also divides $h(u)$ and,
% by degree argument, there is $d_1\in \NN$ such that
%$$g'(u)=g_1^{d_1}(u) h_1(u) \text{ with } h_1(u_1)\neq 0 .$$
%If $h_1(u)$ is not constant, pick up a root $u_2$ of $h_1(u)$ and consider the minimum polynomial $g_2(u)$ of $u_2$. By the same procedure, $g_2(u)$ divides $h_1(u)$ and we obtain
%$$ g'(u)=g_1^{d_1}(u)g_2^{d_2}(u)h_2(u).$$
%This process is repeated until the remaining $h_k(u)$ is constant for $k\in \NN$. Then we eventually obtain   $g_1, \dots, g_k \in \ZZ[u]$ such that $g_1(u)\cdots g_k(u)$ and 
%$g'(u)$ have the same roots. Since $g(u)$ is monic we have 
%\[ g(u)= c u^{d'}g_1(u)\cdots g_k(u)\]
% for some constant $c$ is contained in $\QQ[u]$.
\end{proof}

Let us summarize that by the correspondence of Theorem \ref{thm:perfect1-1} we have 
\[ \{ u \in \Cbb \mid  g^\ee_\cc(u)=0 \}~~  \overset{1:1} \longleftrightarrow ~~
\{ y \in \Cbb \mid  \Riley_\cc(y)=0 \}~~  \overset{1:1} \longleftrightarrow ~~
\Xpar. \addtag\label{eqn:1-1nab}\]
As we mention it in Section \ref{sec:charactervariety}, the higher dimensional part of $\Xpar$ is simply ignored in this paper.  
%Note that there is only one conjugacy class of abelian representation, which gives a $u$ factor of $\widetilde g(u)$. So, 
%\[ \{ u \in \Cbb \mid  g(u)=0 \}~~  \overset{1:1} \longleftrightarrow ~~ \Rab.\]
%In a similar way, we define a unreduced Riley polynomial 
%\[ \widetilde \Riley(y):= y \Riley(y), \]
%whose solutions correspond to non-trivial representations.
\begin{rmk}\label{rmk:upolyQcoeff}
Theorem \ref{thm:Qcoef} and Definition \ref{def:acpoly} works for only knots since  the bijection in Theorem 
\ref{thm:perfect1-1} between $\wQe$ and $\Xpar$ fails  for links.  If $\arc_i$ and $\arc_j$ belong to different components of the link diagram, both $\pm \u_\ij^\ee$ always are contained in $\wQe$, 
even though the sign-type $\ee$ is fixed.
\end{rmk}

\subsection{$u$-polynomial and sign-types}

When the $\ij$ index of $g_{ij}^{\ee}(u)$ doesn't need to be specified, We will simply denote  it as $g_\ee(u)$.
By the decomposition along the obstruction class  in (\ref{eqn:Qdecpm}), we obtain
$$
g_{\ee_\pm}(u)=g_{\ee_+}(u) g_{\ee_-}(u).
$$
Remark that any abelian representation has  positive obstruction.
%  and hence the additional factor $u$ is contained in $ g_{\ee_+}(u)$.

\begin{nota}
For a given polynomial $g(u)$, let    
$g^*(u):= (-1)^{\deg(g)}g(-u)$. Then 
the roots of $g^*(u)$ are   $\{-u_1,\dots,-u_n\}$ if and only if the roots of $g(u)$ are     $\{u_1,\dots,u_n\}$. The sign correction is for $g^*(u)$ to be monic. 
This $*$-notation  will be  often used in this paper.
\end{nota}

Now we consider the change of $ g_\ee(u)$ for different sign-types.
\begin{thm}\label{thm:signchaging}
Let $\ee=(\e_1,\dots,\ee_N)$ and $\ee'=(\e'_1,\dots,\e'_N$) be two sign-types with the same total-sign, i.e., $\e_1\cdots\e_N= \e'_1\cdots\e'_N$. Then 
$$
g_{ij}^{\ee'}(u) = 
\begin{dcases*}
g_{ij}^{\ee}(u) & for $\e_i \,  \e_{i+1}\dots \e_{j-1} = +\e'_i \,\e'_{i+1}\dots \e'_{j-1}$\\
(g_{ij}^{\ee})^*(u) & for $\e_i \,\e_{i+1}\dots \e_{j-1}=-\e'_i \,\e'_{i+1}\dots \e'_{j-1}$.
\end{dcases*}
$$
\end{thm}

\begin{proof}
	
By the  explicit bijection between 
$\Phi: \Qe \to \QQ_{\ee'}$ in Proposition \ref{prop:Qebijection},
 $a'_i  =\tau \frac{\ee'_1 \cdots \ee'_{i-1} }{\ee_1 \cdots \ee_{i-1}}a_i$. So we have
\begin{align*}
\lrbar{ a'_i,a'_j } &= \lrbar{ \tau \tfrac{\ee'_1 \cdots \ee'_{i-1} }{\ee_1 \cdots \ee_{i-1}} a_i,\tau \tfrac{\ee'_1 \cdots \ee'_{j-1} }{\ee_1 \cdots \ee_{j-1}} a_j } \\
&=   \tfrac{\ee'_i \cdots \ee'_{j-1} }{\ee_i \cdots \ee_{j-1}} \lrbar{ a_i,  a_j }.  \qedhere
\end{align*}
\end{proof}

Now we can conclude that there are  four choices of $\u$-polynomial  $g_{\ee}(u)=g_{\ee_+}(u)g_{\ee_-}(u)$ along the choice of the sign-type $\ee_+$ and $\ee_-$ along each obstruction class,
\[ g_{\ee_+}(u) g_{\ee_-}(u),~\; g_{\ee_+}^*(u) g_{\ee_-}(u),~\; g_{\ee_+}(u) g_{\ee_-}^*(u),~\; g_{\ee_+}^*(u) g_{\ee_-}^*(u),\]
since each  $g_{\ee_+}(u)$ and $g_{\ee_-}(u)$ has two choices  of $g_{\ee_\pm}(u)$ and $g_{\ee_\pm}^*(u)$ along the choices   of $\ee_+$ and $\ee_-$ respectively.

Let us consider a canonical choice of sign-type.
For positive total-sign, we have the most natural choice, $\ee_+=(+,\dots,+)$ and hence, denote it simply by $g_+(u)$. But for negative total-sign, there is no canonical choice of sign-type. So $g^-(u)$ is not determined whether it is $g_-(u)$ or $ g_-^*(u)$, unless the sign-type $\ee_-$ is specified. 
However, if we look at a specific crossing, we can think of the following rule to choose a negative sign-type.
\begin{rmk}[Canonical sign-type]\label{rmk:canonicalsigntype}

Recall Notation \ref{nota:crossingindex} and Remark \ref{rmk:uvalue}. In the convention, the index $i$ of incoming arc $\arc_i$ at a crossing $\cc$ is the same as the index of the crossing itself, i.e.,  $\cc=\cc_i$.
For a crossing $\cc$,
we have a canonical choice of sign-type $\ee_\cc$  as follows, 
\begin{align*}
&\ee_\cc^+=(+1,...+1)  &&\text{ and } \\
&\ee_\cc^-=(+1,\dots,-1,\dots,+1) &&\text{ where $-1$ sign only occurs in $\cc_i$.}
\end{align*} 
Therefore, with the definition of crossing $u$-value in Remark  \ref{rmk:uvalue}, we can  define  \emph{canonical $\u$-polynomial} at a crossing $\cc$ without a specified sign-type. As seeing the computation of $8_{18}$ in Section \ref{sec:8_18}, the canonical $u$-polynomials have more sensitive properties than Riley polynomials concerning the symmetry of knot, and hence they seem to deserve   further study. 

\end{rmk}
We also often omit the symbol $\ee$ if there is no need to specify a sign-type.

\subsubsection{Square-splitting}
The following proposition is obvious. 
\begin{prop}\label{prop:Rileyg(u)}
For Riley polynomial $\Riley_\ij(y)$ and $\u$-polynomial $g_\ij(u)$ at a pair of $\arc_i$ and $\arc_j$ with any sign-type $\ee$,
$$\Riley_\ij(u^2)=g_\ij(u)g_\ij^*(u).$$
In particular, $f(u) \mid g_\ij(u)$ implies $f(u)f^*(u) \mid \Riley_\ij(u^2)$.
\end{prop}
\begin{proof}
Recall the definition of $\Riley_\ij(y)$ and $g_\ij(u)$ and consider the bijection between the solution sets by (\ref{eqn:1-1nab}).
\begin{align*}
g_\ij(u)g_\ij(-u) &=\prod_{ \rho \in \Xpar }  (u-  u_{_\ij}) (-u -u_{_\ij}) 
=	\prod_{ \rho \in \Xpar }  (-u^2 + u_{_\ij}^2)\\
&= (-1)^{\deg g_\ij} \prod_{ \rho \in \Xpar } (u^2 - \yv_{_\ij})=(-1)^{\deg g_\ij}\Riley_\ij(u^2). %\qedhere 
\end{align*}  
The second assertion is obvious from the above statement.
\end{proof}
Proposition \ref{prop:Rileyg(u)} comes  simply from $\yv_\ij=\u_\ij^2$ and it alone doesn't imply $g(\u)\in \Qbb[u]$. However, by combining with  Theorem \ref{thm:Qcoef}, we  obtain a decomposition theorem immediately as follows.
\begin{thm}\label{thm:R=gg}
A generalized Riley polynomial $\Riley_\ij(y)$  has rational coefficients and is square-splitting, i.e., 
\[ \Riley_\ij(u^2)=g_\ij(u)g_\ij^*(u) ~\text{ for }~ g_\ij(u)\in\Qbb[u].\]
\end{thm}
Since Riley proved that his polynomial is monic integral, we have the following immediate corollary.
\begin{cor}\label{cor:R=gg}
For a two bridge knot with the top-most or bottom-most generating pair $\arc_i$ and $\arc_j$,
$\Riley_\ij(y) \in  \Zbb[y]$ and hence 
$\Riley_\ij(u^2)=g_\ij(u)g_\ij^*(u)$ with  
a monic integral polynomial $g_\ij(u) \in \Zbb[u]$.
\end{cor}
% \begin{proof} Since	$\Riley_\ij(y)$ is monic integral  due to R. Riley, the $u$-polynomial $g_\ij(u)$ obviously should be monic integral by 
% $\Riley_\ij(u^2)=g_\ij(u)g_\ij^*(u)$. 
% \end{proof}
Corollary \ref{cor:R=gg}  was originally proved in \cite[Theorem 5.1]{jo_symplectic_2020} and Theorem \ref{thm:R=gg}  can be seen as a generalization of the  theorem. 
%\begin{rmk}
Remark that Theorem \ref{thm:R=gg}  holds only for knots,  since Theorem \ref{thm:perfect1-1} is no longer true for links and the proof is also not valid anymore. See the Remark \ref{rmk:upolyQcoeff}.
We can see a counter-example of the Whitehead link in Section \ref{sec:whiteheadlink}.
%\end{rmk}

\begin{rmk}\label{rmk:comparisonRiley}
Let us compare the generalized Riley polynomial $R_\ij(y)$ in this paper with the original Riley's polynomial $\Lambda_{\alpha,\beta}(y)$ in \cite{riley_parabolic_1972} and Jo and Kim's polynomial $P_K(u)$ in \cite{jo_symplectic_2020}. Since $R_\ij(y)$ is defined to be monic,  $  \Lambda_{\alpha,\beta}(y)= (-1)^{\deg( \Lambda_{\alpha,\beta})} R_\ij(y)$. Note $P_K(u)$ of a  knot $K$, has an additional single $u$ factor, therefore, $\P_K(u)= u\Riley_\ij(u^2)$. Compare \cite[Theorem 4.9]{jo_symplectic_2020}.
\end{rmk}

\section{Trace field}
\subsection{Invariant trace fields}

The \emph{trace field} $\trfield$ of a representation $\rho$ is the smallest field containing $\Set{ \tr\circ\rho(g) \given g \in \Gk }$ \cite{maclachlan_arithmetic_2003}. We can also consider the \emph{invariant trace field}  % $\Qbb(\tr^{(2)}\rho)$ 
of a  parabolic representation $\rho$, which is generated by
$\Set{ \tr\circ\rho(g^2) \given g \in \Gk }$. 

It is well-known that  the invariant trace field and the trace field  always coincide for a discrete faithful representation of hyperbolic knot complement \cite[Corollary 4.2.2]{maclachlan_arithmetic_2003}. This is also true for any parabolic representation of a knot group as follows.
\begin{prop}
	For any parabolic representation $\rho$, the invariant trace field is the same as the trace field  $\trfield$.
\end{prop}
\begin{proof}
	The $\rho$-images of $\Gk$ is  generated by $\{\rho(\wg_1),\dots,\rho(\wg_N)\}$ with $\wg_i$'s  Wirtinger generators. 
	By the trace relations, the trace field is generated by the $\rho$-image of the words of at most length 3,  $\{\wg_i,\wg_i\wg_j,\wg_i\wg_j\wg_k\}$.
	By \cite[Lemma 3.5.9]{maclachlan_arithmetic_2003}, the invariant trace field is generated by
	\begin{align*}
		&\tr^2\circ\rho(\wg_i) && \text{ for } 1\leq i \leq N, \\
		&\tr\circ\rho(\wg_i\wg_j)\tr\circ\rho(\wg_i)\tr\circ\rho(\wg_j), && \text{ for } 1\leq i < j \leq N, \\
		&\tr\circ\rho(\wg_i\wg_j\wg_k)\tr\circ\rho(\wg_i)\tr\circ\rho(\wg_j)\tr\circ\rho(\wg_k)  && \text{ for } 1\leq i<j<k \leq N.  
	\end{align*}
	Since  $\tr\circ\rho(\wg_i)=2$, it is the same generating set as $\trfield$.
\end{proof}

One may also study the trace fields of parabolic representations  as a commensurable invariant of $3$-manifolds but we don't consider this direction in this paper. 

\subsection{Riley field, $\u$-field, and trace field}
For a given knot diagram $D$ and a representation $\rho$,  we have the algebraic number $y$-value $y_\ij$ (resp. $u$-value $u_\ij$) for each pair of arcs $\arc_i$ and $\arc_j$. 
Let $\{\yv_\ij\}$ (resp.  $\{\u_\ij\}$) be the set  of the corresponding $\yv$-values (resp. $\u$-values) over all $1\leq i,j\leq N$. Let us call the number fields $\yfield$ (resp. $\ufield$) \emph{Riley field} (resp. \emph{$\u$-field}) generated by $\{y_\ij\}$ (resp. $\{u_\ij\}$), where $y_\ij$ (resp. $u_\ij$) is the corresponding root of Riley polynomial $R_\ij(y)$ (resp.  $u$-polynomial $g_\ij(u)$).

It is obvious that the Riley field $\yfield$ is  a subfield of the trace field $\trfield$, but it seems to be unclear whether it can be defined independently of the choice of knot diagram. However,  we can assure the invariance of a $u$-field as follows. 

\begin{thm}\label{thm:invufield}
	For a parabolic representation $\rho$, the $u$-field  $\ufield$ does not depend on the choice of a knot diagram.
\end{thm}
\begin{proof}
	Let  $\D'$ be the knot diagram connected by a single Reidemeister move from $\D$. Let $(\aa_i)$ and $(\aa'_i)$ be the parabolic quandles of  $\D$ and $\D'$ respectively, and hence $\u_\ij$ and $\u'_\ij$ denote the $\u$-values of them. 
	The first Reidemeister move doesn't change the set of parabolic quandle vectors $\{a_1,a_2,\dots\}$ and hence the set of $\u$-values $\{\u_{11},\u_{12},\dots,\u_\ij,\dots\}$ as well.
	Let us denote $\aa'_*$ the coloring vector of a new arc  by the second or the third Reidemeister moves. Because of the quandle relation of $\aa_*' = \aa_i \pm \lrbar{\aa_i,\aa_j} \aa_j$ for some $i,j$, we can always compute $\aa'_*$  from some arcs $\aa_i$ and $\aa_j$ in the previous diagram $\D$. Look at the newly introduced $\u$-values $\u'_{*k}$ for $k=1,\dots,N$. They are obtained by
	$   \u'_{*k}= \lrbar{\aa'_*,\aa_k}   = \lrbar{\aa_i,\aa_k} \pm \lrbar{\aa_i,\aa_j} \lrbar{\aa_j,\aa_k} = \u_{ik}\pm \u_\ij \u_{jk} \in \ufield.$
	Therefore we have $\Qbb(\{u_\ij'\})\subset \ufield $. Since  the role of $D$ and $D'$ can be exchanged, we conclude that $\Qbb(\{u_\ij'\})= \ufield $.
\end{proof}

Furthermore, we obtain an inclusion relation between the $u$-field and the trace field as follows.

\begin{thm}\label{thm:inclusionFields}
	$$\trfield \subset \ufield$$
\end{thm}
\begin{proof}
	
	Because $\tr\circ\rho(\wg_i) =2  $ and $\tr\circ\rho(\wg_i\wg_j)= 2-(\u_\ij) ^2$, it is enough to show that  $\tr\circ\rho( \wg_i\wg_j\wg_k) \in \ufield$. 
	Without loss of  generality, 	let us  consider $\rho( \wg_1\wg_2\wg_3)$ with simplified the indices. Parabolic quandle map $\M$ is surjective, we have quandle vectors $\aa_i$ such that $\M(\aa_i)=\rho(\wg_i)$ for $i=1,2,3$. As recall the hat-notation in Section \ref{sec:symquan},
	we have $\rho(\wg_i)= I + \aa_i \widehat{\aa_{i}}$ for $i=1,2,3$.
	Then, 
	\begin{align*}
		\rho(\wg_1\wg_2\wg_3) &= (I + \aa_1 \widehat{\aa_{1}})(I + \aa_2 \widehat{\aa_{2}})(I + \aa_3 \widehat{\aa_{3}})\\
		&= 	I+\aa_1 \widehat{\aa_{1}}+\aa_2 \widehat{\aa_{2}}+\aa_3 \widehat{\aa_{3}}
		+\aa_1 \widehat{\aa_{1}}\aa_2 \widehat{\aa_{2}} +\aa_1\widehat{\aa_{1}}\aa_3 \widehat{\aa_{3}}+ \aa_2\widehat{\aa_{2}}\aa_3 \widehat{\aa_{3}} \\
		&~~~~~~~+ \aa_1 \widehat{\aa_{1}}\aa_2 \widehat{\aa_{2}}\aa_3 \widehat{\aa_{3}}.
	\end{align*}
	By the computation of $\tr(\aa_i\widehat{\aa_i})=0$ and
	$\widehat{\aa_i}{\aa_j}=u_\ij=-\tr({\aa_i}\widehat{\aa_j}))$, we have
	$$
	\tr(\rho(\wg_1\wg_2\wg_3))=2-u_{12}^2 - 	u_{13}^2-u_{23}^2 -u_{12}u_{23}u_{13} \in \ufield. \qedhere $$
	
\end{proof}

By the above Theorem \ref{thm:invufield} and Theorem \ref{thm:inclusionFields} , we propose the following conjecture. 

\begin{conj}\label{conj:ufield=trfield}
	$\ufield =\trfield$ 
\end{conj}

Furthermore, from   lots of computer experiments,  we also expect the following stronger statement.
\begin{conj}\label{conj:yfield=ufield}
	$\yfield =\ufield$ 
\end{conj}
% It seems to be much more difficult than Conjecture \ref{conj:ufield=trfield}.
% However, it would be as important as the difficulty.   
Conjecture \ref{conj:yfield=ufield} implies not only Conjecture \ref{conj:ufield=trfield}, but also $\yfield= \trfield$. Hence it follows that the trace field is generated by the trace of words of only length $2$ without worrying the words of length $3$. On the other hand, regardless of the truth of $\yfield=\trfield$, it would be independently interesting whether or not  Riley field $\yfield$ is an invariant, i.e., independent of the choice of a knot diagram.

\begin{rmk}
	We can naturally consider another number fields: $\Qbb(\{y_\cc\})$ and $\Qbb(\{u_\cc\})$ are generated by all $y$-values and $u$-values at each crossing, and these are obviously subfields of $\yfield$ or $\ufield$, respectively. Fortunately, we can find an example of $9_{35}$  such that $\Qbb(\{y_\cc\})$ is a proper subfield of $\trfield$ and, moreover, is not invariant under the choice of knot diagram. On the other hand, we couldn't find an example of $\Qbb({\{u_\cc\}}) \neq \trfield$  for   all knot diagrams in \cite{knotinfo}  up to 12 crossings. For $u$-field case,  we suspect  that $\ufield = \Qbb(\{u_\cc\})$ could be true.
	% , but we have no idea how to approach to answer it.
\end{rmk}

%
% Obviously, Conjecture \ref{conj:yfield=trfield}-(b) implies  	
% \begin{conj}\label{conj:yfield=ufield}
	% 	The $u$-field  $\ufield  $ and Riley field  $\yfield$ are the same.
	% \end{conj}
%

\begin{rmk}
	For two bridge knots, it is easy to see that the trace field is generated by the $\yv$-value at the generating pair, i.e., $\trfield = \Qbb(\yv_\cc)$ where $\cc$ is the top-most (or bottom-most) crossing. Moreover,  in \cite[Proposition 5.5]{jo_symplectic_2020}, it was also shown that  $\u_\cc \in \Qbb(\yv_\cc)$  for  $\yv_\cc=\u_\cc^2$ and hence   Conjecture \ref{conj:yfield=ufield} is true for all two bridge knots.	
\end{rmk}

Although the answer to whether $\ufield=\yfield$ could not be obtained, 
we can obtain an equivalent condition for $\Qbb(\u_\ij)=\Qbb(\yv_\ij)$  which is a kind of `local' statement of Conjecture \ref{conj:yfield=ufield} for each $\ij$ pair.

\begin{thm}\label{thm:equivuiny}
	Consider $\yv$-value   $\yv_\ij$ (resp. $\u$-value  $\u_\ij$) at a pair of arcs $\arc_i$ and $\arc_j$, where  $y_\ij=u_\ij^2$.
	Let the minimal polynomial  of  $\yv_\ij$ (resp. $\u_\ij$) be  $R_0(y)\in \Qbb[y]$ (resp. $g_0(u)\in \Qbb[u]$). 
	Then the followings are equivalent.
	\begin{enumerate}[(i)]
		\item  $\Qbb(\u_\ij)= \Qbb(\yv_\ij)$, or equivalently $\u_\ij \in \Qbb(\yv_\ij)$. 
		\item $R_0(u^2)$ is reducible in $\Qbb[u]$
		\item $g_0(u)$ has an odd-degree term.
	\end{enumerate}
	In particular,   $R_0(u^2)=g_0(u)g_0^*(u)$ if the above equivalent conditions hold and $R_0(u^2)=g_0(u)$ otherwise.
\end{thm}

\begin{proof}
	First, suppose that $\u_\ij\in  \Qbb(\yv_\ij)$. Then $\u_\ij = {f(y_\ij)}/{h(y_\ij)}$ where $$\deg_y f(y) ,  \deg_y h(y) < \deg_y R_0(y).$$ 
	Let $F(u):=u h(u^2)- f(u^2)$.  
	Since
	$$\deg_u F(u) \leq 2 \deg_y R_0(y) -1 < \deg_uR_0(u^2)$$ and 		 $F(\u_\ij)=0$,     $R_0(\u^2)$ is never irreducible in $\Qbb[u]$ and we have (ii).
	Second, suppose $R_0(u^2)$ is reducible. 
	Then we have
	$R_0(\u^2)=p(u)p^*(u)$ for an irreducible $p(u) \in \Qbb[u]$   by \cite[(9.1)]{selmer_irreducibility_1956}.
	Without loss of generality, we can assume $p(u_\ij)=0$. Since two monic irreducible polynomials $p(u)$ and $g_0(u)$  have a common root $\u_\ij$,  $g_0(u)=p(u)$.
		Suppose that $g_0(u)$ has no odd-degree term and  $g_0(u)=q(u^2)$ for some $q(y)\in\Qbb[y]$, then $R_0(y)=(q(y))^2$ and it contradicts  $R_0(y)$ is irreducible. Hence we obtain (iii). 
	Finally, suppose $g_0(u)$ has an odd-degree term. Then $g_0(u)=u f(u^2) +h(u^2)$ with non-zero $f(u)$ and we have $u_\ij= -h(y_\ij) / f(y_\ij) \in \Qbb(y_\ij)$ of (i). 
	It completes the equivalence and
	the final assertions are trivially obtained in the proof.
\end{proof}
In the above theorem,  the minimal polynomials $R_0(y)$ and $g_0(u)$ are irreducible factors of Riley polynomial $R_\ij(y)$ and  $u$-polynomial $g_\ij(u)$, respectively.
Note that the degree of $R_\ij(y)$ and $g_\ij(u)$ should be the same by the definition. The fact of  $\Qbb(\yv_\ij)\neq\Qbb(\u_\ij)$ implies that the Riley polynomial $R_\ij(y)$ at the pair $\ij$ must have multiple roots as follows.

\begin{cor}\label{cor:distinctroot}
	If $\Riley_\ij(y)$ has distinct roots, then   $\Qbb(\yv_\ij)=\Qbb(\u_\ij)$.
\end{cor}

\begin{proof}
	Suppose that $\u_\ij \notin \Qbb(\yv_\ij)$. Then the minimal polynomial $g(u)$ of $\u_\ij$ is an even polynomial by (iii) of Theorem \ref{thm:equivuiny} and hence $g(u)=g^*(u)$.
	Let us put $g(u)=h(u^2)$. Since $g(u)g^*(u) \mid R_\ij(u^2)$ by Proposition \ref{prop:Rileyg(u)}, $(h(y))^2$ should divide $\Riley_\ij(y)$. Therefore $\Riley_\ij(y)$ has a multiple root and it contradicts  the assumption.
\end{proof}

As an immediate corollary, we obtain the answer to Conjecture \ref{conj:yfield=ufield} for two bridge knots.

\begin{prop}[Proposition 5.5 in \cite{jo_symplectic_2020}]
	For two bridge knots, Conjecture \ref{conj:yfield=ufield} is true, i.e.,  
	$\yfield=\trfield=\ufield$. 
\end{prop}
\begin{proof}
	Considering a generating crossing $\cc$ of the two bridge knot, 
	it is clear that $\yfield= \Qbb(\yv_\cc)$ and  $\ufield= \Qbb(\uv_\cc)$ and	
	we have 
	$$
	\Qbb(\yv_\cc)=\yfield\subset\trfield\subset\ufield =\Qbb(\uv_\cc). 
	$$
By Corollary \ref{cor:distinctroot} and  the fact that $\Lambda_{\alpha,\beta}(y)$ has distinct roots due to Riley, we obtain $\Qbb(\yv_\cc)=\Qbb(\u_\cc)$ and it completes the proof.
\end{proof}

\section{Homomorphisms between knot groups}

%In particular, this seems to be a very strong constraint about meridian-preserving automorphism between knot groups. 

Let us consider crossing loops $\alpha$ and $\alpha'$ at the crossings $\cc$ and $\cc'$, as in Figure \ref{fig:crossingloop}, of two knots $K$ and $K'$ respectively. 
Let $\Riley(y)$ and $\Riley'(y)$ denote Riley polynomials at $\cc$ and $\cc'$ of $K$ and $K'$, respectively.

\begin{thm}\label{thm:homomorphism}
	If there is a homomorphism $\phi:\Gk \to \Gk'$ with $\phi(\alpha)=\alpha'$ then any irreducible factor of $\Riley'(y)$  divides $\Riley(y)$. In particular, if $\phi$ is an epimorphism then $\Riley'(y)\,|\,\Riley(y)$, and if $\phi$ is isomorphism then $\Riley'(y) = \Riley(y)$.  	
\end{thm}
\begin{proof}
	We have $\phi^*:\Xpar(K') \to \Xpar(K)$ by $\phi^*(\rho')=\rho:=\rho'\circ\phi$ for $[\rho'] \in \Xpar(K')$.
	Then,	the $\yv$-value	$\y_{\cc'}$ of $K'$  should be  the same as  $y_\cc$  of $K$ 
 since 
	$y_\cc'=2-\tr\circ\rho'(\alpha')=2-\tr\circ\rho(\alpha)= y_\cc$. Therefore, the set of $\{\y_{\cc'}(\rho')\mid \rho' \in \Xpar'\}$ is the subset of $\{\y_\cc(\rho)\mid \rho \in \Xpar\}$ (with ignoring multiplicity). So we conclude that $ \Riley(y)= \prod_{\rho\in \Xpar(K)} (y-y_{\cc}(\rho))$ contains all irreducible factors of $\Riley'(y)$. If $\phi$ is an epimorphism, $\phi^*$ becomes monomorphism. So the roots of $R(y)$ should contain all roots of $R'(y)$ concerning the multiplicity and we conclude that the $R'(y)$ itself should divides $R(y)$. If $\phi$ is an automorphism, $R(y)$ also divides $R'(y)$ and hence $R(y)=R'(y)$.
\end{proof}

\begin{rmk}\label{rmk:linkhomo}
	Theorem \ref{thm:homomorphism}  holds for oriented links. 
	For knot cases, the Riley polynomial is independent of the orientation of knot and hence the theorem is for  unoriented knots. 	 However, for link cases, Riley polynomial depends on the choice of orientation of a link-component since the arc-coloring vector $\aa_i$ is transformed into $\ii \aa_i$ under the change of orientation. This phenomenon will be found in the example of $5_1^2$ in Section \ref{sec:whiteheadlink}.
\end{rmk}

\begin{exam}
	Let us see the Riley polynomials of $7_4$ knots in Section \ref{sec:7_4}. By Theorem \ref{thm:homomorphism}, we can easily prove that there is no homomorphism sending  $\cc_5$ to any other crossing $R_{\cc_i}(y)$   with $i\neq 5$ since $R_{\cc_5}(y)$  is not divided by any factor of $R_{\cc_i}(y)$ and vice versa.
\end{exam}

If $\Xpar$ is decomposed into several irreducible $\Qbb$-subvarieties, i.e.,  $$\Xpar=\Xpar^1 \cup \Xpar^2 \cup \cdots\cup \Xpar^n,$$ then $\Xpar^i \cap \Xpar^j=\varnothing$ for all $i\neq j$  because of the assumption that $\Xpar$ is zero-dimensional. Therefore the Riley polynomial $\Riley_\cc(y)$ is factored  as  $\Riley_\cc(y)=\Riley^1_\cc(y)\Riley^2_\cc(y)\cdots\Riley_\cc^n(y) \in \Qbb[y]$ by definition. Moreover, we can think of a kind of `partial' Riley polynomial $\Riley_\cc^i(y)$ for a given component $\Xpar^i$, i.e. 
\begin{equation*}
	\Riley^i_\cc(y):= \prod_{ \rho\in \Xpar^i } (y- y_\cc(\rho) ). 
\end{equation*}
Remark that each $\Riley_\cc^i(y)$  may not be irreducible although $\Xpar^i$ is  irreducible. See the example for $R_{\cc_5}(y)$ of $7_{4}$ in Section \ref{sec:7_4}.

The above Theorem \ref{thm:homomorphism} can be  elaborated for each irreducible component $\Xpar^i$  by the same proof.  
\begin{thm}\label{thm:homomorfactor}
	Let $\phi^*:\Xpar(K')\to\Xpar(K)$ with $\Xpar(K)=\cup_i \Xpar^i(K)$ and $\Xpar(K')=\cup_j \Xpar^j(K')$. Let $\phi^* (\Xpar^i(K')) \subset \Xpar^j(K)$.
	Under the same condition of Theorem \ref{thm:homomorphism}, any irreducible factor of the Riley polynomial 
	${\Riley^j}(y)$ for $\Xpar^j(K')$
	divides $\Riley^i(y)$ for $\Xpar^i(K)$. In particular, if $\phi$ is epimorphism then ${\Riley^j}(y)\,|\,\Riley^i(y)$, and if $\phi$ is isomorphism then ${\Riley^j}(y) = \Riley^i(y)$.  	
\end{thm}
We use this theorem to analyze the structure of $\Xpar(8_{18})$ in Section \ref{sec:8_18str}.

\section{Computations}\label{sec:examplecomputations}
We present several computation examples. The brief procedure is as follows. Firstly, we consider a knot diagram\footnote{We will use the knot diagrams in Rolfsen table \cite{rolfsen_knots_1990}.} with the indexed arcs and crossings.  Then we create  the system of equation for the normalization in Theorem \ref{sec:normalization} and solve them.  
Then we obtain $\Xpar$ the set of parabolic representations and their arc-colorings $(\aa_1,\dots,\aa_N) \in \QQ = \Qu \cup \Qv $. As  solving the equation, we can also check the sign-types and obtain the obstruction classes by total-sign. To obtain $u$-polynomials and Riley polynomials, we compute the $u$-value and $y$-value at each crossing by finding the polynomial whose roots are $u$-values and $y$-values, respectively.
Note that from the relation of $R(u^2)=g(u)g^*(u)$, we obtain an $u$-polynomial just by factoring $R(u^2)$ with up to the choice of $g(u)$ and $g^*(u)$  for each irreducible factor, which depends on the sign-type.  For the canonical  $u$-polynomials in the sense of Remark \ref{rmk:canonicalsigntype}, we need to change the sign-type for each crossing by Theorem \ref{thm:signchaging}. 

Finally,
we compute complex volumes and cusp-shapes of $\rho$ by $z$-variables (resp.  $w$-variables) of the knot diagram as in \cite[Section 6]{kim_octahedral_2018}, where these variables can be obtained from the parabolic quandles computed above by the formula of  \cite[Theorem 3.2]{cho_quandle_2018} (resp. \cite[Section 2.3]{cho_optimistic_2016}).
One can verify these numerical results are consistent with the computations by SnapPy \cite{SnapPy} except for the opposite sign of cusp shape since the orientation convention of the boundary torus might be different.

%
%We compute the complex volumes of $\rho$ by the formula of \cite{cho_optimistic_2014} (resp.  \cite{cho_optimistic_2016-1}) and the cusp shapes by the formula  \cite[Theorem 5.14]{kim_octahedral_2018} (resp.  \cite[Theorem 5.1]{kim_octahedral_2019}). 

\subsection{$4_1$ knot}\label{sec:4_1}
Let us consider a knot diagram and number  each arc and crossing respectively as in Figure \ref{fig:4_1fig}. We have four arc-colorings $\aa_1,\aa_2,\aa_3,\aa_4$ and quandle relations at each crossing $\cc_1,\cc_2,\cc_3,\cc_4$. 

\begin{figure}[!ht]
	\begin{center}
		\includegraphics{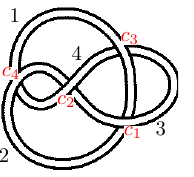}
	\end{center}
	\caption{Labeling a knot diagram of $4_1$ knot.}
	\label{fig:4_1fig}
\end{figure}

\subsubsection{Quandle equations for $\Qu$}
Let us begin with  $a_1=\qvec{1}{0}$ and $a_3=\qvec{0}{u}$. Then each arc-coloring is computed successively along $ \cc_1, \cc_4$  as follows,
\begin{equation}\label{eq:4_1coloring}\aligned
	a_2&\overset{\cc_1}{=}a_1-\lrbar{ a_1, a_3 } a_3=\begin{pmatrix} 1 \\ -u^2 \end{pmatrix},\\ 
	a_4&\overset{\cc_4}{=}a_1-\lrbar {a_1, a_2} a_2=\begin{pmatrix} 1+u^2 \\ -u^4 \end{pmatrix}.
	\endaligned
\end{equation}
Note that there are two relations at $\cc_2$ and $\cc_3$ that have not been used yet. 
We have all arc-colorings $\aa_i$ which should satisfy the equation at $c_2$ as well,
$$a_3 - (a_2 + \lrbar {a_2, a_4 } a_4) = \qvec{h_1(u)}{h_2(u)}=0.$$
By polynomial GCD of $h_1$ and $h_2$, we have  
\begin{equation}\label{eq:4_1 u}
	g(\u):=u^2+u+1.
\end{equation}
Note that every root of $g(u)$ corresponds to each conjugacy class of parabolic representations.  

\subsubsection{Obstruction classes}
The final remaining equation at $\cc_3$ should be satisfied up to sign,
$$\pm a_4 =a_3- \lrbar {a_3, a_1} a_1.$$
We can check that it holds with only the minus sign, not the plus sign and hence every parabolic representation of $4_1$ knot has the negative obstruction class.
Remark that  the $g(u)$ of (\ref{eq:4_1 u}) is the $\u$-polynomial at $\cc_1$ with sign-type $(1,1,-1,1)$.% $\ee_7=\pm1$

\subsubsection{Quandle equations for $\Qv$}
As putting  $a_1=\qvec{1}{0}$ and $a_3=\qvec{v}{0}$, and solving the equations for $\Qv$ in a similar way to $Q_u$, one can check that there is  only a solution of abelian representation. Note that all $u$-values (and $y$-values as well) are zero for abelian representations. 

\subsubsection{Riley polynomials}

By Lemma \ref{lem:lambdaalpha}, we obtain $y$-values from arc-colorings as follows. 
$$
\begin{aligned}
	\yv_{c_1}=\yv_{c_3}=&~ u^2,\\
	\yv_{c_2}=\yv_{c_4}=&~ u^4.  \\
\end{aligned}
$$
Note that this $y$-values do not depend on the choice of  sign-type.
By the bijection between the roots of $g(u)$ and representations in $\Qu$ by Theorem \ref{thm:perfect1-1}, we compute Riley-polynomials by the following definition,
$$\Riley_{c_i}(y)=\prod\limits_{\rho} (y-\yv_{c_i})=\prod\limits_{\Qu} (y-\yv_{c_i})\prod\limits_{\Qv} (y-\yv_{c_i}),$$
%by the definition of trace polynomial in (\ref{def:tracepoly}).
where $\Qv$ is an empty set.
Then, we obtain
\begin{align*}
	\Riley_{c_1}=\Riley_{c_2}=\Riley_{c_3}=\Riley_{c_4}=&~ 1 + y + y^2.
\end{align*}

\subsubsection{$u$-polynomials}
%From the realtion of $R(u^2)=g(u)g^*(u)$, we obtain $u$-polybomial just by factoring $R(u^2)$, up to the choice of $g(u)$ and $g^*(u)$  which depends on the sign-type.  We are going to compute $u$-polynomials of the canonical sign-type in the sense of Remark \ref{rmk:canonicalsigntype}.
First, we compute $u$-values from the arc-colorings of (\ref{eq:4_1coloring}),% by Definition \ref{def:uvalue},
$$
\begin{aligned}
	\u_{c_1}=&~ u, &
	\u_{c_2}=&~   u^2 ,  \\
	\u_{c_3}=&~ -u, &
	\u_{c_4}=&~- u^2 .
\end{aligned}
$$
Recall that these $\u$-values  has the sign-type of (1,1,-1,1). So  we change the $u$-values to be with the canonical sign-types at each crossing by using the formula of Theorem \ref{thm:signchaging}. Then we have  
$$
\begin{aligned}
	\u_{c_1}=&~ - u, &
	\u_{c_2}=&~   u^2 ,  \\
	\u_{c_3}=&~ -u, &
	\u_{c_4}=&~ u^2 .
\end{aligned}
$$
By $g(u)=u^2+u+1=0$, we also derive canonical $u$-polynomials as follows,
\begin{align*}
	g_{c_1}=g_{c_3}=&~ 1 -  u + u^2 \\
	g_{c_2}= g_{c_4}= &~ 1 +  u + u^2.
\end{align*}

\subsubsection{Complex volumes and Cusp shapes}
The complex volumes and cusp shapes are as follows.
\begin{figure}[H]
	$$\begin{array}{|c|c|c|}
		\hline
		\text{Solution to } g(u) & \ii (\vol + \ii \cs)  & \text{Cusp shape}\\ 
		\hline
		\begin{aligned}
			u ~&=&\!\!\!-0.5&- 0.86603 \ii \rule{0em}{1.1em} \\
			u ~&=&\!\!\! -0.5&+ 0.86603 \ii \\
		\end{aligned} &
		\begin{aligned}
			&+ 2.02988 \ii  \rule{0em}{1.1em}\\
			&- 2.02988 \ii\\
		\end{aligned} &
		\begin{aligned}
			&-3.46410 \ii  \rule{0em}{1.1em}\\
			&+3.46410 \ii\\
		\end{aligned}\\
		\hline
	\end{array}
	$$
	%	\caption{Complex volumes and Cusp shapes for all $\rho\in\Xpar(4_1)$.}
\end{figure}

\subsection{$7_4$ knot}\label{sec:7_4}
Let us label each arc and each crossing of $7_4$ knot diagram as in Figure \ref{fig:7_4fig}. We have $7$ arc-colorings $\aa_1,\dots,\aa_7$ and quandle relations at each crossing $\cc_1,\dots,\cc_7$. 

\begin{figure}[h!]
	\begin{center}
		\includegraphics{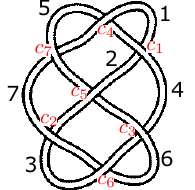}
	\end{center}
	\caption{Labeling a knot diagram of $7_4$ knot.}\label{fig:7_4fig}
\end{figure}

\subsubsection{Quandle equations for $\Qu$}
Let us begin with  $a_1=\qvec{1}{0}$ and $a_4=\qvec{0}{u}$. Then,  each arc-coloring is computed successively along $\cc_1, \cc_4, \cc_5, \cc_3, \cc_6$  as follows,
\begin{equation}\label{eq:7_4coloring}\aligned
	a_2&\overset{\cc_1}{=}a_1+\lrbar{ a_1, a_4 } a_4=\begin{pmatrix} 1 \\ u^2 \end{pmatrix},\\ 
	a_5&\overset{\cc_4}{=}a_4+\lrbar {a_4, a_1} a_1=\begin{pmatrix} -u \\ u \end{pmatrix},\\ a_6&\overset{\cc_5}{=}a_5+\lrbar {a_5, a_2} a_2=\begin{pmatrix} -u^3-2 u \\ -u^5-u^3+u \end{pmatrix},\\
	a_3&\overset{\cc_3}{=}a_4-\lrbar {a_4, a_6} a_6=\begin{pmatrix} u^3 (u^2+2)^2 \\ u^9+3u^7+u^5-2u^3+u \end{pmatrix},\\
	a_7&\overset{\cc_6}{=}a_6+\lrbar {a_6, a_3} a_3=\begin{pmatrix} -u^{11}-6 u^9-12 u^7-8 u^5-u^3-2 u \\ -u^{13}-5 u^{11}-7 u^9+2 u^5-3 u^3+u \end{pmatrix}.\\
	\endaligned
\end{equation}
Note that there are two relations at $\cc_2$ and $\cc_7$ that have not been used yet. 
We have all arc-colorings $\aa_i$ which should satisfy the equation at $c_2$ as well,
$$a_3 - (a_2 + \lrbar {a_2, a_7 } a_7) = \qvec{h_1(u)}{h_2(u)}=0.$$
By polynomial GCD of $h_1$ and $h_2$, we have  
\begin{equation}\label{eq:7_4 u}
	g(\u):=(u^3+2u-1)(u^4+u^3+2u^2+2u+1)=0.
\end{equation}
Note that every root of $g(u)$ corresponds to each conjugacy class of parabolic representations.  

\subsubsection{Obstruction classes}
The final remaining equation at $\cc_7$ should be satisfied up to sign,
$$\pm a_1 =a_7+ \lrbar {a_7, a_5} a_5.$$
We can check that it holds with  only the minus sign, not the plus sign and hence every parabolic representation of $7_4$ knot has the negative obstruction class.
Remark that  the $g(u)$ of (\ref{eq:7_4 u}) is the $\u$-polynomial at $\cc_1$ with sign-type $(1,1,1,1,1,1,-1)$.% $\ee_7=\pm1$

% \subsubsection{Quandle equations for $\Qv$}
% As putting  $a_1=\qvec{1}{0}$ and $a_4=\qvec{v}{0}$ and solving the equations for $\Qv$ in a similar way to $Q_u$, one can check that there is only solution of abelian representation. 

\subsubsection{Arc-coloring vectors}
There are two irreducible components of $g(\u)=g_1(\u)g_2(\u)$
with $g_1(u)=u^3+2u-1$ and $g_2(u)=u^4+u^3+2u^2+2u+1$ and the arc-colorings of each component can be written in the number field $\Qbb(u_i)$ with a root $u_i$ of $g_i$ as follows.
$$
\def\arraystretch{1.3}
\begin{array}{c|c|c}
	\hline
	\text{arc} & g_1(u)=u^3+2u-1 & g_2(u)=u^4+u^3+2u^2+2u+1 \\
	\hline
	a_1 & \qvec{ 1}{ 0} & \qvec{ 1}{ 0}  \rule{0em}{1.7em} \\ [1.5ex] 
	\hline
	a_2 & \qvec{ 1}{ u^2 } &\qvec{ 1}{ u^2 } \rule{0em}{1.7em} \\ [1.5ex]
	\hline
	a_3 & \qvec{ u}{ u^2-2u+1 } & \qvec{2u^3+u^2+3u+3}{ -u^3-u-1 } \rule{0em}{1.7em} \\ [1.5ex]
	\hline
	a_4 & \qvec{ 0}{ u } & \qvec{ 0}{ u }\rule{0em}{1.7em} \\ [1.5ex]
	\hline
	a_5 & \qvec{ -u}{ u } & \qvec{ -u}{ u } \rule{0em}{1.7em} \\ [1.5ex]
	\hline
	a_6 & \qvec{ -1}{ -u^2-u+1 } & \qvec{ -u^3-2u}{ -1 } \rule{0em}{1.7em} \\ [1.5ex]
	\hline
	a_7 & \qvec{ -u^2-1}{ u^2  } & \qvec{ -u^2-1}{u^2  } \rule{0em}{1.7em} \\ [1.5ex]
	\hline
\end{array}
$$

\subsubsection{Quandle equations for $\Qv$}

In a similar way to the case of $4_1$, one can check that there is no representation in $\Qv$ except  abelian representations.

\subsubsection{Riley polynomials}

We obtain $y$-values as follows. 
$$
\begin{aligned}
	\yv_{c_1}=\yv_{c_4}=\yv_{c_7}=&~ u^2,\\
	\yv_{c_2}=\yv_{c_3}=\yv_{c_6}=&~ u^6+u^5+3 u^4+2 u^3+2 u^2+u-1,\\
	\yv_{c_5}=&~ u^6+2 u^4+u^2.  \\
\end{aligned}
$$
%
%Note that this $y$-values does not depend on the choice of a sign-type.
%By the bijection between the roots of $g(u)$ and representations in Theorem \ref{thm:perfect1-1}, We compute Riley-polynomials as follows,
%$$\Riley_{c_i}(y)=\prod\limits_{\rho} (y-\yv(c_i))=\prod\limits_{g(u)=0} (y-\yv(c_i)).$$
%%by the definition of trace polynomial in (\ref{def:tracepoly}).
Since a Riley polynomial $R_{\cc_i}(y)$ at $\cc_i$ is the polynomial whose roots are the above $y$-value $y_{\cc_i}$, each $R_{\cc_i}(y)$ is determined as follows,
\begin{align*}
	\Riley_{c_1}=\Riley_{c_4}=\Riley_{c_7}=\Riley_{c_2}=\Riley_{c_3}=\Riley_{c_6} =&~ (y^3+4y^2+4y-1)(y^4+3y^3+2y^2+1)\\
	\Riley_{c_5} =&~ (y^3+y^2+13y-4)(y^2+y+1)^2.
\end{align*}

\subsubsection{$u$-polynomials}
First, we compute $u$-values from the arc-colorings of (\ref{eq:7_4coloring}),% by Definition \ref{def:uvalue},
$$
\begin{aligned}
	\u_{c_1}=&~ u,\\
	\u_{c_2}=\u_{c_3}=&~  2 u^2 + u^4,  \\
	\u_{c_4}=\u_{c_7}=&~ -u,\\
	\u_{c_5}=&~-u - u^3, \\
	\u_{c_6}=&~ -2 u^2 - u^4.
\end{aligned}
$$
Recall that these $\u$-values have the sign-type of (1,1,1,1,1,1,-1). So  we change the $u$-values to be with the canonical sign-types at each crossing by using the formula of Theorem \ref{thm:signchaging}. Then we have  
$$
\begin{aligned}
	\u_{c_1}=\u_{c_4}=u_{c_7}=&~ -u,\\
	\u_{c_2}=\u_{c_3}=\u_{c_6}=&~ -2 u^2 - u^4 ,\\
	\u_{c_5}=&~-u - u^3.  \\
\end{aligned}
$$
After straightforward computation, we also derive canonical $u$-polynomials as follows,
\begin{align*}
	g_{c_1}=g_{c_4}=g_{c_7}=g_{c_2}=g_{c_3}=g_{c_6} =&~ (1 + 2 u + u^3) (1 - 2 u + 2 u^2 - u^3 + u^4)\\
	g_{c_5} =&~  (2 + 5 u + 3 u^2 + u^3)(1 - u + u^2)^2.
\end{align*}
Note that these canonical $\u$-polynomials respect the diagrammatic symmetry as in the Riley polynomials, whereas the $u$-polynomials with sign-type $(1,1,1,1,1,1,-1)$ does not.

\subsubsection{Complex volumes and Cusp shapes}
The complex volumes  and  cusp shapes are  as follows.%in Figure \ref{table:7_4}.
\begin{figure}[H]
	$$\begin{array}{|c|r|r|}
		\hline
		\text{Solution to } g(u) & \ii (\vol + \ii \cs)\rule{1.5em}{0em} & \text{Cusp shape}\rule{2em}{0em}\\ 
		\thickhline
		\begin{aligned}
			u ~&=&\hspace{-2ex} -0.22670 &- 1.46771 \ii \rule{0em}{1.1em} \\
			u ~&=&\hspace{-2ex} -0.22670 &+ 1.46771 \ii \\
			u ~&=&\hspace{-2ex} 0.45340 & \\
		\end{aligned} &
		\begin{aligned}
			9.44074&- 5.13794 \ii  \rule{0em}{1.1em}\\
			9.44074&+ 5.13794 \ii\\
			-0.78720 &\\
		\end{aligned} &
		\begin{aligned}
			-0.68207 &+ 3.20902 \ii  \rule{0em}{1.1em}\\
			-0.68207 &- 3.20902 \ii\\
			-12.63587 &\\
		\end{aligned}\\
		\hline
		\begin{aligned}
			u ~&=&\hspace{-2ex} -0.62174 &- 0.44060 \ii \rule{0em}{1.1em} \\
			u ~&=&\hspace{-2ex} -0.62174 &+ 0.44060 \ii \\
			u ~&=&\hspace{-2ex} 0.12174&- 1.30662 \ii \\
			u ~&=&\hspace{-2ex} 0.12174&+ 1.30662 \ii \\
		\end{aligned} &
		\begin{aligned}
			3.28987&- 2.02988 \ii  \rule{0em}{1.1em}\\
			3.28987&+ 2.02988 \ii\\
			3.28987&+ 2.02988 \ii\\
			3.28987&- 2.02988 \ii\\
		\end{aligned}&
		\begin{aligned}
			-4 &+ 3.46410 \ii  \rule{0em}{1.1em}\\
			-4 &- 3.46410 \ii\\
			-4 &- 3.46410 \ii\\
			-4 &+ 3.46410 \ii\\
		\end{aligned}\\
		\hline
	\end{array}
	$$
	%\caption{Complex volumes and Cusp shapes for all $\rho\in\Xpar(7_4)$.}
	\label{table:7_4}
\end{figure}

Let us recall that Chern-Simons invariant  is defined by modulo $\pi^2$. The CURVE project \cite{curveproject} shows that the Chern-Simons invariant is $0.42887...$ which is consistent with  $9.44074...$ in our table since $0.42887=-(9.44074-\pi^2)$. 
Remark that the Chern-Simons invariants in the CURVE project are computed only modulo $\pi^2/6$, but our computation is  modulo $\pi^2$. For example, the complex volume ($\vol+\ii \cs$) of geometric representation is $2.02988- 3.28987\ii$  but the CURVE data shows $2.02988...$ which  agrees modulo $\frac{\pi^2}{6} \ii$.

%
%\subsubsection{Discussion}
%The $7_4$ knot is the first example for two irreducible component of $\Xpar$ as a $\Qbb$-variety. The constant term of Riley polynomial at $\cc_5$ is not $\pm1$ 
%
\begin{rmk}\label{rmk:7_4}
The Riley polynomial at $\cc_5$ has the Riley polynomial of $4_1$ as a factor. We can see that the complex volumes and cusp shapes  corresponding to such a factor are closely related to the complex volume and cusp shape of $4_1$  as follows.
$$\cv(\rho) \equiv \cv(4_1) \mod{\frac{\pi^2}{6}},$$    
$$\text{cusp}(\rho)\equiv \text{cusp}(4_1) \mod{\Zbb},$$
where $\rho$ is a representation  of the factor of $y^2+y+1$ which is the same as  the Riley polynomial of $4_1$.    
This phenomenon  with factor-sharing Riley polynomials is observed very often and the geometric interpretation seems worthy of further study.         
\end{rmk}

\subsection{ $9_{29}$ knot}\label{sec:9_29}
Let us label each arc and each crossing of $9_{29}$ knot diagram as in Figure \ref{fig:9_29fig}. We have $9$ arc-colorings $\aa_1,\dots,\aa_9$ and quandle relations at each crossing $\cc_1,\dots,\cc_9$.

\begin{figure}[h!]
	\begin{center}
		\includegraphics{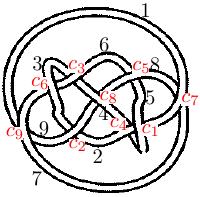}	\end{center}
	\caption{Labeling the diagram of $9_{29}$ knot.}\label{fig:9_29fig}
\end{figure}

\subsubsection{Quandle equations for $\Qu$}\label{sec:9_29Qu}
Let us begin with  $a_1=\qvec{1}{0}$, $a_5=\qvec{0}{u}$ and $a_8=\qvec{a}{b}$ with indeterminate variables $u$, $a$, $b$.
Then, each arc-coloring is computed successively along $\cc_1, \cc_5, \cc_4, \cc_3, \cc_7,\cc_9$  as follows,

\begin{equation}\label{eq:9_29coloring}\aligned
	a_2&\overset{\cc_1}{=}a_1+\lrbar{ a_1, a_5 } a_5=\begin{pmatrix} 1 \\ u^2 \end{pmatrix},\\ 
	a_6&\overset{\cc_5}{=}a_5-\lrbar {a_5, a_8} a_8=\begin{pmatrix} a^2 u\\ a b u+u \end{pmatrix},\\ a_4&\overset{\cc_4}{=}a_5-\lrbar {a_5, a_2} a_2=\begin{pmatrix} u\\ u^3+u \end{pmatrix},\\
	a_3&\overset{\cc_3}{=}a_4-\lrbar {a_4, a_6} a_6=\begin{pmatrix} a^4 u^5+a^4 u^3-a^3 b u^3-a^2 u^3+u \\ a^3 b u^5+a^3 b u^3-a^2 b^2 u^3+a^2 u^5+a^2 u^3-2 a b u^3+u \end{pmatrix},\\
	a_7&\overset{\cc_7}{=}a_8+\lrbar {a_8, a_1} a_1=\begin{pmatrix} a-b \\ b \end{pmatrix},\\
	a_9&\overset{\cc_9}{=}a_1+\lrbar {a_1, a_7} a_7=\begin{pmatrix} a b-b^2+1 \\ b^2 \end{pmatrix}.\\
	\endaligned
\end{equation}
Note that there are three relations at $\cc_2$, $\cc_6$, and $\cc_8$ that have not been used yet. 
We put two of them as follows,
$$\aligned
a_3 - (a_2 - \lrbar {a_2, a_9 } a_9) &= \qvec{h_1(u,a,b)}{h_2(u,a,b)}=0,\\
a_7 - (a_6 + \lrbar {a_6, a_3 } a_3) &= \qvec{h_3(u,a,b)}{h_4(u,a,b)}=0.\endaligned
$$

\subsubsection{Gro\"ebner basis}
We obtained four equations $h_1,h_2,h_3$ and $h_4$ $\in \Qbb[u,a,b]$ for $\Qu$. By using any mathematics software, we have a Gro\"ebner basis as the following form,
\begin{equation}\label{eq:9_29}
	\{ g(u),a-f(u),b-h(u) \},	
\end{equation}
where $g(u)=g_1(u)g_2(u)g_3(u)$ and
\begin{align*}
	g_1(u)&=u-1, \\	
	g_2(u)&=u^{10}-u^9+4 u^8-7 u^7+8 u^6-14 u^5+11 u^4-10 u^3+7 u^2-2 u-1,\\
	g_3(u)&=u^{16}+3 u^{15}+10 u^{14}+22 u^{13}+43 u^{12}+73 u^{11}+101 u^{10}+129 u^9+136 u^8\\
	&\,\,\,\,\, +124 u^7+100 u^6+60 u^5+32 u^4+12 u^3+2 u^2+2 u+1,
\end{align*}
and $f(u)$ and $g(u)$ are large polynomials of degree $26$ in $\Qbb[u]$. (For example, the leading coefficient and the constant term of $f(u)$ are $15822968441895089$ and $80779430690551789$.)
Now we can see that the decomposition of $\Xpar$ is determined by factoring $g(u)$ and consists of three irreducible components $\Xpar^i$ represented by $g_i(u)$ for $i=1,2,3$.

\subsubsection{Quandle equations for $\Qv$}
For $\Qv$, we also have the system of equations
by the similar procedure with the initials $a_1=\qvec{1}{0}$, $a_5=\qvec{v}{0}$ and $a_8=\qvec{a}{b}$,  and compute the Gro\"ebner basis as well. We can check that  there are only abelian representations in $\Qv$.

\subsubsection{Obstruction classes}
The last unused equation at $\cc_8$ should be satisfied up to sign, which produces the obstruction class as follows.
$$\pm a_9 =a_8- \lrbar {a_8, a_4} a_4.$$

We can test  which sign holds for each $\Xpar^i$ by putting $u,a,b$ of (\ref{eq:9_29}). The single  representation in $\Xpar^1$ is of the only positive obstruction and the other $26$ representations in $\Xpar^2$ and $\Xpar^3$ are of negative obstruction.
Remark that  the $g(u)$ of (\ref{eq:9_29}) is the $\u$-polynomial at $\cc_1$ with sign-type $(1, 1, 1, 1, 1, 1, 1, -1, 1)$.

\subsubsection{Riley polynomials}

By Lemma \ref{lem:lambdaalpha}, we obtain $y$-values as follows. 

$$
\begin{aligned}
	\yv_{c_1}=\yv_{c_4}=&~ u^2,\\
	\yv_{c_2}=&~ (a b u^2-b^2 u^2- b^2+u^2)^2,\\
	\yv_{c_3}=\yv_{c_6}=&~ u^4(a^2 u^2+a^2-a b-1)^2,\\
	\yv_{c_5}=&~ a^2 u^2  \\
	\yv_{c_7}=\yv_{c_9}=&~  b^2,\\
	\yv_{c_8}=&~ u^2(a u^2 +a -b)^2.  \\
\end{aligned} 
$$

Since a Riley polynomial $R_{\cc_i}(y)$ at $\cc_i$ is the polynomial whose roots are the above $y$-value $y_{\cc_i}$, each $R_{\cc_i}(y)$ is determined as follows,
\begin{align*}
	\Riley_{c_1}=\Riley_{c_3}=\Riley_{c_4}=\Riley_{c_6}=&~ (y-1)(y^{10}+7y^9+18y^8+9y^7-50y^6-110y^5-83y^4-18y^3-13y^2\\&\phantom{x}-18y+1)(y^{16}+11y^{15}+54y^{14}+140y^{13}+155y^{12}-143y^{11}-689y^{10}\\&\phantom{x}-741y^9+198y^8+1160y^7+926y^6+54y^5-240y^4-56y^3+20y^2+1)\\
	\Riley_{c_2}=\Riley_{c_5} =&~ \frac {1}{16}(4y-1)(4y^{10}-25y^9+100y^8-198y^7+245y^6-161y^5+29y^4+28y^3\\&\phantom{x}-14y^2-y+1)(y^{16}-9y^{15}+66y^{14}-324y^{13}+1319y^{12}-4279y^{11}\\&\phantom{x}+11631y^{10}-26121y^9+49082y^8-76624y^7+99630y^6-107334y^5\\&\phantom{x}+95488y^4-68368y^3+37044y^2-13044y+2209)\\
	\Riley_{c_7}=\Riley_{c_9} =&~ (y-1)(y^{10}-8y^9+26y^8-17y^7-79y^6+123y^5+236y^4-598y^3\\&\phantom{x}+349y^2-65y+16)(y^8-7y^7+19y^6-22y^5+3y^4+14y^3-6y^2\\&\phantom{x}-4y+1)^2\\
	\Riley_{c_8} =&~ y(y^{10}-3y^9+43y^8-126y^7+731y^6-978y^5+2754y^4+1048y^3-639y^2\\&\phantom{x}-244y+64)(y^8-3y^7+7y^6-10y^5+11y^4-10y^3+6y^2-4y+1)^2.
\end{align*}
Note that the factorization of Riley polynomial exactly corresponds to the decomposition of $\Xpar$ and the Riley polynomials for $\Xpar^3$ has multiple roots at $\cc_7$, $\cc_8$, $\cc_9$.

\subsubsection{$u$-polynomials}
First, we compute $u$-values from the arc-colorings of (\ref{eq:9_29coloring}),% by Definition \ref{def:uvalue},
$$
\begin{aligned}
	\u_{c_1}=&~ u,\\
	\u_{c_2}=&~ -a b u^2+b^2 u^2+b^2-u^2,  \\
	\u_{c_3}=&~ u^2 - a^2 u^2 + a b u^2 - a^2 u^4,\\
	\u_{c_4}=&-u, \\
	\u_{c_5}=&~-a u, \\
	\u_{c_6}=&-u^2 + a^2 u^2 - a b u^2 + a^2 u^4 ,\\
	\u_{c_7}=&~ -b,\\
	\u_{c_8}=&~ a u - b u + a u^3,\\
	\u_{c_9}=&~ b.
\end{aligned}
$$
Recall that these $\u$-values   has the sign-type of  $(1, 1, 1, 1, 1, 1, 1, -1, 1)$. So  we change the $u$-values to be with the canonical sign-types at each crossing by using the formula of Theorem \ref{thm:signchaging}. Be careful that the transformation formulas into the canonical sign-type  are different along the obstruction classes.
After straightforward computation, we derive canonical $u$-polynomials as follows,
\begin{align*}
	g_{\cc_1}=g_{\cc_3}=&~
	(-1 + u) (-1 + 2 u + 7 u^2 + 10 u^3 + 11 u^4 + 14 u^5 + 8 u^6 + 
   7 u^7 + 4 u^8 + u^9 + u^{10})\\ &~~~~ (1 - 2 u + 2 u^2 - 12 u^3 + 32 u^4 - 
   60 u^5 + 100 u^6 - 124 u^7 + 136 u^8 - 129 u^9 + 101 u^{10}\\&~~~~ - 
   73 u^{11} + 43 u^{12} - 22 u^{13} + 10 u^{14} - 3 u^{15} + u^{16})\\
   g_{\cc_4}=g_{\cc_6}=&~
   (1 + u) (-1 + 2 u + 7 u^2 + 10 u^3 + 11 u^4 + 14 u^5 + 8 u^6 + 
   7 u^7 + 4 u^8 + u^9 + u^{10})\\
   &~~~~ (1 - 2 u + 2 u^2 - 12 u^3 + 32 u^4 - 
   60 u^5 + 100 u^6 - 124 u^7 + 136 u^8 - 129 u^9 + 101 u^{10} \\ &~~~~- 
   73 u^{11} + 43 u^{12} - 22 u^{13} + 10 u^{14} - 3 u^{15} + u^{16})\\
   g_{\cc_2}=&\frac{1}{4} (-1 + 2 u) (1 - u - 2 u^3 - 5 u^4 + 7 u^5 + 9 u^6 - 8 u^7 - 
   4 u^8 + 3 u^9 + 2 u^{10})\\
   &~~~~ (47 + 136 u + 58 u^2 - 110 u^3 + 40 u^4 + 
   226 u^5 + 6 u^6 - 146 u^7 + 40 u^8 + 101 u^9 - 19 u^{10}\\ &~~~~ - 33 u^{11} + 
   15 u^{12} + 10 u^{13} - 4 u^{14} - u^{15} + u^{16})\\
   g_{\cc_5}=&\frac{1}{4} (1 + 2 u) (1 - u - 2 u^3 - 5 u^4 + 7 u^5 + 9 u^6 - 8 u^7 - 
   4 u^8 + 3 u^9 + 2 u^{10})\\
   &~~~~ (47 + 136 u + 58 u^2 - 110 u^3 + 40 u^4 + 
   226 u^5 + 6 u^6 - 146 u^7 + 40 u^8 + 101 u^9 - 19 u^{10} \\ &~~~~- 33 u^{11} + 
   15 u^{12} + 10 u^{13} - 4 u^{14} - u^{15} + u^{16})\\
   g_{\cc_7}=&~ (1 + u) (-1 + 2 u + 3 u^4 - 2 u^5 - 3 u^6 + u^7 + u^8)^2 (-4 - 11 u - 
   7 u^2 + 18 u^3 + 12 u^4 - 3 u^5 \\&~~~~+ 5 u^6 + u^7 - 4 u^8 + u^{10})\\
   g_{\cc_9}=&~(-1 + u) (-1 + 2 u + 3 u^4 - 2 u^5 - 3 u^6 + u^7 + u^8)^2 (-4 - 
   11 u - 7 u^2 + 18 u^3 + 12 u^4 \\&~~~~- 3 u^5 + 5 u^6 + u^7 - 4 u^8 + 
   u^{10})\\
   g_{\cc_8}=&~u (-1 - 2 u + 2 u^3 + u^4 - 2 u^5 - u^6 + u^7 + u^8)^2 (8 - 30 u + 
   41 u^2 - 60 u^3 + 80 u^4 - 30 u^5\\&~~~~ - 7 u^6 + 8 u^7 + 3 u^8 - 3 u^9 +
    u^{10}).
\end{align*}

Note that these canonical $\u$-polynomials can be changed for the reflection of the diagram, while Riley polynomials are always preserved under the reflection. Such a phenomenon can also be in the $8_{18}$ case in Section \ref{sec:8_18}.

\subsubsection{Complex volumes and Cusp shapes}
The complex volumes  and  cusp shapes are  as in Figure \ref{table:9_29vol}.
\begin{figure}%[H]
	$$\begin{array}{|c|c|c|r|r|}
		\hline
		\multicolumn{2}{|c|}{\Xirr} &
		\lambda& \ii (\vol + \ii \cs)\rule{1.5em}{0em} & \text{Cusp shape}\rule{2em}{0em}\\ 
		\thickhline

		\Xpar^1&  \begin{aligned}
			u ~&=&\hspace{-2ex} -1  \rule{0em}{1.1em} \\
		\end{aligned} & +&
		\begin{aligned}
			0 \hspace{11.5ex} \rule{0em}{1.1em}
		\end{aligned}&
		\begin{aligned}
			-9/4  \hspace{11.5ex}\rule{0em}{1.1em}
		\end{aligned}\\
		\hline
		
		\begin{aligned}\Xpar^2 \\ {\scriptscriptstyle (geom)}
		\end{aligned}	
		 &	\begin{aligned}
			u ~&=&\hspace{-2ex} -0.504 &-1.408  \rule{0em}{1.1em}
			 \ii\\
			u ~&=&\hspace{-2ex} -0.504 &+1.408 \ii\\
			u ~&=&\hspace{-2ex} -0.297 &-1.222 \ii\\
			u ~&=&\hspace{-2ex} -0.297 &+1.222 \ii\\
			u ~&=&\hspace{-2ex} -0.231 &  \\
			u ~&=&\hspace{-2ex} 0.090 &-1.266 \ii\\
			u ~&=&\hspace{-2ex} 0.090 &+1.266 \ii\\
			u ~&=&\hspace{-2ex} 0.643 &-0.378 \ii\\
			u ~&=&\hspace{-2ex} 0.643 &+0.378 \ii\\
			u ~&=&\hspace{-2ex} 1.367
			&  
		\end{aligned} &-&
		\begin{aligned}
			10.4508  & -12.2059 \ii \rule{0em}{1.1em}\\
			10.4508  & +12.2059 \ii\\
			4.7313  & -5.9624 \ii\\
			4.7313  & +5.9624 \ii\\
			1.2631 & \\
			8.9245  & +2.3689 \ii\\
			8.9245  & -2.3689 \ii\\
			-1.1801  & +1.0383 \ii\\
			-1.1801  & -1.0383 \ii\\
			0.5871 & 
		\end{aligned} &
		\begin{aligned}
			7.0577 &+6.5891 \ii \rule{0em}{1.1em} \\
			7.0577 &-6.5891 \ii\\
			6.5576 &+6.4524 \ii\\
			6.5576 &-6.4524 \ii\\
			9.0988 &\\
			11.5357 &-2.9643 \ii\\
			11.5357 &+2.9643 \ii\\
			-4.7369 &-3.7117 \ii\\
			-4.7369 &+3.7117 \ii\\
			12.3229 &
		\end{aligned}\\
		\hline
		
		\begin{aligned}
			\Xpar^3 
		\end{aligned}
		& \begin{aligned}
		u ~&=&\hspace{-2ex} -1.142& -0.105 \ii \rule{0em}{1.1em}\\
		u ~&=&\hspace{-2ex}-1.142 & +0.105 \ii\\
		u ~&=&\hspace{-2ex}-0.603 & -1.446 \ii\\
		u ~&=&\hspace{-2ex}-0.603 & +1.446 \ii\\
		u ~&=&\hspace{-2ex}-0.597 & -0.027 \ii\\
		u ~&=&\hspace{-2ex}-0.597 & +0.027 \ii\\
		u ~&=&\hspace{-2ex}-0.182 & -1.049 \ii\\
		u ~&=&\hspace{-2ex}-0.182 & +1.049 \ii\\
		u ~&=&\hspace{-2ex}-0.073 & -1.153 \ii\\
		u ~&=&\hspace{-2ex}-0.073 & +1.153 \ii\\
		u ~&=&\hspace{-2ex}0.281 & - 0.319 \ii\\
		u ~&=&\hspace{-2ex}0.281 & + 0.319 \ii\\
		u ~&=&\hspace{-2ex}0.309 & - 1.112 \ii\\
		u ~&=&\hspace{-2ex}0.309 & + 1.112 \ii\\
		u ~&=&\hspace{-2ex}0.507 & - 1.457 \ii\\
		u ~&=&\hspace{-2ex}0.507 & + 1.457 \ii

		\end{aligned} &-&
		\begin{aligned}
			5.6695 & -6.4435 \ii \rule{0em}{1.1em}\\
			5.6695 & +6.4435 \ii\\
			9.7926 & \\
			9.7926 & \\
			1.1305 & +2.5785 \ii\\
			1.1305 & -2.5785 \ii\\
			4.1349 & \\
			4.1349 & \\
			4.3305 & -1.1312 \ii\\
			4.3305 & +1.1312 \ii\\
			4.3305 & +1.1312 \ii\\
			4.3305 & -1.1312 \ii\\
			1.1305 & +2.5785 \ii\\
			1.1305 & -2.5785 \ii \\
			5.6695 & +6.4435 \ii\\
			5.6695 & -6.4435 \ii
		\end{aligned}
		&
		\begin{aligned}
			5.4284 &+5.2942 \ii  \rule{0em}{1.1em}\\
			5.4284 &-5.2942 \ii\\
			9.8640 &\\
			9.8640 &\\
			0.2771 &-3.5680 \ii\\
			0.2771 &+3.5680 \ii\\
			7.8945 &\\
			7.8945 &\\
			3.4152 &+0.5108 \ii\\
			3.4152 &-0.5108 \ii\\
			3.4152 &-0.5108 \ii\\
			3.4152 &+0.5108 \ii\\
			0.2771 &-3.5680 \ii\\
			0.2771 &+3.5680 \ii\\
			5.4284 &-5.2942 \ii\\
			5.4284 &+5.2942 \ii
		\end{aligned}\\
		\hline
		
		\hline

	\end{array}
	$$
	\caption{Complex volume and Cusp shape for each representation $\rho\in\Xpar(9_{29})$, where $\lambda$ is the obstruction class.}
	\label{table:9_29vol}
\end{figure}
We remark that there is a missing representation of positive obstruction in the CURVE project \cite{curveproject}.

\subsection{Computations for links: {$5_1^2$} Whitehead link}\label{sec:whiteheadlink}

Essentially,  the link cases are similar to the knot case, but we need some modification for the theory. 
The parabolic quandle map $\M$  without sign-type in Proposition \ref{prop:surjM} is  true  for any link diagram and 
the formula computing obstruction class in Theorem \ref{thm:obs} is also valid if one considers the longitude and the total-sign for each link component.
But,  $\Xpar \neq \bRnt$ for link cases  and the parabolic quandle system $\Qe$ with sign-type in Section \ref{sec:paraquandlesigntype} is not bijective to $\Xpar$ nor $\bRnt$. In particular,  the splitting property of Riley polynomials fails for link cases.
Moreover, even if one fixes a sign-type at each crossing, the sign of arc-colorings $\pm(a_1,\dots,a_{N_1})$ and $\pm(b_1,\dots,b_{N_2}) $  can be chosen arbitrarily where $a_i$'s and $b_j$'s are the arcs in different link components. So we can not define the sign of $u$-value $u_\ij=\lrbar{a_i,b_j}$ when the indices $i$ and $j$ are in the  different  components.  
We will see these things in the computation of Whitehead link in Section \ref{sec:whiteheadlink}.

Let us consider a link diagram of the Whitehead link  and label each arc and crossing by $\aa_i$ and $\cc_i$ respectively as in Figure \ref{fig:5_1^2fig}. 

\begin{figure}[!ht]
	\begin{center}
		\includegraphics{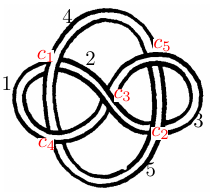}	\end{center}
	\caption{Labeling the diagram of $5_1^2$ knot.}\label{fig:5_1^2fig}
\end{figure}

\subsubsection{Quandle equations for $\Qu$}
Let us begin with  $a_1=\qvec{1}{0}$ and $a_4=\qvec{0}{u}$. Then each arc-coloring is computed successively along $ \cc_1, \cc_4, \cc_2$  as follows,
\begin{equation}\label{whiteheadlink}\aligned
	a_2&=a_1+\langle a_1, a_4 \rangle a_4=\begin{pmatrix} 1 \\ u^2 \end{pmatrix},\\ a_5&=a_4+\langle a_4, a_1 \rangle a_1=\begin{pmatrix} -u \\ u \end{pmatrix},\\ a_3&=a_2-\langle a_2, a_5 \rangle a_5=\begin{pmatrix} u^4+u^2+1 \\ -u^4 \end{pmatrix}.\\
	\endaligned
\end{equation}

Note that there are two relations at $c_3$ and $c_5$ that have not been used yet.

$$\aligned a_1 \pm  (a_3 - \langle a_3, a_2 \rangle a_2) = 0,\\
a_4 \pm  (a_5 - \langle a_5, a_3 \rangle a_3) = 0.
\endaligned $$
The Whitehead link has two link components and  there are 4 kinds of total signs $(-,-)$, $(-,+)$, $(+,-)$ and $(+,+)$ along the choice of sign for each link component, which are determined by  the signs of the trace of longitude for the associate representation.  

By straightforward computation,  the  non-abelian representations are only in  the $(-,-)$ case and the system of equation is equivalent to  
$g(u)=u^4+2u^2+2=0,$ which has four roots $\pm \frac{\sqrt{2\sqrt{2}-2}}{2}\pm \frac{\sqrt{2\sqrt{2}+2}}{2}\ii$.
Note that the correspondence from this normalized quandle solution in $\Qu$ to $\bRpar$ is not bijective but  $2$-to-$1$, and hence two roots of $u_0$ and $-u_0$ of $g(u)=0$ 
gives the same representation by the parabolic quandle map in Proposition \ref{prop:surjM}. 
In summary, we can conclude that there are two irreducible representations up to conjugate.

\subsubsection{Quandle equations for $\Qv$}
As putting  $a_1=\qvec{1}{0}$ and $a_4=\qvec{v}{0}$, and solving the equations for $\Qv$ in a similar way to $Q_u$, one can check that there is a 1-dimensional solution of abelian representation:
$$a_1=a_2=a_3=\qvec{1}{0},~~ a_4=a_5=\qvec{v}{0}.$$

We can see that $\overline{\Rpar^{ab}}$ is 1-dimensional, but $\Xpar^{ab}$ is 0-dimensional and  a singleton set of constant trace $2$ as in the case of knot.

\subsubsection{Abelian representations for $5_1^2$}
% We already know that every knot has two elements of $\bRab =\Set{\id, [\sm{1&1\\0&1}] }$. However link case is different to knot case and has more elements generally.
For link cases, the abelian representations are more complicated than knot cases.
For example, 
the $\overline{\Rpar^{ab}} (5_1^2)$ has 4 abelian (three 0-dimensional and one 1-dimensional) representation classes up to conjugation as follows,
$$
\begin{aligned}
\bRab =&\{(\bm\wg_1,\bm\wg_4) \mid  \bm\wg_1=\bm\wg_4=\id; ~\bm\wg_1=\id, \bm\wg_4=\sm{1&1\\0&1}; \\
&\bm\wg_1=\sm{1&1\\0&1}, \bm\wg_4=\id; ~\bm\wg_1=\sm{1&1\\0&1},\bm\wg_4=\sm{1&v^2\\0&1},\text{ where }v\ne 0 \},
\end{aligned}
$$
where the other $\bm\wg_2$, $\bm\wg_3$, $\bm\wg_5$ are expressed by regular functions of $\bm\wg_1$ and $\bm\wg_4$.
If one considers only the case of non-trivial meridian, i.e., $\rho(\bm\wg_i)\neq \sm{1&0\\0&1}$ for any $i$. There is only one component of dimension $1$.

% \subsubsection{Arc-colorings}
% There is only one irreducible component of $g(u)=u^4+2u^2+2$  and the arc-colorings of each component can be written in the number field $\QQ(u_i)$ with a root $u$ of $g$ as follows,
% $$
% 	a_1 = \qvec{ 1}{ 0},~~  
% 	a_2 = \qvec{ 1}{ u^2 },~~	
% 	a_3 = \qvec{ -u^2-1}{ 2u^2+2 },~~	
% 	a_4 = \qvec{ 0}{ u },~~	
% 	a_5 = \qvec{ -u}{ u }.~~
% $$

\subsubsection{Riley polynomials}

By the formula $\yv_\ij=\lrbar{\aa_i,\aa_j}^2$,
the $y$-values are 
$$
\begin{aligned}
	\yv(c_1)=\yv(c_4)=&~ u^2,\\
	\yv(c_2)=\yv(c_5)=&~ -u^2,\\
	\yv(c_3)=&~ -2u^2-2, \\
\end{aligned}
$$
and  we obtain Riley polynomials as follows
\begin{align*}
	R_{\cc_1}=R_{\cc_4}=&~ y^2+2y+2,\\
	R_{\cc_2}=R_{\cc_5}=&~ y^2-2y+2,\\
	R_{\cc_3} =&~ y^2+4.
\end{align*}

%Let us consider the $\pi$-rotation of the link diagram. 
As we mentioned 
in Remark \ref{rmk:linkhomo} that the Riley polynomials for link crossings depend on the choice of the orientation.
If we reverse the orientation of the figure-8 component and preserve the orientation of the figure-0 component in Figure \ref{fig:5_1^2fig}, then the Riley polynomials at $\cc_1, \cc_4$ and $\cc_2, \cc_5$ are transformed  each other as $y^2+2y+2$ $\longleftrightarrow$ $y^2-2y+2$, which is from  $\langle \ii a_i, a_j \rangle^2=-\langle a_i, a_j \rangle^2$ by (\ref{eqn:reverse}) in Section \ref{sec:symquan}. 
If we apply a $\pi$-rotation of the link diagram in Figure \ref{fig:5_1^2fig}, then the orientation of the figure-8 component is reversed  and we can check that Theorem \ref{thm:homomorphism} holds for such a $\pi$-rotation symmetry.

Note that the square-splitting property is only satisfied for the Riley polynomial $R_{\cc_3}$ which is the self-crossing in the diagram. 
The canonical $u$-polynomial is also only defined at the self-crossing $\cc_3$. 
The canonical sign-type is $\e_3=-1$ at the only self-crossing. The corresponding $u$-values $u_{\cc_3}=\lrbar{\aa_3,\aa_2}$ are $1\pm\ii$ and the canonical $u$-polynomial is as follows,
 $$g_{\cc_3}(u)=(u-1-\ii)(u-1+\ii)=u^2-2u+2.$$

 \subsubsection{Complex volumes and Cusp shapes}
 The complex volumes  and  cusp shapes are  as follows. 
 \begin{figure}[H]
	 $$\begin{array}{|c|c|c|c|}
		 \hline
		  g(u)=0 & \ii (\vol + \ii \cs)\rule{1.5em}{0em} & \text{Cusp shape of }L_1 & \hspace{1em}\text{Cusp shape of }L_2 \rule{2em}{0em}\\ 
		 \thickhline
		 \begin{aligned}
			 u^2 ~&=&\hspace{-2ex} -1 &- \ii \rule{0em}{1.1em} \\
			 u^2 ~&=&\hspace{-2ex} -1  &+  \ii
		 \end{aligned} &
		 \begin{aligned}
			2.4674 &+ 3.6638 \ii  \rule{0em}{1.1em}\\
			2.4674 &- 3.6638  \ii
		 \end{aligned} &
		 \begin{aligned}
			2 &- 2 \ii  \rule{0em}{1.1em}\\
			2 &+ 2 \ii
		 \end{aligned}&
		 \begin{aligned}
			2 &- 2 \ii  \rule{0em}{1.1em}\\
			2 &+ 2 \ii
		\end{aligned}\\
		 \hline
	 \end{array}
	 $$
	 %\caption{Complex volumes and Cusp shapes for all $\rho\in\Xpar(7_4)$.}
	 \label{table:5_1^2}
 \end{figure}
 Remark that, different from knot cases, two conjugacy classes of representations in $\Xpar$ are indicated by the square of the root of $g(u)$. There are two longitudes of $L_1$ and $L_2$ where $L_1$ is the figure-8 component and $L_2$ is the figure-0 component in Figure \ref{fig:5_1^2fig}. We can check that the Chern-Simons invariant is $\pi^2/4$ numerically, but, usually, it is not easy to prove rigorously what the exact value is. 

\section{Computation of $\Xpar(8_{18})$ and diagrammatic Symmetry}\label{sec:8_18}
The $8_{18}$ is a very symmetrical knot. The structure of $\Xpar(8_{18})$ is exceptionally complicated among the knots with a small number of crossings. In this section, we compute $\Xpar$ and study the behavior for the diagrammatic symmetries. 

\subsection{Computation}
Let us consider a knot diagram of $8_{18}$ and label each arc and crossing by $a_i$ and $\cc_i$ as in Figure \ref{fig:8_18fig}.

\begin{figure}[H]
	\begin{center}
		\includegraphics{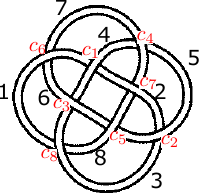}
	\end{center}
	\caption{A diagram of $8_{18}$ knot in Rolfsen table.}\label{fig:8_18fig}
\end{figure}
\subsubsection{Quandle equations for $\Qu$}\label{sec:8_18Qu}
Let us begin with  $a_1=\qvec{1}{0}$, $a_4=\qvec{0}{u}$ and $a_6=\qvec{a}{b}$ with indeterminate variables $u$, $a$, $b$.
Then, each arc-coloring is computed successively along $\cc_1, \cc_6, \cc_7, \cc_4, \cc_2$  as follows,
\begin{equation*}\label{eq:8_18coloring}\aligned
	a_2&\overset{\cc_1}{=}a_1-\lrbar{ a_1, a_4 } a_4=\begin{pmatrix} 1 \\ -u^2 \end{pmatrix},\\ 
	a_7&\overset{\cc_6}{=}a_6+\lrbar {a_6, a_1} a_1=\begin{pmatrix} a-b \\ b \end{pmatrix},\\ 
	a_8&\overset{\cc_7}{=}a_7-\lrbar {a_7, a_2} a_2=\begin{pmatrix} a + a u^2 - b u^2 \\ b - b u^2 - a u^4 + b u^4 \end{pmatrix},\\
	a_5&\overset{\cc_4}{=}a_4+\lrbar {a_4, a_7} a_7=\begin{pmatrix} -a^2 u + 2 a b u - b^2 u\\ u - a b u + b^2 u\end{pmatrix},\\
	a_3&\overset{\cc_2}{=}a_2+\lrbar {a_2, a_5} a_5=\begin{pmatrix*}
		\begin{aligned}
			1 - a^2 u^2 + 2 a b u^2 + a^3 b u^2 - b^2 u^2 - 3 a^2 b^2 u^2 + 
			3 a b^3 u^2 \\- b^4 u^2 + a^4 u^4 - 4 a^3 b u^4 + 6 a^2 b^2 u^4 - 
			4 a b^3 u^4 + b^4 u^4
		\end{aligned} \\[2.5ex]
		\begin{aligned}
			-2 a b u^2 + 2 b^2 u^2 + a^2 b^2 u^2 - 
			2 a b^3 u^2 + b^4 u^2 - a^2 u^4\\ + 2 a b u^4 + a^3 b u^4 - b^2 u^4 - 
			3 a^2 b^2 u^4 + 3 a b^3 u^4 - b^4 u^4
		\end{aligned}
	\end{pmatrix*}.\\
	\endaligned
\end{equation*}

When we obtain all colorings, there still remains three unused relations at the crossings $\cc_3, \cc_5$, and $\cc_8$. 
We put two of them as follows,
\begin{equation}
	\begin{aligned}
		a_4&-(a_3- \lrbar{ a_3, a_6} a_6) = \begin{pmatrix*}
			h_1(u,a,b)\\
			h_2(u,a,b)
		\end{pmatrix*} \text{  at $\cc_3$},\\
		a_6&-(a_5- \lrbar{ a_5, a_8} a_8)=\begin{pmatrix*}
			h_3(u,a,b)\\
			h_4(u,a,b)
		\end{pmatrix*} \text{  at $\cc_5$},
	\end{aligned}	
\end{equation}

and the last equation at $\cc_8$ should be satisfied up to sign, which produces the obstruction class.
\begin{align}
	\pm a_1 =a_8+ \langle a_8, a_3 \rangle a_3, \text{ at $c_8$}.
\end{align}

\subsubsection{Primary decomposition}
We obtained four equations $h_1$, $h_2$, $h_3$, and $h_4$ $\in \Qbb[u,a,b]$ for $\Qu$ and apply primary decomposition algorithms of mathematical softwares\footnote{We  tested Magma, Maple, Singular and Macaulay2 to do the primary decomposition and compared them. All the results are the same. }.
We also have the system of equation for $\Qv$
by the similar procedure with the initials $a_1=\qvec{1}{0}$, $a_4=\qvec{v}{0}$ and $a_6=\qvec{a}{b}$,  and compute the primary decomposition as well.

The irreducible component $\Xpar^2=\left<-1 + v, -1 + b \right>$ (see Figure \ref{table:8_18Xpar}) looks  1-dimensional with coloring of 
$$\aligned 
a_1=a_2=a_4=\begin{pmatrix} 1\\ 0 \end{pmatrix},
a_3=\begin{pmatrix} 1+ a\\  1 \end{pmatrix},
a_5=a_6=a_8=\begin{pmatrix} a\\ 1 \end{pmatrix},
a_7=\begin{pmatrix} a - 1\\ 1 \end{pmatrix}. \endaligned$$ 
Then by multiplication $\begin{pmatrix} 1&-a\\ 0&1 \end{pmatrix} a_i$, we get the following modified solution up to conjugate to the above one: 
$$\aligned 
a_1=a_2=a_4=\begin{pmatrix} 1\\ 0 \end{pmatrix},
a_3=\begin{pmatrix} 1\\  1 \end{pmatrix},
a_5=a_6=a_8=\begin{pmatrix} 0\\ 1 \end{pmatrix},
a_7=\begin{pmatrix}  - 1\\ 1 \end{pmatrix}. \endaligned$$ 
Hence $\Xpar^2$ becomes 0-dimensional representation.

The results are summarized in Figure \ref{table:8_18Xpar}. 
$\Qu$ has 9 irreducible factors where one has  positive obstruction and the other remaining 8 cases have negative obstructions.  $\Qv$ has 2 irreducible factors where  the one is for abelian representations of $\Xpar$ with positive obstruction and the other one $\Xpar^2$ has negative obstruction. 
We index each irreducible component of $\Xpar$ as in Figure \ref{table:8_18Xpar}. 
Remark that  the generating set of $\Xpar^{i}$ can be  differently  expressed. For example,  $\Xpar^7$ is determined by the generating set of  $$\left<b^4 - b^3 + 2b + 1, b^3 - 2b^2 + b + u + 2, 2b^3 - 3b^2 + a + b + 3\right>.$$

\begin{figure}[H]
	\[ 
	\def\arraystretch{1.3}
	\begin{array}{|c|c|c|c|c|}
		\hline
		\multicolumn{2}{|c|}{\text{Decomposition}}  & \text{$\lambda$}& 
		\text{Generating set} &\iffalse \Qu /\Qv \fi\\
		\thickhline
		\multicolumn{2}{|c|}{\Xab}  & + & \left< v-1,a-1,b \right>&\Qv\\
		\hline
		\multirow{10}{*}{$\Xpar$} & \Xpar^1&+ &\left<{1 + u + u^2, 1 + b + u, 1 + a} \right>&\Qu\\
		\cline{2-5}
		& \Xpar^2&- &\left<-1 + v, -1 + b \right>&\Qv\\
		\cline{2-5}
		& \Xpar^3&- &\left<-1 + u, -1 + b, a \right>&\Qu\\
		\cline{2-5}
		& \Xpar^4&- &\left<1 + u, b, 1 + a \right>&\Qu\\
		\cline{2-5}
		& \Xpar^5&- &\left< {1 + u, 1 + b, 1 + a} \right>&\Qu\\
		\cline{2-5}
		& \Xpar^6&- &
		\rule{0em}{4ex}
		\begin{aligned}
			&\left<-1 + 2 u + u^2 - 2 u^3 + u^4,\right.  \\[-1ex] &~~\left. -1 + b + u^2 - u^3, 1 + a \right>
		\end{aligned} &\Qu\\
		\cline{2-5}
		& \Xpar^7&- &\left<{1 + u + u^2, -1 - b + b^2 - u - b u, a - u + b u} \right>&\Qu\\
		\cline{2-5}
		& \Xpar^8&- &\rule{0em}{4ex}\begin{aligned}
			&\left< 3 - 6 u + 6 u^2 - 3 u^3 + u^4,\right.  \\[-1ex] &~~\left. -4 + b + 5 u - 3 u^2 + u^3, -3 + 
			3 a - u^3 \right>
		\end{aligned}
		&\Qu\\
		\cline{2-5}
		& \Xpar^9&- & \rule{0em}{4ex}
		\begin{aligned}
			&\left<1 - 2 u + u^3 + u^4,\right.  \\[-1ex] &~~\left. -2 + b + u + 2 u^2 + u^3,  a - 2 u - 2 u^2 - u^3 \right>
		\end{aligned}&\Qu\\
		\cline{2-5}	
		& \Xpar^{10}&- &\rule{0em}{4ex}\begin{aligned}
			&\left<1 - 2 u + u^3 + u^4,\right.  \\[-1ex] &~~\left. -2 + b + 2 u + 3 u^2 + 2 u^3, 1 + a \right>
		\end{aligned}		
		&\Qu\\
		\hline
	\end{array}
	\]
	\caption{Primary decomposition of $\Xpar(8_{18})$ where $\lambda$ is the obstruction class.}
	\label{table:8_18Xpar}
\end{figure}

\subsubsection{Riley polynomials}
In a similar way to the previous examples, we calculate the $y$-value at each crossing. We, however, do not explicitly write down them here  since the expression itself is not essential and depends on the choice of initial arcs with indeterminate variables $u$, $a$, $b$, and $v$.  
As we compute Riley polynomials from the $y$-values we observe that
all Riley polynomials are the same, i.e.,
\[
\begin{aligned}
	R_{\cc_i}(y)= 
	y(y-1)^3(y^2+y+1)^3(y^4-2y^3+7y^2-6y+1)\\
	(y^4+3y^3+6y^2+9)(y^4-y^3+6y^2-4y+1)^2,
\end{aligned}
\]
for all $1\leq i\leq 8$.

Moreover, we compute the corresponding irreducible factors of  Riley polynomials for  each $\Xpar^i$, $1\leq i\leq 10$ in Figure \ref{table:8_18Xpar}.

\begin{figure}[H]
	\[ 
	\def\arraystretch{1.3}
	\begin{array}{|c|c|c|c|c|c|c|}
		\hline
		& \Xpar^1 &\Xpar^2&\Xpar^3&\Xpar^4&\Xpar^5&\Xpar^6\\
		\thickhline
		\cc_1 &1 + y + y^2&  y& -1+ y& -1+ y& -1+ y &1 - 6 y + 7 y^2 - 2 y^3 + y^4 \\\hline
		\cc_2 &1 + y + y^2& -1+ y& -1+ y&  y& -1+ y &1 - 6 y + 7 y^2 - 2 y^3 + y^4 \\\hline
		\cc_3 &1 + y + y^2& -1+ y&  y& -1+ y& -1+ y &1 - 6 y + 7 y^2 - 2 y^3 + y^4\\\hline
		\cc_4 &1 + y + y^2& -1+ y& -1+ y& -1+ y&  y &1 - 6 y + 7 y^2 - 2 y^3 + y^4\\\hline
		\cc_5 &1 + y + y^2&  y& -1+ y& -1+ y& -1+ y &1 - 6 y + 7 y^2 - 2 y^3 + y^4\\\hline
		\cc_6& 1 + y + y^2& -1+ y& -1+ y&  y& -1+ y &1 - 6 y + 7 y^2 - 2 y^3 + y^4\\\hline
		\cc_7 &1 + y + y^2& -1+ y&  y& -1+ y& -1+ y &1 - 6 y + 7 y^2 - 2 y^3 + y^4\\\hline
		\cc_8 &1 + y + y^2& -1+ y& -1+ y& -1+ y&  y &1 - 6 y + 7 y^2 - 2 y^3 + y^4\\\hline
	\end{array}
	\]
	\[ 
	\def\arraystretch{1.3}
	\begin{array}{|c|c|c|c|c|}
		\hline
		& \Xpar^7 &\Xpar^8&\Xpar^9&\Xpar^{10}\\
		\thickhline
		\cc_1 &\scriptstyle (1 + y + y^2)^2&\scriptstyle 9 + 6 y^2 + 3 y^3 + y^4&\scriptstyle 1 - 4 y + 
		6 y^2 - y^3 + y^4&\scriptstyle 1 - 4 y + 6 y^2 - y^3 + y^4 \\\hline
		\cc_2 &\scriptstyle1 - 4 y + 6 y^2 - y^3 + y^4&\scriptstyle 1 - 4 y + 6 y^2 - y^3 + y^4&\scriptstyle (1 + y + y^2)^2 &\scriptstyle 9 + 6 y^2 + 
		3 y^3 + y^4 \\\hline
		\cc_3 &\scriptstyle9 + 6 y^2 + 3 y^3 + y^4&\scriptstyle (1 + y + y^2)^2&\scriptstyle 1 - 4 y + 
		6 y^2 - y^3 + y^4&\scriptstyle 1 - 4 y + 6 y^2 - y^3 + y^4\\\hline
		\cc_4 &\scriptstyle1 - 4 y + 6 y^2 - y^3 + y^4&\scriptstyle 1 - 4 y + 
		6 y^2 - y^3 + y^4&\scriptstyle 9 + 6 y^2 + 3 y^3 + y^4 &\scriptstyle (1 + y + y^2)^2\\\hline
		\cc_5 &\scriptstyle(1 + y + y^2)^2& \scriptstyle 9 + 6 y^2 + 3 y^3 + y^4&\scriptstyle 1 - 4 y + 
		6 y^2 - y^3 + y^4&\scriptstyle 1 - 4 y + 6 y^2 - y^3 + y^4\\\hline
		\cc_6&\scriptstyle 1 - 4 y + 6 y^2 - y^3 + y^4&\scriptstyle 1 - 4 y + 6 y^2 - y^3 + y^4&\scriptstyle (1 + y + y^2)^2&\scriptstyle 9 + 6 y^2 + 
		3 y^3 + y^4 \\\hline
		\cc_7 &\scriptstyle9 + 6 y^2 + 3 y^3 + y^4&\scriptstyle (1 + y + y^2)^2&\scriptstyle 1 - 4 y + 
		6 y^2 - y^3 + y^4&\scriptstyle 1 - 4 y + 6 y^2 - y^3 + y^4\\\hline
		\cc_8 &\scriptstyle1 - 4 y + 6 y^2 - y^3 + y^4&\scriptstyle 1 - 4 y + 
		6 y^2 - y^3 + y^4&\scriptstyle 9 + 6 y^2 + 3 y^3 + y^4&\scriptstyle (1 + y + y^2)^2\\\hline
	\end{array}
	\]
	\caption{Factorization of Riley polynomials along the decomposition of $\Xpar(8_{18})$.}
	\label{table:8_18Riley}
\end{figure}

\subsubsection{$u$-polynomials}
As previously computed examples, we first obtain the $\uv$-values for the sign-type  $(1, 1, 1, 1, 1, 1, 1, -1)$  used to obtain arc-colorings in Section \ref{sec:8_18Qu}. Then the canonical $u$-polynomials are obtained by transforming into the canonical sign-type at each crossing.
Although Riley polynomials are the same at all crossings, there are two kinds of canonical $u$-polynomials as follows,
\[
\begin{aligned}
	g_{\cc_i}(u)=&(-1 + u)^2 u (1 + u) (1 - u + u^2)^2 (1 + u + 
	u^2) (1 + 2 u - u^3 + u^4)^2\\&~~~ (-1 - 2 u + u^2 + 2 u^3 + 
	u^4) (3 + 6 u + 6 u^2 + 3 u^3 + u^4)
\end{aligned}
\] 
for $i=1,3,5,7$ and 
\[
\begin{aligned}
	g_{\cc_i}(u)=& 
	(-1 + u) u (1 + u)^2 (1 - u + u^2) (1 + u + 
	u^2)^2(1 - 2 u + u^3 + u^4)^2\\&~~~ (-1 + 2 u + u^2 - 2 u^3 +
	u^4)  (3 - 6 u + 6 u^2 - 3 u^3 + u^4)
\end{aligned}
\]
for $i=2,4,6,8$.

Each irreducible factor of $u$-polynomials for  each $\Xpar^i$, $1\leq i\leq 10$ in Figure \ref{table:8_18upoly}.

\begin{figure}[H]
	\[ 
	\def\arraystretch{1.3}
	\begin{array}{|c|c|c|c|c|c|c|}
		\hline
		& \Xpar^1 &\Xpar^2&\Xpar^3&\Xpar^4&\Xpar^5&\Xpar^6\\
		\thickhline
		\cc_1 &1 + u + u^2& u& 1 + u& -1 + u& -1 + u& -1 - 2 u + u^2 + 2 u^3 + u^4 \\\hline
		\cc_2 &1 - u + u^2& 1 + u& 1 + u& u& -1 + u& -1 + 2 u + u^2 - 2 u^3 + u^4 \\\hline
		\cc_3 &1 + u + u^2& 1 + u& u& -1 + u& -1 + u& -1 - 2 u + u^2 + 2 u^3 + u^4\\\hline
		\cc_4 &1 - u + u^2& 1 + u& 1 + u& -1 + u& u& -1 + 2 u + u^2 - 2 u^3 + u^4\\\hline
		\cc_5 &1 + u + u^2& u& 1 + u& -1 + u& -1 + u& -1 - 2 u + u^2 + 2 u^3 + u^4\\\hline
		\cc_6& 1 - u + u^2& 1 + u& 1 + u& u& -1 + u& -1 + 2 u + u^2 - 2 u^3 + u^4\\\hline
		\cc_7 &1 + u + u^2& 1 + u& u& -1 + u& -1 + u& -1 - 2 u + u^2 + 2 u^3 + u^4\\\hline
		\cc_8 &1 - u + u^2& 1 + u& 1 + u& -1 + u& u& -1 + 2 u + u^2 - 2 u^3 + u^4\\\hline
	\end{array}
	\]
	\[ 
	\def\arraystretch{1.3}
	\begin{array}{|c|c|c|c|c|}
		\hline
		& \Xpar^7 &\Xpar^8&\Xpar^9&\Xpar^{10}\\
		\thickhline
		\cc_1 &~\scriptstyle(1 - u + u^2)^2&~\scriptstyle 3 + 6 u + 6 u^2 + 3 u^3 + u^4&~\scriptstyle 1 + 
		2 u - u^3 + u^4&~\scriptstyle 1 + 2 u - u^3 + u^4 \\\hline
		\cc_2 &\scriptstyle1 - 2 u + u^3 + u^4&\scriptstyle 1 - 2 u + u^3 + u^4&\scriptstyle(1 + u + u^2)^2&\scriptstyle 3 - 6 u + 6 u^2 - 
		3 u^3 + u^4 \\\hline
		\cc_3 &\scriptstyle3 + 6 u + 6 u^2 + 3 u^3 + u^4&\scriptstyle (1 - u + u^2)^2&\scriptstyle 1 + 
		2 u - u^3 + u^4&\scriptstyle 1 + 2 u - u^3 + u^4 \\\hline
		\cc_4 &\scriptstyle1 - 2 u + u^3 + u^4&\scriptstyle 1 - 2 u + u^3 + u^4&\scriptstyle 3 - 6 u + 
		6 u^2 - 3 u^3 + u^4 &\scriptstyle (1 + u + u^2)^2 \\\hline
		\cc_5 &\scriptstyle(1 - u + u^2)^2&\scriptstyle 3 + 6 u + 6 u^2 + 3 u^3 + u^4&\scriptstyle 1 + 
		2 u - u^3 + u^4&\scriptstyle 1 + 2 u - u^3 + u^4 \\\hline
		\cc_6&\scriptstyle1 - 2 u + u^3 + u^4&\scriptstyle 1 - 2 u + u^3 + u^4&\scriptstyle (1 + u + u^2)^2&\scriptstyle 3 - 6 u + 6 u^2 - 
		3 u^3 + u^4 \\\hline
		\cc_7 &\scriptstyle3 + 6 u + 6 u^2 + 3 u^3 + u^4&\scriptstyle (1 - u + u^2)^2&\scriptstyle 1 + 
		2 u - u^3 + u^4&\scriptstyle 1 + 2 u - u^3 + u^4 \\\hline
		\cc_8 &\scriptstyle1 - 2 u + u^3 + u^4&\scriptstyle 1 - 2 u + u^3 + u^4&\scriptstyle 3 - 6 u + 
		6 u^2 - 3 u^3 + u^4&\scriptstyle (1 + u + u^2)^2 \\\hline
	\end{array}
	\]
	\caption{Factorization of $u$-polynomials along the decomposition of $\Xpar(8_{18})$.}
	\label{table:8_18upoly}
\end{figure}

\subsection{$\Xpar(8_{18})$ and complex volume}\label{sec:8_18str}
As far as the authors know, the complete list of parabolic representations of $8_{18}$ was not known until this paper. So we would write it down as a proposition.
\begin{prop}
	For $8_{18}$ knot, there are 26  parabolic representations. In particular, 
	there are 10 irreducible components as $\Qbb$-variety  and each component has 1, 2 or 4  parabolic representations. The complex volumes and the cusp shapes are listed in Figure \ref{table:8_18vol}.
\end{prop}

\begin{figure}%[H]
	$$\begin{array}{|c|c|c|r|r|}
		\hline
		\multicolumn{2}{|c|}{\Xirr}&\lambda & \ii (\vol + \ii \cs)\rule{1.5em}{0em} & \text{Cusp shape}\rule{2em}{0em}\\ 
		\thickhline

		\Xpar^1& \begin{aligned}
			u ~&=&\hspace{-2ex} -0.5\phantom{00} &- 0.866\ii  \rule{0em}{1.1em} \\
			u ~&=&\hspace{-2ex} -0.5\phantom{00} &+ 0.866 \ii \\
		\end{aligned} & +&
		\begin{aligned}
			& +  4.05977\ii \rule{0em}{1.1em}\\
			&-4.05977 \ii\\
		\end{aligned}&
		\begin{aligned}
			& -6.92820 \ii \rule{0em}{1.1em}\\
			&+6.92820 \ii\\
		\end{aligned}\\
		\hline

		\Xpar^2& \begin{aligned}
			v ~&=&\hspace{-2ex} 1   \rule{0em}{1.1em} \\
		\end{aligned} &-&
		\begin{aligned}
			-1.64493 \hspace{11.5ex} \rule{0em}{1.1em}
		\end{aligned}&
		\begin{aligned}
			-6  \hspace{11.5ex}\rule{0em}{1.1em}
		\end{aligned}\\
		\hline
		
		\Xpar^3&
		\begin{aligned}
			u ~&=&\hspace{-2ex} 1  \rule{0em}{1.1em} \\
		\end{aligned}
		&-&
		\begin{aligned}
			-1.64493 \hspace{11.5ex} \rule{0em}{1.1em}
		\end{aligned}&
		\begin{aligned}
			-6  \hspace{11.5ex}\rule{0em}{1.1em}
		\end{aligned}\\
		\hline
		
		\Xpar^4&  \begin{aligned}
			u ~&=&\hspace{-2ex} -1  \rule{0em}{1.1em} \\
		\end{aligned} &-&
		\begin{aligned}
			1.64493 \hspace{11.5ex} \rule{0em}{1.1em}
		\end{aligned}&
		\begin{aligned}
			6  \hspace{11.5ex}\rule{0em}{1.1em}
		\end{aligned}\\
		\hline
		
		\Xpar^5&  \begin{aligned}
			u ~&=&\hspace{-2ex} -1  \rule{0em}{1.1em} \\
		\end{aligned} &-&
		\begin{aligned}
			1.64493 \hspace{11.5ex} \rule{0em}{1.1em}
		\end{aligned}&
		\begin{aligned}
			6  \hspace{11.5ex}\rule{0em}{1.1em}
		\end{aligned}\\
		\hline
		
		\begin{aligned}\Xpar^6 \\ {\scriptscriptstyle (geom)}
		\end{aligned}	
		& 	\begin{aligned}
			u ~&=&\hspace{-2ex} -0.883 &  \rule{0em}{1.3em} \\
			u ~&=&\hspace{-2ex} 0.468 & \\
			u ~&=&\hspace{-2ex} 1.207&- 0.978 \ii \\
			u ~&=&\hspace{-2ex} 1.207&+ 0.978 \ii \\
		\end{aligned} &-&
		\begin{aligned}
			1.71901 & \rule{0em}{1.1em}\\
			-1.71901 &  \\
			& - 12.3509 \ii\\ 
			& +12.3509 \ii\\
		\end{aligned} &
		\begin{aligned}
			5.40877 & \rule{0em}{1.1em}\\
			-5.40877 &  \\
			& +7.82655 \ii\\ 
			& -7.82655 \ii\\
		\end{aligned}\\
		\hline

		\Xpar^7 & \begin{aligned}
			b ~&=&\hspace{-2ex} -0.621 &- 0.187\ii   \rule{0em}{1.1em} \\
			b ~&=&\hspace{-2ex} -0.621 &+ 0.187 \ii  
		 \\
			b ~&=&\hspace{-2ex} 1.121&- 1.053 \ii
		 \\
			b ~&=&\hspace{-2ex} 1.121&+ 1.053 \ii
		 \\
		\end{aligned} &-&
		\begin{aligned}
			-1.64493 &- 4.05977 \ii \rule{0em}{1.1em}\\
			-1.64493 &+4.05977 \ii \\
			-1.64493 &+4.05977 \ii\\ 
			-1.64493 &- 4.05977 \ii\\
		\end{aligned}
		&
		\begin{aligned}
			-6& + 6.92820 \ii \rule{0em}{1.1em}\\
			-6& - 6.92820 \ii \\
			-6& - 6.92820 \ii\\ 
			-6& + 6.92820 \ii \\
		\end{aligned}\\
		\hline
		
		\Xpar^8 &\begin{aligned}
			u ~&=&\hspace{-2ex}\phantom{-} 0.648&- 1.498 \ii \rule{0em}{1.3em} \\
			u ~&=&\hspace{-2ex} 	0.648&+ 1.498 \ii \\
			u ~&=&\hspace{-2ex}  	0.851&- 0.632 \ii\\
		u ~&=&\hspace{-2ex}  	0.851&+ 0.632 \ii\\
		\end{aligned} &-&
		\begin{aligned}
			-1.64493 &+ 4.05977 \ii \rule{0em}{1.3em}\\
			-1.64493 &- 4.05977 \ii\\
			-1.64493 &- 4.05977 \ii\\ 
			-1.64493 &+ 4.05977 \ii\\
		\end{aligned}&
		\begin{aligned}
			-6& - 6.92820 \ii \rule{0em}{1.3em}\\
			-6& + 6.92820 \ii \\
			- 6& + 6.92820 \ii\\ 
			- 6& - 6.92820 \ii\\
		\end{aligned}\\
		\hline
		
		\Xpar^9&\begin{aligned}
			u ~&=&\hspace{-2ex} -1.121 &- 1.053 \ii \rule{0em}{1.3em} \\
			u ~&=&\hspace{-2ex} -1.121 &+ 1.053 \ii \\
			u ~&=&\hspace{-2ex} 0.621 &- 0.187 \ii \\
			u ~&=&\hspace{-2ex} 0.621 &+ 0.187 \ii \\
		\end{aligned} &-&
		\begin{aligned}
			1.64493 &+ 4.05977 \ii \rule{0em}{1.3em}\\
			1.64493 &- 4.05977 \ii \\
			1.64493 &- 4.05977 \ii\\ 
			1.64493 &+ 4.05977 \ii\\
		\end{aligned}&
		\begin{aligned}
			6& - 6.92820 \ii \rule{0em}{1.3em}\\
			6& + 6.92820 \ii \\
			6& + 6.92820 \ii\\ 
			6& - 6.92820 \ii\\
		\end{aligned}\\
		\hline

		\Xpar^{10}&\begin{aligned}
			u ~&=&\hspace{-2ex} -1.121 &- 1.053 \ii \rule{0em}{1.3em} \\
			u ~&=&\hspace{-2ex} -1.121 &+ 1.053 \ii \\
			u ~&=&\hspace{-2ex} 0.621 &- 0.187 \ii \\
			u ~&=&\hspace{-2ex} 0.621 &+ 0.187 \ii \\
		\end{aligned} &-&
		\begin{aligned}
			1.64493 &+ 4.05977 \ii \rule{0em}{1.3em}\\
			1.64493 &- 4.05977 \ii \\
			1.64493 &- 4.05977 \ii\\ 
			1.64493 &+ 4.05977 \ii\\
		\end{aligned}&
		\begin{aligned}
			6& - 6.92820 \ii \rule{0em}{1.3em}\\
			6& + 6.92820 \ii \\
			6& + 6.92820 \ii\\ 
			6& - 6.92820 \ii\\
		\end{aligned}\\
		\hline

	\end{array}
	$$
	\caption{Complex volume and Cusp shape for each representation $\rho\in\Xpar(8_{18})$, where $\lambda$ is the obstruction class.}
	\label{table:8_18vol}
\end{figure}
To specify each representation in $\Xpar^i$ in Figure \ref{table:8_18vol},  the second column of the table indicates the numerical value of $u$. For $\Xpar^8$, the values of $b$ instead of $u$ are given because $u$ does not indicate a single representation in $\Xpar^8$.
\begin{rmk}
	The  $\Xpar^6$ is the geometric component  containing a discrete faithful representation $\rho_{geo}$ since the maximal hyperbolic volume $12.3509...$ appears in the component. We can verify that the Chern-Simons invariant of $\rho_{geo}$ is zero and it should be so, since $8_{18}$ is an amphichiral knot.
	Note that 20 non-geometric representations have non-trivial Chern-Simons invariants all of them being $\pm 1.64493...=\pm\pi^2/6$ which is the Chern-Simons invariant of $3_1$. We can observe that,  except in the geometric component $\Xpar^6$, all complex volumes of $\Xpar(8_{18})$ are of the form $\pm 2 \vol(4_1)\pm   \ii{\cs}(3_1)$. This phenomenon is also found in the non-geometric component of $\Xpar(7_4)$, see Remark \ref{rmk:7_4}.

\end{rmk}

Let us say that two representations $\rho$ and $\rho'$ are \emph{essentially equivalent} if $\rho=\rho'\circ \sigma_*$ for an orientation preserving homeomorphism $\sigma$ of $S^3\setminus K$ and its induced knot group automorphism $\sigma_*:\Gk \to \Gk$. Remark  that  the induced full-back map  $\sigma^*: \Xpar \to \Xpar$ preserves the complex volumes, i.e.  $\cv(\rho)=\cv(\sigma^*(\rho)))$, since it is determined by the group homological fundamental class.  
The 26 representations in $\Xpar$ are classified by the essential equivalence and they are completely determined by the complex volumes as follows.
\begin{thm}\label{thm:818essential}
	For $8_{18}$, there are exactly 12 classes of essentially equivalent representations in $\Xpar$. In particular, the complex volumes completely classify the parabolic representations up to essential  equivalence.
\end{thm}
\begin{proof}
	In $\Xpar^1$ and $\Xpar^6$, all complex volumes are distinct and separate from the other representations.  
	Let us consider an automorphism $\sigma_*:\Gk \to \Gk$ given by that the crossing labels changed by  $i\to i+2 \mod 8$, which comes from nothing but the $\frac{\pi}{2}$-counterclockwise rotation $\sigma$ of the diagram. 
	$\sigma^*(\Xpar^2)$ should coincide with some irreducible component of $\Xpar$. The automorphism  $\sigma$ sends the crossing label $\cc_1$ to $\cc_3$, hence  the Riley polynomial of $\Xpar^2$ at $\cc_1$ is $y$ and $\sigma^*(\Xpar^2)=\Xpar^3$, since only $\Xpar^3$ has the same Riley polynomial $y$ at $\cc_3$. Therefore two representations in $\Xpar^2$ and $\Xpar^3$ are essentially equivalent. 
	The same holds true for  $\Xpar^4$ and $\Xpar^5$.
	
	Similarly, we obtain $\sigma^*(\Xpar^7)=\Xpar^8$ and $\sigma^*(\Xpar^9)=\Xpar^{10}$ by checking the Riley polynomials. 
\end{proof}

% \subsubsection{The structure of $\Xpar$}\label{sec:8_18str}

Finally, we analyze the structure of $\Xpar$ concerning  diagrammatic symmetry.
Let us consider two automorphism $\sigma$ and $\tau$ of $\Gk$ from the diagrammatic symmetry. The $\sigma$ is given by that the crossing labels changed by  $i\to i+2 \mod 8$, i.e., $\frac{\pi}{2}$-counterclockwise rotation of the diagram. The $\tau$ is given by that the crossing labels exchanged by $2i\leftrightarrow 2i+1$, where outer crossings $\{2,4,6,8\}$ transformed into inner crossings $\{3,5,7,1\}$ by an  isotopy and reflection. 

Note that the  $\sigma$ and $\tau$ induce automorphisms $\sigma^*$ and $\tau^*$ of $\Xpar$ and also 
we can analyze these actions on $\Xpar$ as follows.

\begin{thm}
	\begin{enumerate}
		\item $\sigma^*\circ\sigma^*=\tau^*\circ\tau^*=id$
		\item $\Xpar^i = \sigma^*(\Xpar^i)=\tau^*(\Xpar^i)$ for $i=1,6$	
		\item $\sigma^*:(\Xpar^2, \Xpar^4,
		\Xpar^7, \Xpar^9) \leftrightarrow 
		(\Xpar^3, \Xpar^5,
		\Xpar^8, \Xpar^{10})$
		\item $\tau^*:(\Xpar^2, \Xpar^3,
		\Xpar^7, \Xpar^8) \leftrightarrow 
		(\Xpar^5, \Xpar^4,
		\Xpar^{10}, \Xpar^9)$.
		
	\end{enumerate}
	%In particular, there are essentially 6 parabolic representations up to knot group automorphisms. 
\end{thm}
%$\sigma^*\circ\sigma^*$ goes the index i+4
\begin{proof}
	All of these arguments are easily derived by   Theorem \ref{thm:homomorfactor}. 
\end{proof}

\bibliographystyle{amsalpha}
\bibliography{bibliog}
\end{document}